\documentclass{article}

\usepackage[latin2]{inputenc}
\usepackage{amssymb,amsfonts,amsmath,amsbsy,amscd}
\usepackage{dina42e}
\usepackage{graphicx}
\usepackage{esint}
\usepackage{color}
\usepackage{amsthm}
\allowdisplaybreaks
\numberwithin{equation}{section}
\usepackage{epsfig}
\usepackage{tikz}
\usepackage[english]{babel}

\usepackage{bookmark,hyperref}

\usepackage{stmaryrd}
\usepackage{extarrows}
\usepackage{csquotes}

\newtheorem{theorem}{Theorem}[section]
\newtheorem{proposition}[theorem]{Proposition}
\newtheorem{lemma}[theorem]{Lemma}
\newtheorem{corollary}[theorem]{Corollary}

\newtheorem{remark}[theorem]{Remark}

\newtheorem*{theorem*}{Theorem}

\def\XXint#1#2#3{{\setbox0=\hbox{$#1{#2#3}{\int}$ }
\vcenter{\hbox{$#2#3$ }}\kern-.6\wd0}}

\definecolor{Yellow}{rgb}{0.95,0.9,0.0} 
\definecolor{Red}{rgb}{0.8,0.1,0.1}
\definecolor{Green}{rgb}{0.1,0.65,0.2}
\definecolor{Blue}{rgb}{0.1,0.1,0.8}
\definecolor{Purple}{rgb}{0.7,0.1,0.7}
\definecolor{Grey}{rgb}{0.6,0.6,0.6}

\newcommand{\supp}{\operatorname{supp}}

\newcommand{\dist}{\operatorname{dist}}

\newcommand{\R}{\mathbb{R}}
\newcommand{\Rd}{{\mathbb{R}^d}}
\newcommand{\Hc}{\mathcal{H}}
\newcommand{\Mc}{\mathcal{M}}
\newcommand{\ve}{\mathbf{v}}

\newcommand{\we}{\mathbf{w}}

\newcommand{\BV}{\operatorname{BV}}

\newcommand{\dy}{{\mathrm{d}}y}

\newcommand{\Div}{\operatorname{div}}
\newcommand{\no}{\mathbf{n}} 
\newcommand{\C}{\mathbb{C}}

\newcommand{\eps}{\varepsilon}
\renewcommand{\vec}[1]{\operatorname{\mathbf{#1}}}

\begin{document}
\begin{titlepage}
\title{Approximation of Classical Two-Phase Flows of Viscous
Incompressible Fluids by a Navier-Stokes/Allen-Cahn System}
\author{Helmut Abels, Julian Fischer, and Maximilian Moser}
\date{}
\end{titlepage}

\maketitle
\begin{abstract}
We show convergence of the Navier-Stokes/Allen-Cahn system to a classical sharp interface model for the two-phase flow of two viscous incompressible fluids with same viscosities in a smooth bounded domain in two and three space dimensions as long as a smooth solution of the limit system exists. Moreover, we obtain error estimates with the aid of a relative entropy method. Our results hold provided that the mobility $m_\eps>0$ in the Allen-Cahn equation tends to zero in a subcritical way, i.e., $m_\eps= m_0 \eps^\beta$ for some $\beta\in (0,2)$ and $m_0>0$.
The proof proceeds by showing via a relative entropy argument that the solution to the Navier-Stokes/Allen-Cahn system remains close to the solution of a perturbed version of the two-phase flow problem, augmented by an extra mean curvature flow term $m_\eps H_{\Gamma_t}$ in the interface motion. In a second step, it is easy to see that the solution to the perturbed problem is close to the original two-phase flow.
\end{abstract}

{\small\noindent
}

\section{Introduction}

During its evolution, the interface between two immiscible fluids may undergo topological changes, such as the merging or pinchoff of droplets. Mathematically, when modeling the interface classically as a $(d-1)$-dimensional manifold, this creates challenges for the analysis and for numerical approximations. Diffuse-interface models circumvent these problems by replacing the sharp interface by a diffuse transition layer of a finite width $\eps>0$, reducing the problem to a set of PDEs posed on the entire domain. However, this procedure comes at the cost of introducing an additional modeling error: For many diffuse-interface models for fluid-fluid interfaces it has remained an open problem to rigorously show {convergence} to the original sharp-interface model in the limit of vanishing interface width $\eps\rightarrow 0$, even prior to any topology change. In the present work we prove convergence of a diffuse-interface approximation for one of the most fundamental macroscopic models for a fluid-fluid interface, the Navier-Stokes equation for two immiscible incompressible fluids separated by a sharp interface with surface tension. To the best of our knowledge, the present work is the first general\footnote{The only previous work in this direction \cite{AbelsFeiMoser} only explicitly covers the case of the scaling regime $m_\eps=\eps^{1/2}$ and works under substantially stronger assumptions on the initial data.} quantitative convergence result  for any diffuse-interface approximation of the standard free boundary problem for the interface between two immiscible incompressible viscous fluids.

More specifically, in this contribution we rigorously identify the sharp-interface limit of the following Navier-Stokes/Allen-Cahn system
\begin{subequations}
	\begin{alignat}{2}\label{eq_NSAC1}
  \partial_t \ve_\eps +\ve_\eps\cdot \nabla \ve_\eps-\Delta \ve_\eps  +\nabla p_\eps & = -\eps \Div (\nabla \varphi_\eps \otimes \nabla \varphi_\eps)&\quad & \text{in}\ \Omega\times(0,T_0),\\\label{eq_NSAC2}
  \Div \ve_\eps& = 0&\quad & \text{in}\ \Omega\times(0,T_0),\\\label{eq_NSAC3}
  \partial_t \varphi_\eps +\ve_\eps\cdot \nabla \varphi_\eps & =m_\eps\left(\Delta \varphi_\eps - \tfrac{1}{\eps^{2}} W'(\varphi_\eps)\right) &\quad & \text{in}\ \Omega\times(0,T_0),\\
  \label{eq_NSAC4}
  (\ve_\eps,\varphi_\eps)|_{\partial\Omega}&= (0,-1)&\quad & \text{on }\partial\Omega\times(0,T_0),\\
  \label{eq_NSAC5}
  (\ve_\eps,\varphi_\eps) |_{t=0}& = (\ve_{0,\eps},\varphi_{0,\eps})&\quad& \text{in }\Omega,
  \end{alignat}
\end{subequations}
in a bounded domain $\Omega\subseteq \R^d$, $d=2,3$, with smooth boundary. Here $\ve_\eps\colon \Omega\times (0,T_0)\to \R^d$ is the mean velocity of the fluid mixture, $p_\eps\colon \Omega\times (0,T_0)\to \R$ its pressure and $\varphi_\eps\colon \Omega\times (0,T_0)\to \R$ an order parameter (e.g.\ the volume fraction difference of the fluids) related to the different phases, where the values $\varphi_\eps=\pm 1$ describe that only one fluid is present. Moreover, $\eps>0$ is a constant related to the thickness of the diffuse interface and $m_\eps>0$ is a constant diffusion coefficient, which depends on $\eps>0$. Here $W\colon \R\to \R$ is a double well-potential, satisfying standard assumptions. More precisely, we assume that $W$ is twice continuously differentiable and we have for some $c>0$
\begin{equation*}
  W(\pm 1) = 0, \quad W(-s)=W(s), \quad W(s)\geq c\min \{|s-1|^2, |s+1|^2\}
\end{equation*}
for all $s\in\R$. A standard example is $W(s) = c(1-s^2)^2$ for $c>0$.

This model was introduced by Liu and Shen in \cite{LiuShenModelH} to describe a two-phase flow for incompressible fluids with the same viscosity and densities. For simplicity we have set the densities and viscosities to one. The model can be considered as the analogue of the well-know ``model H'', cf.\ \cite{GurtinTwoPhase,HohenbergHalperin}, if one replaces the convective Cahn-Hilliard equation by a convective Allen-Cahn equation. A first analytic study of the system \eqref{eq_NSAC1}-\eqref{eq_NSAC5} was done by Gal and Grasselli~\cite{GalGrasselliDCDS} in the case of a bounded smooth domain in two space dimensions, where the existence of global and exponential attractors and convergence to stationary solutions was shown. For a more general model with different densities and viscosities Jiang, Li, Liu~\cite{VariableDensityNSAC} proved the existence of weak solutions globally in time (in two and three space dimensions). Moreover, in the case of two space dimensions they proved global well-posedness in the strong sense and studied the longtime behavior of strong solutions. We refer to Giorgini, Grasselli and Wu~\cite{GiorginiGrasselliWuJFA} for a mass-conserving variant of the Navier-Stokes/Allen-Cahn system with different densities  and further references.

It is the purpose of this contribution to study the limit \eqref{eq_NSAC1}-\eqref{eq_NSAC5} as $\eps\to 0$ in the case that $m_\eps\to_{\eps\to 0} 0$ suitably. Then one expects to have convergence to solutions of the classical two-phase Navier-Stokes equation with surface tension: 
\begin{subequations}
\begin{alignat}{2}\label{eq_TwoPhase1}
	\partial_t \ve_0^\pm+\ve_0^\pm\cdot\nabla\ve_0^\pm-\Delta \ve_0^\pm  +\nabla p^\pm_0 &= 0 &\qquad &\text{in }\Omega^{\pm}_t, t\in [0,T_0],\\\label{eq_TwoPhase2}
	\Div \ve_0^\pm &= 0 &\qquad &\text{in }\Omega^{\pm}_t, t\in [0,T_0],\\\label{eq_TwoPhase3}
	\llbracket 2D\ve_0^\pm -p_0^\pm \textup{I}\rrbracket\no_{\Gamma_t} &= -\sigma H_{\Gamma_t}\no_{\Gamma_t} && \text{on }\Gamma_t, t\in [0,T_0],\\ \label{eq_TwoPhase4}
	\llbracket\ve_0^\pm \rrbracket &=0 && \text{on }\Gamma_t, t\in [0,T_0],\\
	\label{eq_TwoPhase5}
	V_{\Gamma_t}  &=\no_{\Gamma_t}\cdot \ve_0^\pm && \text{on }\Gamma_t, t\in [0,T_0],\\
	\label{eq_TwoPhase6}
	\ve_0^-|_{\partial\Omega}&= 0&&\text{on }\partial\Omega\times(0,T_0),\\
	\Gamma_0=\Gamma^0,\qquad \ve_0^\pm|_{t=0}&= \ve_{0,0}^\pm && \text{ in }\Omega^{\pm}_0,\label{eq_TwoPhase7}
\end{alignat}
\end{subequations}
      where $(\Gamma_t)_{t\in[0,T_0]}$ is an evolving $(d-1)$-dimensional submanifold of $\Omega$ such that $\Omega$ is the disjoint union of two smooth domains $\Omega^\pm_t$ and $\Gamma_t$ as well as $\partial\Omega_t^\pm =\Gamma_t$ for every $t\in[0,T_0]$. Here the absence of a boundary contact of $\Gamma_t$ is assumed for all $t\in[0,T_0]$. Moreover,  $\ve_0^\pm(.,t)\colon\Omega^\pm_t\to \R^d$ and $p_0^\pm(.,t)\colon\Omega^\pm_t\to \R$ are the velocity and pressure of two fluids filling $\Omega^+_t$ and $\Omega_t^-$ for every $t\in [0,T_0]$, $\no_{\Gamma_t}$ denotes the interior normal of $\Gamma_t$ with respect to $\Omega^+_t$, $H_{\Gamma_t}$ and $V_{\Gamma_t}$ denote the mean curvature (sum of principle curvature) and normal velocity, resp., of $\Gamma_t$ with respect to the orientation given by $\no_{\Gamma_t}$. Furthermore, $\llbracket f \rrbracket(s)= \lim_{h\to 0+}(f(s+h\no_{\Gamma_t}(s))-f(s+h\no_{\Gamma_t}(s)))$, $s\in\Gamma_t$, denotes the jump of a function $f$ defined in a neighborhood of $\Gamma_t$ and $D\ve_0^\pm:=\frac{1}{2}\left(\nabla\ve_0^\pm+(\nabla\ve_0^\pm)^\top\right)$ is the symmetric gradient. Finally, $\sigma= \int_{-1}^1\sqrt{2W(s)}\,ds>0$ is a surface tension coefficient. For the following we set
      \begin{equation}
        \label{eq:defnGamma}
      \Gamma:=\bigcup_{t\in[0,T_0]}\Gamma_t\times\{t\},\quad \Omega^\pm:=\bigcup_{t\in[0,T_0]}\Omega_t^\pm\times\{t\}, \quad
      \ve^0:=\sum_\pm \ve_0^\pm\chi_{\Omega^\pm}.        
      \end{equation}
      For given $\ve_\eps=\ve$ the convective Allen-Cahn equation, i.e., \eqref{eq_NSAC3}-\eqref{eq_NSAC4}, was discussed formally by the first author in \cite{SILConvectiveAC} and it was shown formally that the limit system is given by the transport equation \eqref{eq_TwoPhase5}-\eqref{eq_TwoPhase6}  in the case that $m_\eps=m_0 \eps$ for some $m_0>0$. We note that in this case these arguments can be extended to show formally  convergence of the full system \eqref{eq_NSAC1}-\eqref{eq_NSAC5} to \eqref{eq_TwoPhase1}-\eqref{eq_TwoPhase6} by combining it with the arguments in \cite{AGG} for a Navier-Stokes/Cahn-Hilliard system. But in the case that $m_\eps=m_0 \eps^\beta$ for some $\beta>2$ nonconvergence was shown in the sense that in general
      \begin{equation*}
        \varphi_\eps (x,t) = \theta_0\left(\frac{d_{\Gamma_t}(x)}\eps\right) + O(\eps)\qquad \text{as }\eps\to 0,
      \end{equation*}
      where $d_{\Gamma_t}$ is the signed distance function to $\Gamma_t$,
       no longer holds and
      a weak formulation of the right-hand side of \eqref{eq_NSAC1} does not converge to the mean curvature functional $\sigma H_{\Gamma_t}\no_{\Gamma_t}\delta_{\Gamma_t}$, which appears in a weak formulation of \eqref{eq_TwoPhase1} and \eqref{eq_TwoPhase3}. Here $\theta_0\colon \R \to \R$ is the so-called optimal profile, which is the unique solution of
      \begin{equation*}
        -\theta_0''(s)+ W'(\theta_0(s))= 0\quad \text{for all }s\in\R,\quad \theta_0(s)\to_{s\to\pm \infty} \pm 1,\quad \theta_0(0)=0. 
      \end{equation*}
      Therefore convergence of the full system \eqref{eq_NSAC1}-\eqref{eq_NSAC5} cannot be expected in this case. We note that in this case the counterexample given in \cite{AbelsLengeler} for a Navier-Stokes/Cahn-Hilliard system in a radially symmetric situation can be adapted to the present Navier-Stokes/Allen-Cahn system. It is the purpose of the present contribution to show convergence of solutions of \eqref{eq_NSAC1}-\eqref{eq_NSAC5} to the smooth solution of \eqref{eq_TwoPhase1}-\eqref{eq_TwoPhase6} on a time interval $[0,T_0]$ for which the latter exists in the case of a subcritical scaling of the mobility $m_\eps = m_0\eps^\beta$ for some $\beta\in (0,2)$. Moreover, we will derive error estimates with the aid of a relative entropy method.
      We note that this is the first rigorous convergence result for a vanishing mobility $m_\eps \to_{\eps\rightarrow0}0$ in this regime, which includes the natural choice $m_\eps=m_0\eps$.
Finally, let us remark that the derivation of a similar convergence result using a relative entropy method was attempted in the recent work \cite{JiangEtAlArXiv}; however, as the approach of \cite{JiangEtAlArXiv} relies on the invalid estimate \cite[equation (2.4)]{JiangEtAlArXiv}, it overlooks the need to devise a careful estimate for the critical interface stretching term that forms the main challenge for our result.

      Except to \cite{AbelsFeiMoser}, so far only convergence in the case of a non-vanishing mobility $m_\eps = m_0>0$ for all $\eps>0$ was shown. First this was done in the case of a Stokes/Allen-Cahn system with same viscosities by Abels and Liu~\cite{AbelsLiuNSAC}, then for a Navier-Stokes/Allen-Cahn system with different viscosities in Abels and Fei~\cite{AbelsFei}, both in two space dimensions, and by Hensel and Liu~\cite{HenselLiu} for the Navier-Stokes/Allen-Cahn system with same viscosities in two and three space dimensions. We note that the first two results are based on a refined spectral estimate for the linearized Allen-Cahn operator and rigorous asymptotic expansions, while the latter result uses the relative entropy method similarly as for the convergence of the Allen-Cahn equation to the mean curvature flow shown by Fischer, Laux, and Simon~\cite{FischerLauxSimon}. In the case of a non-vanishing mobility the limit system consists of a system, where \eqref{eq_TwoPhase5} is replaced by a convective mean curvature flow equation:
      \begin{equation*}
        V_{\Gamma_t} = \no_{\Gamma_t}\cdot \ve_0^\pm + m_0 H_{\Gamma_t}\qquad \text{on }\Gamma_t,t\in[0,T_0].
      \end{equation*}
      In the contribution by Abels, Fei, and Moser \cite{AbelsFeiMoser} convergence of a Navier-Stokes/Allen-Cahn system with different viscosities was shown in the special case of $m_\eps= m_0 \sqrt{\eps}$ and two space dimensions using the same method as in \cite{AbelsLiuNSAC, AbelsFei} refined for this degenerate case. We note that the arguments could be extended to $m_\eps=m_0\eps^{\beta}$ for $\beta\in (0,\tfrac12]$, but the case $\beta=\frac12$ seems to be critical for the estimates in this contribution and new ideas and refinements seem to be needed to treat the cases with $\beta>\frac12$ with this method. At this stage the relative entropy method used in the present contribution appears to be more flexible.

      In the following we consider a situation, in which the limit system \eqref{eq_TwoPhase1}-\eqref{eq_TwoPhase7} is known to possess a unique smooth solution for some $T_0>0$. We note that strong well-posedness of this system was extensively studied starting with the results by Denisova and Solonnikov~\cite{DenisovaTwoPhase}. Moreover, it was shown by Pr\"uss and Simonett~\cite{AnalyticityTwoPhaseFlow} that strong solutions become analytic instantaneously in time. We refer to Köhne, Prüss, and Wilke \cite{KoehnePruessWilkeTwoPhase} and the monograph by Pr\"uss and Simonett~\cite{PruessSimonettMovingInterfaces} for results on local well-posedness in an $L^p$-setting and further references. Finally, we note that global-in-time existence of a notion of weak solutions, called varifold solutions, was shown in \cite{GeneralTwoPhaseFlow} and weak-strong uniqueness for these kind of solutions was shown by Fischer and Hensel~\cite{FischerHensel}.

      In general, there are two main mathematical approaches to the quantitative justification of sharp-interface limits: An approach pioneered by de~Mottoni and Schatzman \cite{DeMottoniSchatzman} and Chen \cite{Chen} relies on a matched asymptotic expansion around the sharp-interface limit to obtain an approximate solution to the diffuse interface model; by means of a stability analysis of the linearized operator, it is possible to derive rates of convergence. This approach has recently also been successfully adapted to our Navier-Stokes/Allen-Cahn system with mobility $m_\eps=\eps^{1/2}$ by Fei and the first and the third author \cite{AbelsFeiMoser}. An alternative approach -- recently developed by the second author, Laux, and Simon \cite{FischerLauxSimon} -- proceeds via a suitably defined relative entropy.
In \cite{FischerLauxSimon}, the relative entropy approach is used to give a short proof of convergence of the Allen-Cahn equation towards mean curvature flow (valid for well-prepared initial data and as long as a classical solution to the latter exists); this result was extended in \cite{HenselMoser} to interfaces with boundary contact and in \cite{LauxStinsonUllrich} to the anisotropic case. The general approach has found numerous further applications: In \cite{FischerMarveggio}, convergence of the vectorial Allen-Cahn equation with multi-well potential towards multiphase mean curvature flow has been established by the second author and Marveggio; the case of the vectorial Allen-Cahn equation with two-well potential has been considered by Liu \cite{Liu1,Liu2}. Furthermore, Laux and Liu \cite{LauxLiu} have obtained the sharp-interface limit for a model for liquid crystals. Hensel and Liu \cite{HenselLiu} have used the relative entropy approach to study the sharp-interface limit of the Navier-Stokes/Allen-Cahn system \eqref{eq_NSAC1}-\eqref{eq_NSAC5} in the regime of nonvanishing mobility, deriving a Navier-Stokes/mean curvature flow system in the limit. In general, a key advantage of the relative entropy approach to sharp interface limits is its robustness, for instance requiring only convergence of the initial energy of solutions to the phase-field model. In contrast, the approach of matched asymptotic expansions may be used to establish an approximation of the diffuse interface model to arbitrary order.

      The main result of our contribution is as follows:
\begin{theorem}[\textbf{Convergence}]\label{th_conv_intro}
	Let $m_\eps:=m_0\eps^\beta$ for $\eps>0$, where $m_0>0$ and $\beta\in(0,2)$ are fixed, $q=2$ if $d=2$, and $q=\frac43$ if $d=3$. Moreover, let $T_0>0$ be such that the two-phase Navier-Stokes system with surface tension \eqref{eq_TwoPhase1}-\eqref{eq_TwoPhase7} has a smooth solution $(\ve_0^\pm,p_0^\pm,\Gamma)$ on $[0,T_0]$. Let $(\ve_\eps, \varphi_\eps)$ with $\ve_\eps\in L^\infty(0,T_0;L^2_\sigma(\Omega))\cap L^2(0,T_0;H^1_0(\Omega)^d)$, $\varphi_\eps\in L^2(0,T_0;H^2(\Omega))\cap W^1_{q}(0,T_0;L^2(\Omega))$ for $\eps>0$ be energy-dissipating weak solutions to the Navier-Stokes/Allen-Cahn system \eqref{eq_NSAC1}-\eqref{eq_NSAC5} on $[0,T_0]$ as in Remark \ref{th_weak_sol_NS} below for the constant mobility $m_\eps$ and for initial data $(\ve_{0,\eps},\varphi_{0,\eps})$ with energy uniformly bounded with respect to $\eps$ and satisfying
	\begin{align}\begin{split}\label{eq_well_prep0}
			\sum_{\pm}\int_{\Omega_0^\pm} \frac{1}{2}|\ve_{0,\eps}-\ve^\pm_{0,0}|^2\,dx + \int_\Omega\frac{\eps}{2}|\nabla\varphi_{0,\eps}|^2+ \frac{1}{\eps} W(\varphi_{0,\eps})-(\xi\cdot\nabla\psi_{0,\eps})\,dx\\ 
			+ \int_\Omega|\sigma\chi_{\Omega^+_0} -\psi_{0,\eps}|\min\{\dist_{\Gamma^0},1\}\,dx\leq C\frac{\eps^2}{m_\eps}\end{split}
	\end{align}
	for $\varepsilon>0$ sufficiently small, where $\psi(r):=\int_{-1}^r\sqrt{2W(s)}\,ds$ and $\psi_{0,\eps}:=\psi(\varphi_{0,\eps})$. Set $\psi_\eps:=\psi(\varphi_\eps)$. 
	
	Then for $\varepsilon>0$ small and a.e. $T\in[0,T_0]$ it holds
	\begin{subequations}\label{eq_conv0}
	\begin{align}
		\|(\ve_\eps-\ve^0)(.,T)\|_{L^2(\Omega)} + \|\sigma\chi_{\Omega^+_T} -\psi_\eps(.,T)\|_{L^1(\Omega)} &\leq C\left(\frac{\eps}{\sqrt{m_\eps}}+m_\eps\right)\\
		\|\nabla\ve_\eps-\nabla\ve^0\|_{L^2(0,T_0;L^2(\Omega))} & \leq C\left(\frac{\eps}{\sqrt{m_\eps}}+m_\eps\right)
	\end{align}
	\end{subequations}
        for some $C>0$ independent of $\eps>0$ and $T\in [0,T_0]$.
	Finally, there are well-prepared initial data $(\ve_{0,\eps},\varphi_{0,\eps})$ for $\eps>0$ small in the sense that \eqref{eq_well_prep0} is satisfied, even with rate $\eps^2$.
\end{theorem}
This result will be a consequence of Corollary \ref{th_conv} below.

Let us briefly comment on some aspects of our main theorem. First, note that the choice $m_\eps:=\eps^{2/3}$ leads to the best-possible overall convergence rate $O(\eps^{2/3})$ in \eqref{eq_conv0}. If one employs the Navier-Stokes/Allen-Cahn system as a numerical approximation of the two-phase flow with sharp interface, this thus suggests the choice of mobility $m_\eps \sim \eps^{2/3}$.

Second, our theorem requires a lower bound on the mobility of the form $m_\eps \gtrsim \eps^\beta$ for some $\beta<2$. With a slightly more careful argument, it would in fact be possible to weaken this assumption to $m_\eps \gtrsim \eps^2 |\log\eps|^3$. However, in the regime $m_\eps \lesssim \eps^2$ one would expect that the Allen-Cahn term no longer suffices to stabilize the Modica-Mortola type profile of the diffuse interface, making it susceptible to stretching or squeezing by drift: To see this heuristically, notice for instance that in this scaling regime the reaction term $-\tfrac{m_\eps}{\eps^2} W'(\varphi_\eps)$ from the Allen-Cahn equation \eqref{eq_NSAC3} no longer suffices to drive the values of $\varphi_\eps$ towards the minima of the potential $W$ in finite time. This possible stretching of the profile by advection in turn is expected to lead to errorneous capillary effects in \eqref{eq_NSAC1} and hence convergence to a different limiting system; see also \cite{AbelsLengeler}, in which this has been observed rigorously for certain choices $m_\eps=\eps^\beta$, $\beta>2$.

\begin{remark}\upshape\label{th_weak_sol_NS}
Let $q=2$ if $d=2$, and $q=\frac43$ if $d=3$. We note that weak solutions $(\ve_\eps, \varphi_\eps)$ with $\ve_\eps\in L^\infty(0,T_0;L^2_\sigma(\Omega))\cap L^2(0,T_0;H^1_0(\Omega)^d)$, $\varphi_\eps\in L^2(0,T_0;H^2(\Omega))\cap W^1_{q}(0,T_0;L^2(\Omega))$ for $\eps>0$ to the Navier-Stokes/Allen-Cahn system \eqref{eq_NSAC1}-\eqref{eq_NSAC5} on $[0,T_0]$ are understood in the sense of \cite[Definition~3]{HenselLiu} with the difference that we only require $\varphi_\eps\in W^1_{4/3}(0,T_0;L^2(\Omega))$ if $d=3$.
In particular, we assume that the energy inequality
\begin{align*}
  &\int_\Omega \left(\frac{|\ve_\eps(x,t)|^2}2+\frac{\eps}2|\nabla \varphi_\eps(x,t)|^2+ \frac{W(\varphi_\eps(x,t))}\eps\right) \,dx \\
  &\qquad +\int_0^t \int_\Omega \left(|\nabla \ve_\eps(x,\tau)|^2+ |\partial_t \varphi_\eps(x,\tau)+\ve_\eps(x,\tau)\cdot \nabla \varphi_\eps(x,\tau)|^2\right)\,dx \,d \tau\\
  & \quad \leq \int_\Omega \left(\frac{|\ve_{0,\eps}|^2}2+\frac{\eps}2|\nabla \varphi_{0,\eps}|^2+ \frac{W(\varphi_{0,\eps})}\eps\right) \,dx\quad \text{for every } t\in [0,T_0]
\end{align*}
holds true, which is essential for the following arguments.
Existence of weak solutions follows from the results in \cite{VariableDensityNSAC}. We note that in the latter contribution the authors do not prove the energy inequality. However, it can be proved in a standard manner using the energy identity for the solutions of the Galerkin system in   \cite{VariableDensityNSAC} and lower semi-continuity of norms. We also refer to \cite[Theorem~3.1]{GiorginiGrasselliWuJFA} for the mass-conserved variant of the system.
\end{remark}

Let us comment on the novelty of our contribution. As mentioned before, this is the first convergence result for \eqref{eq_NSAC1}-\eqref{eq_NSAC5} in the case of a vanishing mobility $m_\eps\to 0$, except for the special case $m_\eps =m_0 \sqrt{\eps}$ in two space dimensions. We use a relative entropy similar as in \cite{HenselLiu}, which extends the one for the Allen-Cahn equation used in \cite{FischerLauxSimon} to the coupled Navier-Stokes/Allen-Cahn system. The main step in the proof of convergence consists in showing a suitable estimate for the relative entropy. Parts of the arguments and calculations in our situation follow closely \cite{HenselLiu}, but certain estimates degenerate as $m_\eps\to 0$ and some terms become critical. Therefore essential new ingredients are needed.

More precisely, it will turn out that the main challenge in controlling the growth of the relative entropy for $m_\eps\ll 1$ consist of controlling terms involving the failure of equipartition of energy such as
\begin{align*}
\int \left(\frac{\eps}{2}|\partial_n\varphi_\eps|^2-\frac{1}{\eps}W(\varphi_\eps)\right) \partial_n\tilde{\eta}\,dx.
\end{align*}
A naive direct estimate would bound such terms by the square root of the relative entropy, insufficient for a subsequent control of the growth via the Gronwall lemma. Instead, we shall see that by careful integration by parts arguments and an approximation of the diffuse interface by a graph, they may in fact be controlled by the dissipation term from the Allen-Cahn equation and the relative entropy, provided that $m_\eps \gg \eps^2$.

We deal with these critical terms in Sections~\ref{sec_Erel_inequ_lastterm}-Section~\ref{sec_energy_out_strip}. The estimates make use of a suitable parametrization of a suitably chosen level set $\{\varphi_\eps= b(t)\}$ for some $b(t)\in (-\tfrac12,\tfrac12)$ up to an error controlled by the relative entropy using results from \cite{FischerHensel}. Moreover, it is essential for the construction of the relative entropy that we use sufficiently smooth solutions of a modification of the  limit system \eqref{eq_TwoPhase1}-\eqref{eq_TwoPhase7}, where the interface evolution \eqref{eq_TwoPhase5} is replaced by the 
convective mean curvature flow equation with vanishing mobility
\begin{equation*}
  V_{\Gamma_t} = \no_{\Gamma_t}\cdot \ve_0^\pm + m_\eps H_{\Gamma_t}\qquad \text{on }\Gamma_t,t\in[0,T_0].
\end{equation*}
This is used to obtain a sufficiently good approximation and control of some critical remainder terms. However, it makes the solution depend on $m_\eps$, $\eps$, respectively, and existence of such solutions together with uniform estimates in $\eps>0$ sufficiently small needs to be shown. {We note that existence of strong solutions for a fixed $m_\eps>0$ locally in time was shown by the first and third author in \cite{AbelsMoserNSMF} and in \cite[Appendix]{HenselLiu}. But the existence time might depend on $\eps$. To obtain the uniform bounds for small $\eps>0$ } a Hanzawa transformation {is used to transform} the modified system of the limit system \eqref{eq_TwoPhase1}-\eqref{eq_TwoPhase7}. Then a fixed point argument can be used to obtain strong solutions for $m_\eps$ sufficiently small using suitable uniform bounds for the $m_\eps$-dependent linearized system in spaces of maximal $L^q$-regularity.

The structure of the contribution is as follows: In Section~\ref{sec_Energy_def_coercivity} we define the energy-type functionals, which will be essential for the proof of convergence, and study their coercivity properties. Afterwards the central stability estimate and convergence result is given in Section~\ref{sec:Stability}. The essential estimates for the relative entropy are done in Section~\ref{sec_Erel_estimate}. The proof of the stability estimate, which implies the convergence result, is given in Section~\ref{sec_main_proof}. Finally, in Section~\ref{sec_approx_twophase} existence of strong solutions for the modification of the limit system \eqref{eq_TwoPhase1}-\eqref{eq_TwoPhase7} for sufficiently small $m_\eps>0$ is shown {and its difference to the solution of the real limit system is estimated by a multiple of $m_\eps$, cf.\ Theorem~\ref{th_approx_m_dependence} below. }

Let us finish the introduction with some notation. For example, we denote with $W^k_p(\Omega)$ for $k\in\mathbb{N}$, $1\leq p\leq\infty$ and some domain $\Omega$ the Sobolev space with $k$ weak derivatives and integrability exponent $p$. Moreover, let $H^k:=W^k_2$ and $L^2_\sigma(\Omega):=\overline{\{\ve\in C_0^\infty(\Omega)^d:\textup{div}\,\ve=0\}}^{L^2(\Omega)^d}$, where $C_0^\infty(\Omega)$ is the space of smooth functions with compact support in $\Omega$.

\section{Definition and Coercivity Properties of the Energy Functionals}\label{sec_Energy_def_coercivity}

In this section we define the energy/entropy functionals used for our Gronwall argument and show suitable coercivity properties. The definitions are similar to \cite{HenselLiu}, but based on a modification of the limit system as mentioned in the introduction. To this end, we need some notation.

Let $(\ve_\eps, \varphi_\eps)$ with $\ve_\eps\in L^\infty(0,T_0;L^2_\sigma(\Omega))\cap L^2(0,T_0;H^1_0(\Omega)^d)$ and $\varphi_\eps\in L^2(0,T_0;H^2(\Omega))\cap W^1_{q}(0,T_0;L^2(\Omega))$, where $q=2$ if $d=2$ and $q=\frac{4}{3}$ if $d=3$, for $\eps>0$ small be energy-dissipating weak solutions to the Navier-Stokes/Allen-Cahn system \eqref{eq_NSAC1}-\eqref{eq_NSAC5} on $[0,T_0]$ with constant mobility $m_\eps>0$ as in Remark \ref{th_weak_sol_NS}. Moreover, for $m_\eps>0$ small let $(\ve_{m_\eps}^\pm,p_{m_\eps}^\pm,(\Gamma_t^{m_\eps})_{t\in[0,T_0]})$ be solutions of the two-phase Navier-Stokes equation with surface tension on $[0,T_0]$ but with the evolution equation $V_{\Gamma^{m_\eps}_t}=\no_{\Gamma^{m_\eps}_t}\cdot \ve_{m_\eps}^\pm+m_\eps H_{\Gamma^{m_\eps}_t}$ on $\Gamma_t^{m_\eps}$ instead of \eqref{eq_TwoPhase5}, cf.~\eqref{eq:ApproxLimit1}-\eqref{eq:ApproxLimit7} and Theorem \ref{th:approx_lim_uniform} in Section \ref{sec_approx_twophase} below. We define 
\begin{align}\label{eq_notation_meps}
	\Gamma^{m_\eps}:=\bigcup_{t\in[0,T_0]}\Gamma^{m_\eps}_t\times\{t\},\quad \Omega^{m_\eps,\pm}:=\bigcup_{t\in[0,T_0]}\Omega_t^{m_\eps,\pm}\times\{t\}\quad\text{ and }\quad\ve^{m_\eps}:=\sum_\pm \ve_{m_\eps}^\pm\chi_{\Omega^{m_\eps,\pm}}.
\end{align}

Let $d_{\Gamma^{m_\eps}}$ be the signed distance function of $\Gamma^{m_\eps}$ and $P_{\Gamma^{m_\eps}}$ be the orthogonal projection on a tubular neighbourhood $\Gamma^{m_\eps}(2\delta)$ of $\Gamma^{m_\eps}$ of the width $2\delta$, where $\delta>0$ is sufficiently small and can be chosen independently of $m_\eps>0$ sufficiently small, cf.~Remark \ref{rem:UniformDeltaNeigborhood}. Let $\Gamma^{m_\eps}_t(2\delta):=\Gamma^{m_\eps}(2\delta)\cap (\R^d\times\{t\})$ be the time-slice. On $\Gamma^{m_\eps}(2\delta)$ we set
\begin{align}\label{eq_def_n_H}
\vec{n}:=\vec{n}_{\Gamma^{m_\eps}}|_{P_{\Gamma^{m_\eps}}}\quad\text{ and }\quad H:=H_{\Gamma^{m_\eps}}|_{P_{\Gamma^{m_\eps}}},
\end{align}
where we avoided to add $m_\eps$ in the notation for $\vec{n}$ and $H$ for convenience, and we used the notation $f|_{P_{\Gamma^{m_\eps}}}:=f\circ P_{\Gamma^{m_\eps}}$ for suitable $f$. Moreover, {we denote by $X_{m_\eps}:=(d_{\Gamma^{m_\eps}},P_{\Gamma^{m_\eps}},\textup{pr}_t)^{-1}$ on $(-2\delta,2\delta)\times\Gamma^{m_\eps}$ the coordinate map describing the tubular neighbourhood strip $\Gamma^{m_\eps}(2\delta)$} and define the normal derivative and the tangential gradient by 
\begin{align}\label{eq_deln_nablatau_def}
\partial_n:=\vec{n}\cdot\nabla\quad\text{ and }\quad\nabla_\tau:=(\textup{Id}-\vec{n}\otimes \vec{n})\nabla,\quad\text{respectively}.
\end{align} 
On $\Gamma^{m_\eps}(2\delta)$ it holds 
\begin{align}\label{eq_deln_nablatau_prop}
	\nabla&=\vec{n}\partial_n+\nabla_\tau,\qquad|\nabla u|^2=|\partial_n u|^2+|\nabla_\tau u|^2\quad\text{ for suitable }u,\\
\partial_n&=[\partial_1(.\circ {X_{m_\eps}})]\circ {X_{m_\eps}^{-1}}\quad\text{ and }\quad\nabla_\tau=D_xP_{\Gamma^{m_\eps}}\,[\nabla_{\Gamma^{m_\eps}_t}(.\circ {X_{m_\eps}})]\circ {X_{m_\eps}^{-1}}.\label{eq_nablatau_trafo}
\end{align}

Finally, let us recall
\begin{align}\label{eq_psi_def}
\psi(r):=\int_{-1}^r \sqrt{2W(s)}\,ds,\quad \sigma:=\psi(1)\quad\text{ and }\quad\psi_\eps:=\psi(\varphi_\eps)
\end{align}
and define 
\begin{align}\label{eq_n_eps}
\vec{n}_\eps:=\begin{cases}
		\frac{\nabla\varphi_\eps}{|\nabla\varphi_\eps|},\quad&\text{ if }\nabla\varphi_\eps\neq 0,\\
		\vec{s},\quad&\text{ else},
		\end{cases}
\end{align}
where $\vec{s}$ is a fixed unit vector in $\R^d$. Then $\psi_\eps$ and $\vec{n}_\eps$ are defined on $\Omega\times(0,T_0)$ and it holds
\[
\vec{n}_\eps|\nabla\varphi_\eps|=\nabla\varphi_\eps\quad\text{ and }\quad\vec{n}_\eps|\nabla\psi_\eps|=\nabla\psi_\eps.
\]

\subsection{Relative Entropy}\label{sec_Erel}
We define the relative entropy as follows for a.e.~$t\in[0,T_0]$:
\begin{align}\label{eq_Erel}
	E[\ve_\eps,\varphi_\eps| \ve^{m_\eps},\Gamma^{m_\eps}](t)&:=\int_\Omega \frac{1}{2}|\ve_\eps-\ve^{m_\eps}|^2(x,t)\,dx + E[\varphi_\eps|\Gamma^{m_\eps}](t),\\
	E[\varphi_\eps|\Gamma^{m_\eps}](t)&:= \int_\Omega \frac{\eps}{2}|\nabla\varphi_\eps|^2(x,t) + \frac{1}{\eps} W(\varphi_\eps(x,t))-(\xi\cdot\nabla\psi_\eps)(x,t)\,dx,\label{eq_Erel2}
\end{align}
where we chose to introduce a separate notation for the second interface-related part for convenience. Here $\xi$ is an extension (with quadratic cutoff) of the unit normal on $\Gamma^{m_\eps}$.
Note that the interface-related part of the relative entropy -- being the same as the relative entropy in \cite{FischerLauxSimon} -- is motivated by the Modica-Mortola trick; as we shall see below, it controls both the error in the equipartition of the diffuse interface energy \eqref{eq_Erel_coercivity_equipart} and a tilt-excess type error quantity for the interface normal \eqref{eq_Erel_coercivity_tilt}. At the same time, the time evolution of the relative entropy \eqref{eq_Erel2} can be calculated in a straightforward manner using the energy dissipation inequality for the Navier-Stokes/Allen-Cahn equation as well as the phase-field equation \eqref{eq_NSAC3}.

For the precise definition of $\xi:\overline{\Omega}\times[0,T_0]\rightarrow
\R^d$ and an accompanying $B:\overline{\Omega}\times[0,T_0]\rightarrow
\R^d$ (which has the role of an approximate transport and rotational velocity for $\xi$) we first introduce suitable cutoff functions (analogous to \cite[Proof of Theorem 1]{HenselLiu}). Let $\bar{\eta}:\R\rightarrow[0,1]$ be smooth and even such that $\supp\bar{\eta}\subseteq[-1,1]$ and such that it satisfies the quadratic decay
\begin{align}\label{eq_cutoff_quadr}
1-C_{\bar{\eta}} r^2 \leq \bar{\eta}(r) \leq 1-c_{\bar{\eta}} r^2 \quad\text{ and }\quad |\bar{\eta}'(r)|\leq C|r|\quad\text{ for }r\in[-1,1],
\end{align}
where $c_{\bar{\eta}}, C_{\bar{\eta}},C>0$. Moreover, let $\tilde{\eta}:\R\rightarrow[0,1]$ be smooth with $\supp\tilde{\eta}\subseteq[-2,2]$ and $\tilde{\eta}=1$ on $[-1,1]$. Finally, we set $\eta_{m_\eps}:=\bar{\eta}(\frac{d_{\Gamma^{m_\eps}}}{\delta})$ and $\tilde{\eta}_{m_\eps}:=\tilde{\eta}(\frac{d_{\Gamma^{m_\eps}}}{\delta})$. Then we define 
\begin{align}\label{eq_xi_B_def}
	\xi:= \vec{n}\eta_{m_\eps}\quad\text{ and }\quad B:=\ve^{m_\eps}+m_\eps H\vec{n}\tilde{\eta}_{m_\eps},
\end{align}
where $\vec{n}=\vec{n}(m_\eps)$ and $H=H(m_\eps)$ were defined in \eqref{eq_def_n_H}. For important properties of $\xi$ and $B$ we refer to Lemma \ref{th_xi_B} below.

The relative entropy from \eqref{eq_Erel} satisfies the following coercivity properties.
\begin{lemma}[\textbf{Coercivity Properties of the Relative Entropy}]\label{th_Erel_coercivity}
	First, the relative entropy provides a control of the velocity error in terms of
\begin{align}\label{eq_Erel_coercivity_v}
	\int_\Omega\frac{1}{2}|\ve_\eps-\ve^{m_\eps}|^2\,dx\leq E[\ve_\eps,\varphi_\eps| \ve^{m_\eps},\Gamma^{m_\eps}].
\end{align}
	Moreover, it yields a tilt-excess-type error estimate of the form
\begin{align}\label{eq_Erel_coercivity_tilt}
	\int_\Omega (1-\vec{n}_\eps\cdot\xi)|\nabla\psi_\eps|\,dx\leq E[\varphi_\eps|\Gamma^{m_\eps}].
\end{align}
	Additionally, we have some control of the error in the equipartition of the energy in the sense
\begin{align}\label{eq_Erel_coercivity_equipart}
	\int_\Omega\frac{1}{2}\left(\sqrt{\eps}|\nabla\varphi_\eps|-\frac{1}{\sqrt{\eps}}\sqrt{2W(\varphi_\eps)}\right)^2\,dx \leq E[\varphi_\eps|\Gamma^{m_\eps}].
\end{align}
	Furthermore, one obtains control of tangential derivatives and for the lack of equipartition of energy in normal direction: for a.e.~$t\in[0,T_0]$ it holds
\begin{align}\label{eq_coercivity_n_tau}
	\int_{\Gamma^{m_\eps}_t(2\delta)} \frac{\eps}{2}|\nabla_\tau\varphi_\eps|^2(x,t) + \frac{1}{2}\left(\sqrt{\eps}\partial_n\varphi_\eps-\frac{1}{\sqrt{\eps}}\sqrt{2W(\varphi_\eps)}\right)^2(x,t)\,dx \leq E[\varphi_\eps|\Gamma^{m_\eps}](t),
\end{align}
and one can replace $\partial_n\varphi_\eps$ by $|\partial_n\varphi_\eps|$ in the estimate. Finally, for some $C=C(\delta)>0$ it holds
\begin{align}\label{eq_Erel_coercivity_1}
	\int_\Omega \left(|\vec{n}_\eps-\xi|^2+\min\{d_{\Gamma^{m_\eps}}^2,1\}\right)\left(\eps|\nabla\varphi_\eps|^2+|\nabla\psi_\eps|\right)\,dx&\leq C E[\varphi_\eps|\Gamma^{m_\eps}],\\
	\int_\Omega\left(\min\{d_{\Gamma^{m_\eps}},1\}+\sqrt{1-\vec{n}_\eps\cdot\xi}\right)\left|\eps|\nabla\varphi_\eps|^2-|\nabla\psi_\eps|\right|\,dx &\leq C E[\varphi_\eps|\Gamma^{m_\eps}].\label{eq_Erel_coercivity_2}
\end{align}
\end{lemma}

\begin{proof}
	Up to \eqref{eq_coercivity_n_tau} the estimates are analogous to Hensel, Liu \cite[Lemma 5]{HenselLiu}. Note that one only uses the definitions and the elementary estimate $|\xi|\leq 1- c\min\{d_{\Gamma^{m_\eps}}^2,1\}$, cf.~also Lemma \ref{th_xi_B} below. The new estimate \eqref{eq_coercivity_n_tau} is a simple consequence of the properties \eqref{eq_deln_nablatau_prop} of $\partial_n$ and $\nabla_\tau$.
\end{proof}

\begin{lemma}[\textbf{Properties of $\xi$ and $B$}]\label{th_xi_B} 
	For $(\xi,B)=(\xi,B)(m_\eps)$ from \eqref{eq_xi_B_def} we have:
	\begin{enumerate}
	\item Regularity: for some $p>d+5$ it holds 
	\begin{align*}
		\xi&\in C^1([0,T_0],C^0(\overline{\Omega})^d) \cap C^0([0,T_0];C^2_c(\Omega)^d),\\
		B&\in C^0([0,T_0];C^{0,1}(\Omega)^d)\cap L^p(0,T_0;W^2_p(\Omega\setminus\Gamma^{m_\eps}_t)^d), \quad \nabla_\tau\nabla B\in L^\infty(\Gamma^{m_\eps}(\tfrac{3\delta}{2}))^{d^3},
	\end{align*}
	we have uniform bounds with respect to $m_\eps$ and $B|_{\partial\Omega}=0$.
	\item Coercivity and consistency: we have
	\begin{alignat}{2}\label{eq_calib1}
		|\xi| &\leq 1- c\min\{d_{\Gamma^{m_\eps}}^2,1\} &\quad &\text{ in }\Omega\times[0,T],\\
		\xi&=\vec{n}, \quad \nabla\cdot\xi= -H&\quad&\text{ on }\Gamma^{m_\eps},\label{eq_calib2}\\
		\left|\frac{B-\ve^{m_\eps}}{m_\eps}\cdot\xi + \nabla\cdot\xi\right| &\leq C\min\{d_{\Gamma^{m_\eps}},1\} &\quad &\text{ a.e. in }\Omega\times[0,T],\label{eq_calib3}
	\end{alignat}
where \eqref{eq_calib2} is a relation for the mean curvature and \eqref{eq_calib3} is an approximate equation for interface normal velocity.
	\item Evolution equations for $\xi$: it holds (where $\nabla B$ is defined to be the Jacobian)
	\begin{alignat}{2}
		|(\partial_t+B\cdot\nabla)|\xi|^2|                        &\leq C\min\{d_{\Gamma^{m_\eps}}^2,1\} &\quad &\text{ a.e. in }\Omega\times[0,T],\label{eq_calib4}\\
		|\partial_t\xi + (B\cdot \nabla)\xi + (\textup{Id}-\xi\otimes\xi)(\nabla B)^\top\xi| &\leq C\min\{d_{\Gamma^{m_\eps}},1\}   &\quad &\text{ a.e. in }\Omega\times[0,T],\label{eq_calib5}
	\end{alignat}\end{enumerate}
where \eqref{eq_calib4} means that $|\xi|^2$ is approximately transported by the vector field $B$ and \eqref{eq_calib5} that $\xi$ is approximately transported and rotated by $B$.
\end{lemma}
\begin{proof}
	The regularity and uniform bounds for $\xi$ are obtained directly from the ones for $d_{\Gamma^{m_\eps}}$ in Theorem \ref{th:approx_lim_uniform}. Moreover, $\ve^{m_\eps}$ instead of $B$ satisfies the regularity, uniform bounds and the boundary condition stated for $B$ due to Theorem \ref{th:approx_lim_uniform} together with embeddings and interpolation theory. Moreover, the second term in the definition \eqref{eq_xi_B_def} of $B$ is contained in $C^0([0,T_0],C_c^2(\Omega)^2)$ because of Theorem \ref{th:approx_lim_uniform} and due to the $m_\eps$-prefactor and the well-known identities $H=(-\Delta d_{\Gamma^{m_\eps}})|_{P_{\Gamma^{m_\eps}}}$, $\vec{n}=\nabla d_{\Gamma^{m_\eps}}$ and $P_{\Gamma^{m_\eps}}=\textup{Id}-\vec{n} d_{\Gamma^{m_\eps}}$.
	
	The properties \eqref{eq_calib1}-\eqref{eq_calib2} are clear from the definition and the well-known identities $\vec{n}=\nabla d_{\Gamma^{m_\eps}}$, $\nabla\cdot\vec{n}=\Delta d_{\Gamma^{m_\eps}}$ on $\Gamma^{m_\eps}(2\delta)$ and $\Delta d_{\Gamma^{m_\eps}}|_{\Gamma^{m_\eps}}=-H$ on $\Gamma^{m_\eps}$. Moreover, a direct calculation gives
	\[
	\frac{B-\ve^{m_\eps}}{m_\eps}\cdot\xi + \nabla\cdot\xi = H \eta_{m_\eps}\tilde{\eta}_{m_\eps} + \nabla\cdot\xi.
	\]
	The latter vanishes on $\Gamma^{m_\eps}$, hence \eqref{eq_calib3} follows. Moreover,
	\begin{align}\label{eq_dist_evol}
	-\partial_td_{\Gamma^{m_\eps}}=V_{\Gamma^{m_\eps}}=\ve^{m_\eps}\cdot\vec{n} + m_\eps H = B\cdot\nabla d_{\Gamma^{m_\eps}}\quad\text{ on }\Gamma^{m_\eps}
	\end{align}
	 due to \eqref{eq:ApproxLimit6} for $m_\eps$ instead of $m$ and therefore
	\begin{align}\label{eq_calib_proof1}
		(\partial_t+B\cdot\nabla)d_{\Gamma^{m_\eps}}=0\quad\text{ on }\Gamma^{m_\eps}. 
	\end{align}
	We compute $|\xi|^2=\eta_{m_\eps}^2$ and
	\[
	(\partial_t+B\cdot\nabla)|\xi|^2=\frac{2}{\delta}\eta_{m_\eps} \bar{\eta}'(\frac{d_{\Gamma^{m_\eps}}}{\delta}) (\partial_t+B\cdot\nabla)d_{\Gamma^{m_\eps}}.
	\]
	Due to \eqref{eq_calib_proof1}, $\bar{\eta}'(0)=0$ together with a transformation in tubular neighbourhood coordinates and the Taylor Theorem, we obtain \eqref{eq_calib4}.
	
	Finally, let us prove \eqref{eq_calib5}. Equation \eqref{eq_dist_evol} yields by definition 
	\[
	\partial_td_{\Gamma^{m_\eps}}+\ve^{m_\eps}|_{P_{\Gamma^{m_\eps}}}\cdot\vec{n} + m_\eps H = 0 \quad\text{ on }\Gamma^{m_\eps}(2\delta).
	\]
	Differentiating the previous identity implies
	\begin{align}\label{eq_normal_evol}
	\partial_t\vec{n}+\nabla(\ve^{m_\eps}|_{P_{\Gamma^{m_\eps}}})^\top\cdot\vec{n} + (\nabla\vec{n})^\top \ve^{m_\eps}|_{P_{\Gamma^{m_\eps}}} + m_\eps \nabla H = 0 \quad\text{ on }\Gamma^{m_\eps}(2\delta),
	\end{align}
	where $\nabla(\ve^{m_\eps}|_{P_{\Gamma^{m_\eps}}})=\nabla \ve^{m_\eps}|_{P_{\Gamma^{m_\eps}}}\nabla P_{\Gamma^{m_\eps}}$ and it is well-known that $\nabla P_{\Gamma^{m_\eps}}|_{\Gamma^{m_\eps}}=\textup{Id}-\vec{n}\otimes\vec{n}$ on $\Gamma^{m_\eps}$. Moreover, on $\Gamma^{m_\eps}$ it holds
	\begin{align*}
		&\partial_t\xi + (B\cdot \nabla)\xi + (\textup{Id}-\xi\otimes\xi)(\nabla B)^\top\xi \\
		&= \partial_t\vec{n} + (\ve^{m_\eps}+m_\eps H\vec{n})\cdot\nabla\vec{n}+(\textup{Id}-\vec{n}\otimes\vec{n})\left[\nabla\ve^{m_\eps}+m_\eps(\vec{n}\nabla H^\top+H\nabla\vec{n})\right]^\top\vec{n}\\
		&=\partial_t\vec{n} + \ve^{m_\eps}\cdot\nabla\vec{n}+(\textup{Id}-\vec{n}\otimes\vec{n})(\nabla\ve^{m_\eps})^\top\vec{n}+m_\eps\nabla H,
	\end{align*}
	where we used $(\nabla\vec{n})^\top\vec{n}=0$ and $\vec{n}\cdot\nabla\vec{n}=\vec{n}\cdot\nabla H=0$. Finally, due to $\vec{n}=\nabla d_{\Gamma^{m_\eps}}$ on $\Gamma^{m_\eps}(2\delta)$ we have $\partial_{x_j}\vec{n}=\nabla n_j$ for $j=1,...,d$ and hence $\ve^{m_\eps}\cdot \nabla\vec{n}=(\nabla\vec{n})^\top \ve^{m_\eps}$ on $\Gamma^{m_\eps}(2\delta)$. Therefore we obtain \eqref{eq_calib5} from \eqref{eq_normal_evol}.
\end{proof}

\subsection{Bulk Error Functional}\label{sec_Ebulk}
We define the bulk error functional for all $t\in[0,T_0]$ by
\begin{align}\label{eq_Ebulk}
E_\textup{bulk}[\varphi_\eps|\Gamma^{m_\eps}](t):=\int_\Omega \left(\sigma\chi_{\Omega^{m_\eps,+}_t} -\psi_\eps(x,t)\right)\vartheta(x,t)\,dx,
\end{align}
where $\vartheta:\overline{\Omega}\times[0,T_0]\rightarrow[0,1]$ is defined as $\vartheta:=\overline{\vartheta}(\frac{d_{\Gamma^{m_\eps}}}{\delta})$, where $\overline{\vartheta}:\R\rightarrow[0,1]$ is smooth with 
\begin{align*}
\overline{\vartheta}>0\text{ on }(1,\infty),\quad\overline{\vartheta}<0\text{ on }(-\infty,0)\quad\text{ and }|\overline{\vartheta}|=1\text{ on }\R\setminus[-1,1],
\end{align*}
as well as $c_{\overline{\vartheta}}|r|\leq|\overline{\vartheta}(r)|\leq C_{\overline{\vartheta}}|r|$ for all $r\in[-1,1]$ and some $c_{\overline{\vartheta}},C_{\overline{\vartheta}}>0$. Note that we use a different sign convention for $\overline{\vartheta}$ as in \cite{HenselLiu}. Hence $\vartheta$ is roughly proportional to the signed distance function of $\Gamma^{m_\eps}$ close to $\Gamma^{m_\eps}$ and appropriately truncated to $\pm1$ outside. The required properties of $\vartheta$ will be shown in Lemma \ref{th_theta} below.

Let us now prove coercivity properties for the bulk error functional:
\begin{lemma}[\textbf{Coercivity Properties of }$E_\textup{bulk}$]\label{th_Ebulk_coercivity}
	It holds $\left(\sigma\chi_{\Omega^{m_\eps,+}} -\psi_\eps\right)\vartheta\geq 0$,
	in particular $E_\textup{bulk}[\varphi_\eps|\Gamma^{m_\eps}]\geq 0$. Moreover,
	\begin{align}\label{eq_Ebulk_coercivity1}
		\int_\Omega\left|\sigma\chi_{\Omega^{m_\eps,+}} -\psi_\eps\right|\min\{d_{\Gamma^{m_\eps}},1\}\,dx + \|\sigma\chi_{\Omega^{m_\eps,+}} -\psi_\eps\|_{L^1(\Omega)}^2 \leq C E_\textup{bulk}[\varphi_\eps|\Gamma^{m_\eps}].
	\end{align}
	Finally, for all $c_0>0$ there exists $C=C(c_0)>0$ such that
	\begin{align}\begin{split}\label{eq_Ebulk_coercivity2}
	    &\int_\Omega|\sigma\chi_{\Omega^{m_\eps,+}} -\psi_\eps||\ve_\eps-\ve^{m_\eps}|\,dx\\
	    &\leq c_0\int_\Omega|\nabla\ve_\eps-\nabla\ve^{m_\eps}|^2\,dx + C\left(E[\ve_\eps,\varphi_\eps| \ve^{m_\eps},\Gamma^{m_\eps}] + E_\textup{bulk}[\varphi_\eps|\Gamma^{m_\eps}]\right).\end{split}
	\end{align}
\end{lemma}
\begin{proof}
	Because of $|\varphi_\eps|\leq 1$ due to the maximum principle, it holds $|\psi_\eps|\leq\sigma$ by definition. Hence the properties of $\overline{\vartheta}$ yield $\left(\sigma\chi_{\Omega^{m_\eps,+}} -\psi_\eps\right)\vartheta\geq 0$. Moreover,
	\[
	|\sigma\chi_{\Omega^{m_\eps,+}} -\psi_\eps|\min\{d_{\Gamma^{m_\eps}},1\}\leq C|\sigma\chi_{\Omega^{m_\eps,+}} -\psi_\eps||\vartheta|=C\left(\sigma\chi_{\Omega^{m_\eps,+}} -\psi_\eps\right)\vartheta
	\]
	and this estimates the first term in \eqref{eq_Ebulk_coercivity1}. This yields that the second term in \eqref{eq_Ebulk_coercivity1} is controlled by using the inequality (cf.~\cite[Proof of Theorem 1]{FischerLauxSimon})
	\begin{align}\label{eq_aux_inequ_Fubini}
		\left(\int_0^\delta |g|\,dr\right)^2\leq 2\|g\|_{L^\infty(0,\delta)}\int_0^\delta |g|(r)r\,dr\quad\text{ for all }g\in L^\infty(0,\delta),
	\end{align}
	which is derived by dividing the square $[0,\delta]^2$ into two triangles and applying Fubini's theorem.
	
	Finally, \eqref{eq_Ebulk_coercivity2} can be shown analogously to \cite[proof of (31)]{HenselLiu} with \eqref{eq_Ebulk_coercivity1} and elementary estimates, in particular the Gagliardo-Nirenberg inequality in normal direction of $\Gamma^{m_\eps}$, the Hölder and Young inequality as well as \eqref{eq_aux_inequ_Fubini}.
\end{proof}
\begin{lemma}[\textbf{Properties of $\vartheta$}]\label{th_theta}
	For $\vartheta=\vartheta(m_\eps)$ defined after \eqref{eq_Ebulk} the following properties hold:
	\begin{enumerate}
		\item Regularity: it holds
		\[
		\vartheta\in C^1([0,T_0],C^1(\overline{\Omega})) \cap C^0([0,T_0];C^3(\overline{\Omega}))
		\]
		and we have uniform bounds with respect to $m_\eps$.
		\item Coercivity and consistency: we have 
		\begin{align}\label{eq_theta_coerc1}
		&c\min\{d_{\Gamma^{m_\eps}},1\}\leq |\vartheta|\leq C\min\{d_{\Gamma^{m_\eps}},1\}\quad\text{for some }c,C>0,\\
		&\vartheta>0\text{ in }\Omega^{m_\eps,+},\quad \vartheta<0\text{ in }\Omega^{m_\eps,-}.\label{eq_theta_coerc2}
		\end{align}
		\item Evolution equation: it holds
		\begin{align}\label{eq_theta_evol}
			|(\partial_t+B\cdot\nabla)\vartheta|\leq C\min\{d_{\Gamma^{m_\eps}},1\}\quad\text{ a.e.~on }\Omega\times(0,T_0).
		\end{align}
	\end{enumerate}
\end{lemma}
\begin{proof}
	The regularity and uniform bounds for $\vartheta$ follow directly from the ones for $d_{\Gamma^{m_\eps}}$ by Theorem \ref{th:approx_lim_uniform}. The estimates \eqref{eq_theta_coerc1}-\eqref{eq_theta_coerc2} follow directly from the definition of $\vartheta$ and the properties of $\overline{\vartheta}$. Moreover, \eqref{eq_theta_evol} is shown via the chain rule and $(\partial_t+B\cdot\nabla)d_{\Gamma^{m_\eps}}=0$ on $\Gamma^{m_\eps}$, cf.~\eqref{eq_calib_proof1}.
\end{proof}

\section{Stability Estimate and Convergence Result}\label{sec:Stability}
In this section we formulate our main results on stability and quantitative convergence for solutions of the Navier-Stokes/Allen-Cahn system \eqref{eq_NSAC1}-\eqref{eq_NSAC5} towards solutions of the classical two-phase Navier-Stokes equation \eqref{eq_TwoPhase1}-\eqref{eq_TwoPhase7} for suitable scalings of the mobility $m_\eps$. 

We obtain the following stability result:

\begin{theorem}[\textbf{Stability Estimate}]\label{th_stability}
	Let $m_\eps:=m_0\eps^\beta>0$, where $m_0>0$ and $\beta\in(0,2)$ are fixed. Let $T_0>0$ be such that the two-phase Navier-Stokes system with surface tension \eqref{eq_TwoPhase1}-\eqref{eq_TwoPhase7} has a smooth solution $(\ve_0^\pm,p_0^\pm,\Gamma)$ on $[0,T_0]$. Moreover, let $(\ve_{m_\eps}^\pm,p_{m_\eps}^\pm,\Gamma^{m_\eps})$ be strong solutions to the modified system \eqref{eq:ApproxLimit1}-\eqref{eq:ApproxLimit7} for $\eps>0$ small, cf.~Theorem \ref{th:approx_lim_uniform} below. Furthermore, let $(\ve_\eps, \varphi_\eps)$ for $\eps>0$ be energy-dissipating weak solutions to the Navier-Stokes/Allen-Cahn system \eqref{eq_NSAC1}-\eqref{eq_NSAC5} on $[0,T_0]$ for the constant mobility $m_\eps$ as in Remark \ref{th_weak_sol_NS} starting from initial data with energy uniformly bounded with respect to $\eps$. We use the notation from Section \ref{sec_Energy_def_coercivity}, in particular we define the relative energy functional $E[\ve_\eps,\varphi_\eps| \ve^{m_\eps},\Gamma^{m_\eps}]$ and the bulk error functional $E_\textup{bulk}[\varphi_\eps|\Gamma^{m_\eps}]$ as in \eqref{eq_Erel} and \eqref{eq_Ebulk}. 
	
	Then for $\eps>0$ small and a.e. $T\in[0,T_0]$ it holds
	\begin{align}\begin{split}\label{eq_stab}
			&\frac{1}{2}\|\nabla\ve_\eps-\nabla\ve^{m_\eps}\|_{L^2(0,T;L^2(\Omega))}^2+E[\ve_\eps,\varphi_\eps| \ve^{m_\eps},\Gamma^{m_\eps}](T)+E_\textup{bulk}[\varphi_\eps|\Gamma^{m_\eps}](T)\\
			&\leq e^{C(\beta,m_0)T}\left(E[\ve_\eps,\varphi_\eps| \ve^{m_\eps},\Gamma^{m_\eps}](0)+E_\textup{bulk}[\varphi_\eps|\Gamma^{m_\eps}](0)+C\frac{\eps^2}{m_\eps}\right).\end{split}
	\end{align}
      \end{theorem}
    
The proof is done via a Gronwall-type argument in Section \ref{sec_main_proof}, using coercivity properties for the relative entropy and the bulk error in Section \ref{sec_Energy_def_coercivity} as well as preliminary estimates in Section \ref{sec_Erel_estimate} and Section \ref{sec_Ebulk}. As an immediate consequence of Theorem \ref{th_stability} we get:

\begin{corollary}[\textbf{Convergence Result}]\label{th_conv}
	Let the assumptions of Theorem \ref{th_stability} hold. Moreover, let the initial data satisfy
	\begin{align}\label{eq_well_prep}
		E[\ve_\eps,\varphi_\eps| \ve^0,\Gamma](0)+E_\textup{bulk}[\varphi_\eps|\Gamma](0) \leq C\frac{\eps^2}{m_\eps}
	\end{align}
	for $\varepsilon>0$ small. Then for $\varepsilon>0$ small and a.e. $T\in[0,T_0]$ it holds
	\begin{align}\begin{split}\label{eq_conv1}
		\|(\ve_\eps-\ve^{m_\eps})(.,T)\|_{L^2(\Omega)} + \|\sigma\chi_{\Omega^{m_\eps,+}_T} -\psi_\eps(.,T)\|_{L^1(\Omega)} &\leq Ce^{C(\beta,m_0)T}\frac{\eps}{\sqrt{m_\eps}},\\
		\|\nabla\ve_\eps-\nabla\ve^{m_\eps}\|_{L^2(0,T;L^2(\Omega))}&\leq Ce^{C(\beta,m_0)T}\frac{\eps}{\sqrt{m_\eps}}\end{split}
	\end{align}
	and for the true limit we obtain 
	\begin{align}\begin{split}\label{eq_conv2}
		\|(\ve_\eps-\ve^0)(.,T)\|_{L^2(\Omega)} + \|\sigma\chi_{\Omega^+_T} -\psi_\eps(.,T)\|_{L^1(\Omega)} &\leq C\left(e^{C(\beta,m_0)T}\frac{\eps}{\sqrt{m_\eps}}+m_\eps\right),\\
		\|\nabla\ve_\eps-\nabla\ve^0\|_{L^2(0,T;L^2(\Omega))} &\leq C\left(e^{C(\beta,m_0)T}\frac{\eps}{\sqrt{m_\eps}}+m_\eps\right).\end{split}
	\end{align}
	Finally, there are well-prepared initial data $(\ve_{0,\eps},c_{0,\eps})$ for $\eps>0$ small in the sense that \eqref{eq_well_prep} is satisfied, even with rate $\eps^2$.
\end{corollary}
\begin{proof}
	Estimate \eqref{eq_conv1} is a direct consequence of Theorem \ref{th_stability} and the coercivity estimates \eqref{eq_Erel_coercivity_v} from Lemma \ref{th_Erel_coercivity} and \eqref{eq_Ebulk_coercivity1} from Lemma \ref{th_Ebulk_coercivity}. Then \eqref{eq_conv2} follows from Theorem \ref{th_approx_m_dependence}. Because of $(\ve_{m_\eps}^\pm,\Gamma^{m_\eps})(0)=(\ve_{0,0}^\pm,\Gamma_0)$, it holds $E[\ve_\eps,\varphi_\eps| \ve^{m_\eps},\Gamma^{m_\eps}](0)=E[\ve_\eps,\varphi_\eps| \ve^0,\Gamma](0)$ and $E_\textup{bulk}[\varphi_\eps|\Gamma^{m_\eps}](0)=E_\textup{bulk}[\varphi_\eps|\Gamma](0)$. The existence of well-prepared initial data is well-known, cf.~\cite[Proof of Theorem 1]{FischerLauxSimon} and \cite[Appendix B]{HenselMoser}.
\end{proof}

\section{Relative Entropy Estimate}\label{sec_Erel_estimate}
Let the assumptions and notation of Section \ref{sec_Energy_def_coercivity} be in place. In this section we derive an inequality for the relative entropy $E[\ve_\eps,\varphi_\eps| \ve^{m_\eps},\Gamma^{m_\eps}]$ defined in \eqref{eq_Erel} that can be employed later to obtain a Gronwall-type estimate. We use the notation
\begin{align}\label{eq_Heps_def}
	H_\eps := -\eps\Delta\varphi_\eps + \frac{1}{\eps} W'(\varphi_\eps).
\end{align}
\subsection{Preliminary Relative Entropy Inequality}
We derive the first important estimate for the relative entropy:
\begin{lemma}[\textbf{Relative Entropy Inequality}]\label{th_Erel_inequ} Let the assumptions and notation of Section \ref{sec_Energy_def_coercivity} be valid and $H_\eps$ be defined as in \eqref{eq_Heps_def}. More precisely, let $(\ve_{m_\eps}^\pm,p_{m_\eps}^\pm,(\Gamma_t^{m_\eps})_{t\in[0,T_0]})$ for $m_\eps>0$ small be solutions of the adjusted two-phase Navier-Stokes equation \eqref{eq:ApproxLimit1}-\eqref{eq:ApproxLimit7} on $[0,T_0]$, cf.~Theorem \ref{th:approx_lim_uniform} below. Additionally, let $(\ve_\eps, \varphi_\eps)$ for $\eps>0$ small be energy dissipating weak solutions to the Navier-Stokes/Allen-Cahn system \eqref{eq_NSAC1}-\eqref{eq_NSAC5} on $[0,T_0]$ with constant mobility $m_\eps>0$ as in Remark \ref{th_weak_sol_NS}. Finally, let $\Gamma^{m_\eps}, \Omega^{m_\eps,\pm}, \ve^{m_\eps}$ be as in \eqref{eq_notation_meps}, $\sigma$, $\psi_\eps$, $\vec{n}_\eps$ be as in \eqref{eq_psi_def}-\eqref{eq_n_eps}, $\xi$, $B$ be as in \eqref{eq_xi_B_def}, $E[\ve_\eps,\varphi_\eps| \ve^{m_\eps},\Gamma^{m_\eps}]$ be as in \eqref{eq_Erel}. Then for a.e.~$T\in[0,T_0]$ we obtain:\begin{subequations}
	\begin{align}
		E[\ve_\eps,&\varphi_\eps| \ve^{m_\eps},\Gamma^{m_\eps}](T)\leq E[\ve_\eps,\varphi_\eps| \ve^{m_\eps},\Gamma^{m_\eps}](0)-\int_0^T\int_\Omega |\nabla \ve_\eps-\nabla \ve^{m_\eps}|^2\,dx\,dt\notag\\
		& - \int_0^T \int_\Omega \frac{m_\eps}{2\eps}\left|H_\eps + \sqrt{2W(\varphi_\eps)}\nabla\cdot\xi\right|^2 dx\,dt \label{eq_Erel_inequ_pos1}\\
		& - \int_0^T \int_\Omega \frac{m_\eps}{2\eps}\left|H_\eps - \frac{B-\ve^{m_\eps}}{m_\eps}\cdot \xi\,\eps|\nabla\varphi_\eps|\right|^2 dx\,dt \label{eq_Erel_inequ_pos2}\\
		& - \int_0^T \int_\Omega (\ve_\eps-\ve^{m_\eps})\cdot ((\ve_\eps-\ve^{m_\eps})\cdot\nabla)\ve^{m_\eps}\,dx\,dt\notag\\
		& - \int_0^T \int_\Omega (\sigma\chi_{\Omega^{m_\eps,+}}-\psi_\eps) ((\ve_\eps-\ve^{m_\eps})\cdot\nabla)(\nabla\cdot\xi)\,dx\,dt\notag\\
		& + \int_0^T \int_\Omega m_\eps \left|\frac{B-\ve^{m_\eps}}{m_\eps}\cdot\xi + \nabla\cdot\xi\right|^2 \eps|\nabla\varphi_\eps|^2\,dx\,dt \label{eq_Erel_inequ_problem1}\\
		& + \int_0^T \int_\Omega m_\eps |\nabla\cdot\xi|^2\left(\frac{\sqrt{2W(\varphi_\eps)}}{\sqrt{\eps}} - \sqrt{\eps}|\nabla\varphi_\eps|\right)^2 dx\,dt\notag\\
		& - \int_0^T \int_\Omega \frac{1}{\sqrt{\eps}}\left(H_\eps+\sqrt{2W(\varphi_\eps)}\nabla\cdot\xi\right) (\ve^{m_\eps}-B)\cdot(\vec{n}_\eps-\xi)\sqrt{\eps}|\nabla\varphi_\eps|\,dx\,dt \label{eq_Erel_inequ_problem2}\\
		& - \int_0^T \int_\Omega (\partial_t\xi+(B\cdot\nabla)\xi+(\textup{Id}-\xi\otimes\xi)(\nabla B)^\top\xi)\cdot(\vec{n}_\eps-\xi)|\nabla\psi_\eps|\,dx\,dt\notag\\
		& - \int_0^T \int_\Omega (\xi\otimes\xi(\nabla B)^\top\xi)\cdot(\vec{n}_\eps-\xi)|\nabla\psi_\eps|\,dx\,dt\label{eq_Erel_inequ_new}\\
		& - \int_0^T \int_\Omega \left((\partial_t+B\cdot\nabla)|\xi|^2\right) |\nabla\psi_\eps|\,dx\,dt\notag\\
		& - \int_0^T \int_\Omega \nabla B:(\xi-\vec{n}_\eps)\otimes(\xi-\vec{n}_\eps)|\nabla\psi_\eps|\,dx\,dt\notag\\
		& + \int_0^T \int_\Omega (\nabla\cdot B)(1-\xi\cdot \vec{n}_\eps)|\nabla\psi_\eps|\,dx\,dt\notag\\
		& + \int_0^T \int_\Omega (\nabla\cdot B)\frac{1}{2}\left(\sqrt{\eps}|\nabla\varphi_\eps|-\frac{1}{\sqrt{\eps}}\sqrt{2W(\varphi_\eps)}\right)^2 dx\,dt\notag\\
		& - \int_0^T \int_\Omega (\vec{n}_\eps\otimes \vec{n}_\eps-\xi\otimes\xi):\nabla B(\eps|\nabla\varphi_\eps|^2-|\nabla\psi_\eps|)\,dx\,dt\notag\\
		& - \int_0^T \int_\Omega \xi\otimes\xi:\nabla B(\eps|\nabla\varphi_\eps|^2-|\nabla\psi_\eps|)\,dx\,dt. \label{eq_Erel_inequ_problem3}
	\end{align}\end{subequations}
\end{lemma}
\begin{proof}
Unlike in \cite{HenselLiu}, our choice of $B$ does not have compact support, but due to the boundary condition $B|_{\partial\Omega}=0$ by Lemma \ref{th_xi_B} we observe that analogous computations as in the proof of \cite[Proposition 6]{HenselLiu} may be carried out.
\end{proof}

\begin{remark}\label{th_Erel_inequ_problem_rem}\upshape
   { The choice of $B$ in \eqref{eq_xi_B_def}, i.e.~$B:=\ve^{m_\eps}+m_\eps H\vec{n}\tilde{\eta}_{m_\eps}$ with the plateau cutoff $\tilde{\eta}_{m_\eps}$, is natural in order to control the term \eqref{eq_Erel_inequ_problem1}, i.e.~$\int_0^T \int_\Omega m_\eps \left|\frac{B-\ve^{m_\eps}}{m_\eps}\cdot\xi + \nabla\cdot\xi\right|^2 \eps|\nabla\varphi_\eps|^2\,dx\,dt$. Note that in Hensel, Liu \cite{HenselLiu} the projected velocity field $(\vec{n}\cdot \ve^{m_\eps}|_{P_{\Gamma^{m_\eps}}})\vec{n}$ is used within the definition of $B$ instead. This is not possible here because one would then obtain from \eqref{eq_Erel_inequ_problem1} a remainder of the form}
    \[
 {   \frac{1}{m_\eps}\int_0^T\int_\Omega\min\{d_{\Gamma^{m_\eps}}^2,1\}\eps|\nabla\varphi_\eps|^2\,dx\,dt,}
    \]
{    which is only controlled by $\frac{C}{m_\eps}\int_0^T E[\varphi_\eps|\Gamma^{m_\eps}](t)\,dt$ due to \eqref{eq_Erel_coercivity_1}. However, with the new choice of $B$ it is not clear anymore how to estimate the last term \eqref{eq_Erel_inequ_problem3}, i.e.~}
    \[
{    \int_0^T \int_\Omega \xi\otimes\xi:\nabla B(\eps|\nabla\varphi_\eps|^2-|\nabla\psi_\eps|)\,dx\,dt.}
    \]
{    To this end we rewrite and estimate the term \eqref{eq_Erel_inequ_problem3} in a novel way in the following Section \ref{sec_Erel_inequ_lastterm}.} The idea is to write $\xi\otimes\xi:\nabla B$ as a normal derivative and use integration by parts in a suitable way. The other terms in the relative entropy estimate from Lemma \ref{th_Erel_inequ} will turn out to be controllable with the choice of $B$ in \eqref{eq_xi_B_def}, the coercivity properties of the relative entropy and the bulk error functional from Section \ref{sec_Erel} and Section \ref{sec_Ebulk}, cf.~Section \ref{sec_main_proof} below.
\end{remark}
\subsection{The Remaining Problematic Term}\label{sec_Erel_inequ_lastterm}
For a.e.~$t\in[0,T_0]$ consider a $1$-Lipschitz-function $h=h_\eps(.,t):\Gamma^{m_\eps}_t\rightarrow(-\delta,\delta)$. This function $h$ will be constructed in Section \ref{sec_level_set_h} below, and it will be used to approximate a suitable level set of $\varphi_\eps(.,t)$. Moreover, the energy outside a strip around the graph 
\[
\Gamma^{m_\eps}_{t,h}:=\{s+h_\eps(s,t)\vec{n}(s,t):s\in \Gamma_t^{m_\eps}\}
\] 
over $\Gamma^{m_\eps}_t$ determined by $h$ will be estimated in Section \ref{sec_energy_out_strip} below. Let us define shifted tubular neighbourhoods for a.e.~$t\in[0,T_0]$:
\begin{align}\label{eq_tub_neighb_shifted}
\Gamma^{m_\eps}_{t,h}(\tilde{\delta}):=\{x\in \Gamma^{m_\eps}_t(2\delta):d_{\Gamma^{m_\eps}}(x,t)\in h_\eps(P_{\Gamma^{m_\eps}}(x,t),t)+(-\tilde{\delta},\tilde{\delta})\}\quad\text{ for }\tilde{\delta}\in(0,\delta]. 
\end{align}

In order to estimate the problematic term from Remark \ref{th_Erel_inequ_problem_rem}, we need the following lemma whose proof is based on integration by parts and will be postponed.

\begin{lemma}\label{th_equipart_rewrite}
	For a.e.~$t\in[0,T_0]$ let $h=h_\eps(.,t):\Gamma^{m_\eps}_t\rightarrow(-\delta,\delta)$ be a $1$-Lipschitz-function. Moreover, let $\delta_\eps\in(0,\frac{\delta}{2}]$ for $\eps>0$ small, $\Gamma^{m_\eps}_{t,h}(\delta_\eps)$ be defined as in \eqref{eq_tub_neighb_shifted} and $\tilde{\eta}\in C^{0,1}_c(\Gamma^{m_\eps}_{t,h}(\delta_\eps))$. Then for a.e.~$t\in[0,T_0]$ it holds\begin{subequations}
	\begin{align}
		&\int_{\Gamma^{m_\eps}_{t,h}(\delta_\eps)} \left(\frac{\eps}{2}|\partial_n\varphi_\eps|^2-\frac{1}{\eps}W(\varphi_\eps)\right) \partial_n\tilde{\eta}\,dx\label{eq_equipart1}\\
		&=\int_{\Gamma^{m_\eps}_{t,h}(\delta_\eps)}\left(\frac{\eps}{2}|\partial_n\varphi_\eps|^2+\frac{1}{\eps}W(\varphi_\eps)- \sqrt{2W(\varphi_\eps)} \partial_n\varphi_\eps\right) \tilde{\eta} \nabla\cdot \vec{n}\,dx\label{eq_equipart2}\\
		&+\int_{\Gamma^{m_\eps}_{t,h}(\delta_\eps)} \left(H_\eps + \sqrt{2W(\varphi_\eps)} \nabla\cdot \vec{n}\right) \partial_n\varphi_\eps \tilde{\eta}\,dx\label{eq_equipart3}\\ 
		&-\int_{\Gamma^{m_\eps}_{t,h}(\delta_\eps)} \eps \nabla_\tau\varphi_\eps\cdot\nabla_\tau\tilde{\eta}\,\partial_n\varphi_\eps\,dx\label{eq_equipart4}\\
		&-\int_{\Gamma^{m_\eps}_{t,h}(\delta_\eps)} \eps\nabla_\tau\varphi_\eps \cdot (\nabla\vec{n}^\top\nabla_\tau\varphi_\eps)\tilde{\eta}\,dx\label{eq_equipart5}\\
		&+\int_{\Gamma^{m_\eps}_{t,h}(\delta_\eps)} \frac{\eps}{2}|\nabla_\tau\varphi_\eps|^2 \left(\tilde{\eta} \nabla\cdot\vec{n} + \partial_n\tilde{\eta}\right)\,dx.\label{eq_equipart6}
	\end{align}	\end{subequations}
\end{lemma}

We obtain the following important estimate:
\begin{lemma}\label{th_Erel_lastterm}
	For a.e.~$t\in[0,T_0]$ let $h=h_\eps(.,t):\Gamma^{m_\eps}_t\rightarrow(-\delta,\delta)$ be a $1$-Lipschitz-function. Moreover, let $\delta_\eps\in(0,\frac{\delta}{2}]$ for $\eps>0$ small and let $\Gamma^{m_\eps}_{t,h}(\delta_\eps)$ be defined as in \eqref{eq_tub_neighb_shifted}. Then for a.e.~$t\in[0,T_0]$ it holds\begin{subequations}
	\begin{align}
		&\left|\int_{\Omega} \xi\otimes\xi:\nabla B\left(\eps|\nabla\varphi_\eps|^2 - |\nabla\psi_\eps|\right)|_{(x,t)} \,dx\right|\label{eq_estimate_lastterm1}\\
		\leq &C\int_{\Gamma^{m_\eps}_t(2\delta)\setminus \Gamma^{m_\eps}_{t,h}(\frac{\delta_\eps}{2})} \left(\eps|\nabla\varphi_\eps|^2+\frac{1}{\eps}W(\varphi_\eps)\right)|_{(x,t)}\,dx\label{eq_estimate_lastterm2}\\
		&+C\int_{\Gamma^{m_\eps}_{t,h}(\delta_\eps)} \frac{\eps}{m_\eps} |\partial_n\varphi_\eps|^2 |d_{\Gamma^{m_\eps}}-h(P_{\Gamma^{m_\eps}})|^2|_{(x,t)}\,dx\label{eq_estimate_lastterm3}\\
		&+\frac{m_\eps}{4\eps}\int_{\Gamma^{m_\eps}_{t,h}(\delta_\eps)} \left|H_\eps + \sqrt{2W(\varphi_\eps)} \nabla\cdot\xi \right|^2|_{(x,t)} \,dx\label{eq_estimate_lastterm4}\\
		&+C\int_{\Gamma^{m_\eps}_{t,h}(\delta_\eps)}(|h|_{P_{\Gamma^{m_\eps}}}|^2+|\nabla_\tau (h|_{P_{\Gamma^{m_\eps}}})|^2)\left(\eps|\nabla\varphi_\eps|^2+\frac{1}{\eps}W(\varphi_\eps)\right)|_{(x,t)}\,dx\label{eq_estimate_lastterm5}\\
		&+C E[\varphi_\eps| \Gamma^{m_\eps}](t),\label{eq_estimate_lastterm6}
	\end{align}\end{subequations}
	where $C>0$ is independent of $\eps$ and $t$.
\end{lemma}

\begin{proof}
First of all, it suffices to show the estimate with \eqref{eq_estimate_lastterm1} replaced by
\begin{align}\label{eq_lastterm_proof1}
\left|\int_{\Gamma^{m_\eps}_t(2\delta)} \xi\otimes\xi:\nabla B \left(\frac{\eps}{2}|\partial_n\varphi_\eps|^2-\frac{1}{\eps}W(\varphi_\eps)\right)|_{(x,t)}\,dx\right|
\end{align}
because of the cutoff for $\xi$ in \eqref{eq_xi_B_def}, the identity
\[
\eps|\nabla\varphi_\eps|^2-|\nabla\psi_\eps| = \frac{\eps}{2}|\nabla\varphi_\eps|^2-\frac{1}{\eps}W(\varphi_\eps) + \frac{1}{2}\left(\sqrt{\eps}|\nabla\varphi_\eps|-\frac{1}{\sqrt{\eps}}\sqrt{2W(\varphi_\eps)}\right)^2,
\]
the coercivity estimate \eqref{eq_Erel_coercivity_equipart} and finally the control of the tangential gradient in \eqref{eq_coercivity_n_tau} in terms of the relative entropy. In order to rewrite the expression by Lemma \ref{th_equipart_rewrite}, we use the following idea: for a.e.~$t\in[0,T_0]$ we write $\xi\otimes\xi:\nabla B(.,t)$ as $\partial_n\eta$, where $\eta=\eta(.,t)$ is defined as
\[
\eta(x):=\int_{h(P_{\Gamma^{m_\eps}}(x,t))}^{d_{\Gamma^{m_\eps}}(x,t)} \xi\otimes\xi:\nabla B (P_{\Gamma^{m_\eps}}(x,t) + r\vec{n}(P_{\Gamma^{m_\eps}}(x,t),t),t)\,dr\quad\text{ for }x\in \Gamma^{m_\eps}_t(2\delta).
\]
Moreover, to avoid boundary terms on $\partial \Gamma^{m_\eps}_t(\delta_\eps)$ we introduce a smooth $\alpha:\R\rightarrow[0,1]$ with $\alpha\equiv1$ on $[-\frac{1}{2},\frac{1}{2}]$ and $\alpha\equiv0$ on $\R\setminus[-\frac{3}{4},\frac{3}{4}]$ and set for a.e.~$t\in[0,T_0]$
\[
\tilde{\eta}=\tilde{\eta}(.,t):=(\tilde{\alpha}\eta)(.,t),\quad \tilde{\alpha}(x,t):=\alpha\left(\frac{d_{\Gamma^{m_\eps}}(x,t)-h(P_{\Gamma^{m_\eps}}(x,t))}{\delta_\eps}\right),\quad\text{ for }x\in \Gamma^{m_\eps}_t(2\delta).
\]
Note that Theorem \ref{th:approx_lim_uniform} and the regularity and uniform bounds in Theorem \ref{th_xi_B} yield that $\tilde{\alpha}(.,t)$, $\tilde{\eta}(.,t)$ and $\eta(.,t)$ are in $C^{0,1}(\overline{\Omega})$. We rewrite as follows:
\begin{align}\label{eq_lastterm_proof2}
\xi\otimes\xi:\nabla B(.,t)=\partial_n\eta=(1-\tilde{\alpha})\partial_n\eta - \partial_n\tilde{\alpha}\,\eta +\partial_n\tilde{\eta}.
\end{align}
Because of the definition of $\tilde{\alpha}$ it holds $1-\tilde{\alpha}=0$ on $\Gamma^{m_\eps}_{t,h}(\frac{\delta_\eps}{2})$. Moreover, 
\[
\partial_n\tilde{\alpha}=\frac{1}{\delta_\eps}\alpha'\left(\frac{d_{\Gamma^{m_\eps}}(x,t)-h(P_{\Gamma^{m_\eps}}(x,t))}{\delta_\eps}\right)=0\quad\text{ on }\Gamma^{m_\eps}_{t,h}(\tfrac{\delta_\eps}{2})\cup \left[\Gamma^{m_\eps}_t(2\delta)\setminus \Gamma^{m_\eps}_{t,h}(\tfrac{3}{4}\delta_\eps)\right]
\]
and $|\eta|\leq C\delta_\eps$ on $\Gamma^{m_\eps}_{t,h}(\delta_\eps)$ for a.e.~$t\in[0,T_0]$. Therefore the first two terms on the right hand side of \eqref{eq_lastterm_proof2} yield contributions in \eqref{eq_lastterm_proof1} that can be estimated by \eqref{eq_estimate_lastterm2}. Hence we can replace $\xi\otimes\xi:\nabla B(.,t)$ by $\partial_n\tilde{\eta}$ in \eqref{eq_lastterm_proof1} for a.e.~$t\in[0,T_0]$.

For the remaining term
\begin{align}
	\left|\int_{\Gamma^{m_\eps}_{t,h}(\delta_\eps)} \partial_n\tilde{\eta} \left(\frac{\eps}{2}|\partial_n\varphi_\eps|^2-\frac{1}{\eps}W(\varphi_\eps)\right)|_{(x,t)}\,dx\right|
\end{align}
one can apply integration by parts on $\Gamma^{m_\eps}_{t,h}(\delta_\eps)$ for a.e.~$t\in[0,T_0]$ and rewrite in a suitable way, cf.~Lemma \ref{th_equipart_rewrite}. Let us estimate the corresponding terms \eqref{eq_equipart2}-\eqref{eq_equipart6} from Lemma \ref{th_equipart_rewrite} now. Note that \eqref{eq_equipart2} and \eqref{eq_equipart5}-\eqref{eq_equipart6} are directly controlled by \eqref{eq_estimate_lastterm6} because of the coercivity properties \eqref{eq_coercivity_n_tau} of the relative entropy. 

Moreover, \eqref{eq_equipart3} with $\nabla\cdot\xi$ instead of $\nabla\cdot\vec{n}$ can be estimated by \eqref{eq_estimate_lastterm3} and \eqref{eq_estimate_lastterm4} due to Young's inequality and $|\tilde{\eta}|(.,t)\leq C|d_{\Gamma^{m_\eps}}-h(P_{\Gamma^{m_\eps}})|(.,t)$. Because of the definition of $\xi$ in \eqref{eq_xi_B_def} it holds
\[
\nabla\cdot\vec{n}=(1-\eta_{m_\eps}) \nabla\cdot\vec{n} - \frac{1}{\delta}\bar{\eta}'\left(\frac{d_{\Gamma^{m_\eps}}}{\delta}\right) + \nabla\cdot\xi,
\]
where $|1-\eta_{m_\eps}|\leq C|d_{\Gamma^{m_\eps}}|^2$ and $|\bar{\eta}'\left(\frac{d_{\Gamma^{m_\eps}}}{\delta}\right)|\leq C|d_{\Gamma^{m_\eps}}|$. Hence from exchanging $\nabla\cdot\vec{n}$ with $\nabla\cdot\xi$ we obtain the following error term:
\[
\int_{\Gamma^{m_\eps}_{t,h}(\delta_\eps)}\sqrt{W(\varphi_\eps)} |\partial_n\varphi_\eps|\,(|d_{\Gamma^{m_\eps}}|^2+|h(P_{\Gamma^{m_\eps}})|^2)|_{(x,t)}\,dx,
\]
which is controlled by \eqref{eq_estimate_lastterm5} and \eqref{eq_estimate_lastterm6} due to the bound $\sqrt{W(\varphi_\eps)} |\partial_n\varphi_\eps|\leq|\nabla\psi_\eps|$ and \eqref{eq_Erel_coercivity_1}. Hence \eqref{eq_equipart3} is estimated.

Finally, it remains to estimate \eqref{eq_equipart4}. It holds $\nabla_\tau\tilde{\eta}=\nabla_\tau\,\tilde{\alpha}\eta+\tilde{\alpha}\nabla_\tau\eta$ due to $\tilde{\eta}=\tilde{\alpha}\eta$, where we have because of \eqref{eq_nablatau_trafo}
\begin{align*}
	\nabla_\tau\tilde{\alpha}=&-\frac{1}{\delta_\eps} \alpha'\left(\frac{d_{\Gamma^{m_\eps}}(.,t)-h(P_{\Gamma^{m_\eps}}(.,t))}{\delta_\eps}\right) \nabla_\tau(h(P_{\Gamma^{m_\eps}}))|_{(.,t)},\\
	\nabla_\tau\eta =& -\partial_n\eta|_{(P_{\Gamma^{m_\eps}}(.,t)+h(P_{\Gamma^{m_\eps}}(.,t))\vec{n}(P_{\Gamma^{m_\eps}}(.,t),t)} \nabla_\tau(h(P_{\Gamma^{m_\eps}}))|_{(.,t)} \\
	&+ \int_{h(P_{\Gamma^{m_\eps}}(.,t))}^{d_{\Gamma^{m_\eps}}(.,t)} \nabla_\tau\left[\xi\otimes\xi:\nabla B(P_{\Gamma^{m_\eps}}+r\vec{n}(P_{\Gamma^{m_\eps}}),t)\right]|_{(.,t)}\,dr.
\end{align*}
Due to Lemma \ref{th_xi_B} we obtain
\[
|\nabla_\tau\tilde{\eta}| \leq C\left(|d_{\Gamma^{m_\eps}}|+|h|_{P_{\Gamma^{m_\eps}}}|+|\nabla_\tau(h|_{P_{\Gamma^{m_\eps}}})|\right)|_{(.,t)}.
\]
Altogether, \eqref{eq_equipart4} is controlled by \eqref{eq_estimate_lastterm5} and \eqref{eq_estimate_lastterm6} due to Young's inequality, \eqref{eq_coercivity_n_tau} and \eqref{eq_Erel_coercivity_1}. This shows Lemma \ref{th_Erel_lastterm}.
\end{proof}

\begin{proof}[Proof of Lemma \ref{th_equipart_rewrite}.]
	First, we use $\partial_n\tilde{\eta}=\vec{n}\cdot\nabla\tilde{\eta}=\sum_{i=1}^d n_i\partial_{x_i}\tilde{\eta}$ and integration by parts on $\Gamma^{m_\eps}_{t,h}(\delta_\eps)$ with $\tilde{\eta}\in C^{0,1}_c(\Gamma^{m_\eps}_{t,h}(\delta_\eps))$ for a.e.~$t\in[0,T_0]$. This yields
	\begin{align*}
	&\int_{\Gamma^{m_\eps}_{t,h}(\delta_\eps)} \left(\frac{\eps}{2}|\partial_n\varphi_\eps|^2-\frac{1}{\eps}W(\varphi_\eps)\right) \partial_n\tilde{\eta}\,dx \\
	&=-\int_{\Gamma^{m_\eps}_{t,h}(\delta_\eps)} \left(\frac{\eps}{2}|\partial_n\varphi_\eps|^2-\frac{1}{\eps}W(\varphi_\eps)\right) (\nabla\cdot \vec{n}) \tilde{\eta}\,dx
	-\int_{\Gamma^{m_\eps}_{t,h}(\delta_\eps)} \partial_n \left(\frac{\eps}{2}|\partial_n\varphi_\eps|^2-\frac{1}{\eps}W(\varphi_\eps)\right) \tilde{\eta}\,dx
	\end{align*} 
	for a.e.~$t\in[0,T_0]$. Note that $\partial_n^2\varphi_\eps=((\vec{n}\cdot\nabla)\vec{n})\cdot\nabla{\varphi_\eps} + \vec{n}\otimes\vec{n}:D^2\varphi_\eps=\vec{n}\otimes\vec{n}:D^2\varphi_\eps$. Therefore 
	\begin{align*}
	&-\int_{\Gamma^{m_\eps}_{t,h}(\delta_\eps)} \partial_n \left(\frac{\eps}{2}|\partial_n\varphi_\eps|^2-\frac{1}{\eps}W(\varphi_\eps)\right) \tilde{\eta}\,dx\\
	&=-\int_{\Gamma^{m_\eps}_{t,h}(\delta_\eps)} \left(\eps\Delta\varphi_\eps-\frac{1}{\eps}W'(\varphi_\eps)\right)\partial_n\varphi_\eps \tilde{\eta}\,dx
	+\int_{\Gamma^{m_\eps}_{t,h}(\delta_\eps)} \eps \left(\textup{Id}-\vec{n}\otimes\vec{n}\right):D^2\varphi_\eps \partial_n\varphi_\eps \tilde{\eta}\,dx
	\end{align*} 
	for a.e.~$t\in[0,T_0]$. In order to rewrite the last term, we compute
	\begin{align*}
	&\nabla\cdot\left((\textup{Id}-\vec{n}\otimes\vec{n})\nabla\varphi_\eps \partial_n\varphi_\eps\tilde{\eta}\right)\\
	&=-(\nabla\cdot\vec{n})\vec{n}\cdot \nabla\varphi_\eps \partial_n\varphi_\eps\tilde{\eta} + (\textup{Id}-\vec{n}\otimes\vec{n}):D^2\varphi_\eps\partial_n\varphi_\eps\tilde{\eta} +\nabla_\tau\varphi_\eps\cdot\nabla(\partial_n\varphi_\eps\tilde{\eta}),
	\end{align*}
	where we used $\nabla\cdot(\textup{Id}-\vec{n}\otimes\vec{n})=-(\nabla\cdot\vec{n})\vec{n}$. Because of 
	\[
	\nabla(\partial_n\varphi_\eps)=\nabla(\vec{n}\cdot\nabla)\varphi_\eps=(\nabla\vec{n})^\top\nabla\varphi_\eps+(\vec{n}\cdot\nabla)\nabla\varphi_\eps
	\]
	and integration by parts we obtain for a.e.~$t\in[0,T_0]$
	\begin{align*}
	&\int_{\Gamma^{m_\eps}_{t,h}(\delta_\eps)} \eps \left(\textup{Id}-\vec{n}\otimes\vec{n}\right):D^2\varphi_\eps \partial_n\varphi_\eps \tilde{\eta}\,dx\\
	&=\int_{\Gamma^{m_\eps}_{t,h}(\delta_\eps)}\eps|\partial_n\varphi_\eps|^2(\nabla\cdot\vec{n})\tilde{\eta}\,dx
	-\int_{\Gamma^{m_\eps}_{t,h}(\delta_\eps)}\eps\nabla_\tau\varphi_\eps\cdot\nabla\tilde{\eta}\,\partial_n\varphi_\eps\,dx\\
	&\quad-\int_{\Gamma^{m_\eps}_{t,h}(\delta_\eps)}\eps\nabla_\tau\varphi_\eps\cdot(\nabla\vec{n}^\top\nabla\varphi_\eps)\tilde{\eta}\,dx 
	-\int_{\Gamma^{m_\eps}_{t,h}(\delta_\eps)}\eps\nabla_\tau\varphi_\eps\cdot\left((\vec{n}\cdot\nabla)\nabla\varphi_\eps\right)\tilde{\eta}\,dx.
	\end{align*}
    Note that in \eqref{eq_equipart2}-\eqref{eq_equipart3} the $\sqrt{2W(\varphi_\eps)}$-terms cancel and were added for convenience. Moreover, one can directly prove that $\nabla\vec{n}^\top\nabla\varphi_\eps=\nabla\vec{n}^\top\nabla_\tau\varphi_\eps$. Hence it is left to show that for a.e.~$t\in[0,T_0]$
    \begin{align}\label{eq_equipart_proof}
    -\int_{\Gamma^{m_\eps}_{t,h}(\delta_\eps)}\eps\nabla_\tau\varphi_\eps\cdot\left((\vec{n}\cdot\nabla)\nabla\varphi_\eps\right)\tilde{\eta}\,dx 
    =\int_{\Gamma^{m_\eps}_{t,h}(\delta_\eps)} \frac{\eps}{2}|\nabla_\tau\varphi_\eps|^2 \left(\tilde{\eta} \nabla\cdot\vec{n} + \partial_n\tilde{\eta}\right) dx.
    \end{align} 
	To this end we use $\vec{n}\cdot\nabla=\sum_{i=1}^d n_i\partial_{x_i}$ and integration by parts for $\partial_{x_i}$. This yields 
	\begin{align*}
	-\int_{\Gamma^{m_\eps}_{t,h}(\delta_\eps)}\eps\nabla_\tau\varphi_\eps\cdot\left((\vec{n}\cdot\nabla)\nabla\varphi_\eps\right)\tilde{\eta}\,dx
	=-\frac{1}{2}\int_{\Gamma^{m_\eps}_{t,h}(\delta_\eps)} \eps (\vec{n}\cdot\nabla)|\nabla_\tau\varphi_\eps|^2 \tilde{\eta}\,dx
	=\int_{\Gamma^{m_\eps}_{t,h}(\delta_\eps)} \frac{\eps}{2}|\nabla_\tau\varphi_\eps|^2 \nabla\cdot(\tilde{\eta}\vec{n})\,dx,
    \end{align*}
	where we used $(\vec{n}\cdot\nabla)(\textup{Id}-\vec{n}\otimes\vec{n})=0$ in the second step. Therefore \eqref{eq_equipart_proof} holds and this shows Lemma \ref{th_equipart_rewrite}.
\end{proof}

\subsection{Parametrization of the Majority of the Interface}\label{sec_level_set_h}

In order for our estimate on the problematic term \eqref{eq_Erel_inequ_problem3} provided by Lemma~\ref{th_Erel_lastterm} to work, we need to represent a suitable level set of the phase-field $\varphi_\eps$ as approximately a graph of a small function $h_\eps$ over the surface $\Gamma^{m_\eps}_t$, as only then the term \eqref{eq_estimate_lastterm3} becomes controllable in terms of a constant $\eps^2/m_\eps$ and the relative entropy (see Corollary~\ref{th_energy_out2} below for details).

In this section we construct for a.e.~$t\in[0,T_0]$ a $1$-Lipschitz-function $h=h_\eps(.,t):\Gamma^{m_\eps}_t\rightarrow(-\delta,\delta)$ that is used to approximate a suitable level set of $\varphi_\eps(.,t)$. To this end, we need the following proposition about local interface errors of a BV-set compared to the strong interface $\Gamma_t^{m_\eps}$ from Section \ref{sec_Energy_def_coercivity}.
\begin{proposition}\label{th_PropositionFH}
	Let $T_0>0$, $\delta>0$ and $\Gamma^{m_\eps}$, $\Omega^{m_\eps,\pm}$, $\ve^{m_\eps}$ for $m_\eps>0$ small be as in Section \ref{sec_Energy_def_coercivity}. Moreover, let $\vec{n}$ be the extension of the normal from \eqref{eq_def_n_H} and $\xi$ be defined as in \eqref{eq_xi_B_def}. Let $t\in[0,T_0]$ be fixed and $\tilde{\chi}\in \BV(\mathbb{R}^d;\{0,1\})$ be arbitrary. We set $\chi:=\chi_{\Omega^{m_\eps,+}_t}$.
	
	Let $\theta:[0,\infty)\rightarrow [0,1]$ be a smooth cutoff with $\theta\equiv 0$ outside of $[0,\frac{1}{2}]$ and $\theta\equiv 1$ in $[0,\frac{1}{4}]$. We define the local height of the one-sided interface errors $h^\pm_t=h^\pm_{t,\eps}:\Gamma^{m_\eps}_t\rightarrow [0,\frac{\delta}{2}]$ in $\pm\vec{n}$-direction as
	\[
	h^\pm_t(s):=
	\int_0^\infty (\chi-\tilde{\chi})(s+y\vec{n}(s,t)) \, \theta\Big(\frac{y}{\delta}\Big) \,\dy\quad\text{ for a.e. }s\in\Gamma^{m_\eps}_t.
	\]
	Then $h^\pm_t$ are $BV$-functions and we denote the distributional tangential derivative by $D^\textup{tan}h^\pm_t$, by $\nabla^\textup{tan}h^\pm_t$ the density of the absolutely continuous part of $D^\textup{tan}h^\pm_t$ with respect to $\Hc^{d-1}$ and by $D^sh^\pm_t$ the singular part. Finally, let $\tilde{G}_t:=\{s+(h^+_t-h^-_t)(s)\vec{n}(s,t) : s\in \Gamma^{m_\eps}_t\}$ denote the graph of $h^+_t-h^-_t$.
	
	Then we have the following estimates with constants independent of $\tilde{\chi}$ and $t\in[0,T_0]$:
	\begin{align}\label{eq_HeightFctEstimate}
		\int_{\Gamma^{m_\eps}_t} |h^\pm_t|^2\,d\Hc^{d-1}
		\leq
		C \int_\Rd |\chi-\tilde{\chi}| \min\left\{d_{\Gamma^{m_\eps}_t},1\right\} \,dx
	\end{align}
	as well as
	\begin{align}
		\label{eq_HeightFunctionGradientEstimate}
		&\int_{\Gamma^{m_\eps}_t} \min\{|\nabla^{\tan} h^\pm_t|^2,|\nabla^{\tan} h^\pm_t|\} \,d\Hc^{d-1}  + |D^s h^\pm_t|(\Gamma^{m_\eps}_t)\\
		&\leq
		C \int_\Rd {\Big(}1-\xi(.,t)\cdot \frac{\nabla\tilde{\chi}}{|\nabla\tilde{\chi}|}{\Big)} 
		\,\mathrm{d}|\nabla\tilde{\chi}|+ C \int_\Rd |\chi-\tilde{\chi}| \min\left\{d_{\Gamma^{m_\eps}_t},1\right\} \,dx,\notag
	\end{align}
	and
	\begin{align}
		\label{eq_ErrorGraph}
		\int_{\Rd\setminus \tilde{G}_t} 1 \,\mathrm{d}|\nabla\tilde{\chi}|
		\leq C \int_{\Rd} {\Big(}1-\xi(.,t)\cdot \frac{\nabla\tilde{\chi}}{|\nabla\tilde{\chi}|}{\Big)} \,\mathrm{d}|\nabla\tilde{\chi}|
		+ C\int_\Rd |\chi-\tilde{\chi}| \min\left\{d_{\Gamma^{m_\eps}_t},1\right\} \,dx.
	\end{align}
\end{proposition}
\begin{proof}
	The assertions up to \eqref{eq_ErrorGraph} can be shown analogously to \cite[Proposition~26, a)-b)]{FischerHensel}. Moreover, \eqref{eq_ErrorGraph} may be established readily using the arguments for statement c) of \cite[Proposition~26]{FischerHensel}.
\end{proof}

The plan is to apply the previous Proposition \ref{th_PropositionFH} to a suitable level set of $\varphi_\eps(.,t)$ for a.e.~$t\in[0,T_0]$. The latter is selected in the following lemma.
\begin{lemma}\label{th_levelset}	
	Let $t\in[0,T_0]$ be such that $\varphi_\eps(.,t)\in H^1(\Omega)$. Then the relative entropy $E[\varphi_\eps|\Gamma^{m_\eps}](t)$ from \eqref{eq_Erel2} is well-defined and finite. Moreover, there exists a level $b(t)=b_\eps(t)\in (-\frac{1}{2},\frac{1}{2})$ such that the corresponding super-level set $S_{b(t)}:=\{x\in\Omega:\varphi_\eps(x,t)>b(t)\}$ is a set of finite perimeter (possibly empty) and satisfies with $\xi$ as in \eqref{eq_xi_B_def} and $\psi$ as in \eqref{eq_psi_def} the estimate
	\[
	\int_{\Rd} {\Big(}1-\xi(.,t) \cdot \frac{\nabla \chi_{S_{b(t)}}}{|\nabla \chi_{S_{b(t)}}|}{\Big)} \,\mathrm{d}|\nabla \chi_{S_{b(t)}}| \leq \frac{2}{\psi(\frac12)-\psi(-\frac12)} E[\varphi_\eps|\Gamma^{m_\eps}](t).
	\]
\end{lemma}
\begin{proof}
	Let $t$ be as in the lemma. Then $E[\varphi_\eps|\Gamma^{m_\eps}](t)$ is well-defined and finite because of $\varphi_\eps(.,t)\in H^1(\Omega)$. Moreover, $\varphi_\eps(.,t)\in BV(\Omega)$ yields together with the coarea-formula for BV-functions, cf.~Ambrosio, Fusco, Pallara \cite[Theorem 3.40]{Ambrosio2000a} that $S_b:=\{x\in\Omega:\varphi_\eps(x,t)>b\}$ is a set of finite perimeter (possibly empty) for a.e.~$b\in\R$. Moreover, we use the coarea-formula for BV-functions \cite[Theorem 3.40]{Ambrosio2000a} applied to $\psi_\eps(.,t)=\psi(\varphi_\eps(.,t))\in H^1(\Omega)$ {(together with an approximation argument by simple functions)} and obtain
	\begin{align*}
		E[\varphi_\eps|\Gamma^{m_\eps}](t)&\geq \int_{\Omega} \left[|\nabla \psi_\eps| - \xi\cdot \nabla \psi_\eps\right]\!(.,t)\,dx
		=\int_{\R^d} \left[1 - \xi \cdot \frac{\nabla \varphi_\eps}{|\nabla \varphi_\eps|}\right]\!(.,t)\, |\nabla\psi_\eps(.,t)|\,dx\\
		&=\int_{0}^\sigma \int_{\Rd} \left[1 - \xi\cdot \frac{\nabla \varphi_\eps}{|\nabla \varphi_\eps|}\right]\!(.,t) \,\mathrm{d}|\nabla \chi_{S_{\psi^{-1}(r)}}| \,dr.
	\end{align*}
	This shows the claim by a contradiction argument.
\end{proof}

We combine Proposition \ref{th_PropositionFH} and Lemma \ref{th_levelset} in the following lemma.

\begin{lemma}\label{th_lem_height}
	Let $t\in[0,T_0]$ be fixed such that $\varphi_\eps(.,t)\in H^1(\Omega)$. Let the level $b(t)=b_\eps(t)\in (-\frac{1}{2},\frac{1}{2})$ and the super-level set $S_{b(t)}$ be as in Lemma \ref{th_levelset}. We set $\chi:=\chi_{\Omega^{m_\eps,+}_t}$. Then there is a $1$-Lipschitz function $h_t=h_{t,\eps}:\Gamma^{m_\eps}_t \rightarrow (-\delta,\delta)$ subject to the estimate
	\begin{align}
		\label{eq_HeightFunctionDirichletEstimate}
		&\int_{\Gamma^{m_\eps}_t} |h_t|^2 + |\nabla^{\tan} h_t|^2 \,d\Hc^{d-1}
		\leq C E[\varphi_\eps|\Gamma^{m_\eps}](t)
		+C \int_\Rd |\chi-\chi_{S_{b(t)}}| \min\left\{d_{\Gamma^{m_\eps}_t},1\right\} \,dx
	\end{align}
	such that the graph $G_t:=\{s+h_t(s)\vec{n}(s,t):s\in \Gamma^{m_\eps}_t\}$ approximates the reduced boundary of $S_{b(t)}$ in the following sense:
	\begin{align}
		\label{eq_ErrorHeightFunction}
		\int_{\Rd\setminus G_t} 1 \,\mathrm{d}|\nabla\chi_{S_{b(t)}}|
		\leq C E[\varphi_\eps|\Gamma^{m_\eps}](t)
		+C \int_\Rd |\chi-\chi_{S_{b(t)}}| \min\left\{d_{\Gamma^{m_\eps}_t},1\right\} \,dx.
	\end{align}
	Finally, constants in the estimates \eqref{eq_HeightFunctionDirichletEstimate}-\eqref{eq_ErrorHeightFunction} are independent of $\varphi_\eps$ and $t$.
\end{lemma}
\begin{proof}
	By Lemma \ref{th_levelset} it holds $\tilde{\chi}:=\chi_{S_{b(t)}}\in BV(\R^d;\{0,1\})$.
	Hence by Proposition \ref{th_PropositionFH} applied for $\tilde{\chi}$ there exist BV-functions $h_t^\pm:\Gamma^{m_\eps}_t\rightarrow[0,\frac{\delta}{2}]$ such that \eqref{eq_HeightFctEstimate}-\eqref{eq_ErrorGraph} hold. We set $\tilde{h}_t:=h^+_t-h^-_t$. It remains to modify $\tilde h$ to obtain a 1-Lipschitz function. To do so, we use a standard Lipschitz truncation strategy (see e.g.\ \cite{AcerbiFusco}, \cite[Section~6.6.3]{EvansGariepy} or \cite{BreitDieningGmeinederLipschitzBV}): we consider the maximal operator $\Mc$ over $\Gamma^{m_\eps}_t$, which can be defined as in \cite[Chapter I, \S 8.1 (ii)]{SteinHarmonicAnalysis} since $\mathcal{H}^{d-1}\lfloor \Gamma_t^{m_\eps}$ satisfies the doubling condition with respect to the geodesic balls on $\Gamma_t^{m_\eps}$.
	Then there is some $C_1>0$ such that
	\[
	|u(s_1)-u(s_2)|\leq C_1|s_1-s_2| \left(\Mc|D^\textup{tan}u|(s_1)+\Mc|D^\textup{tan}u|(s_2)\right)\quad\text{for }\Hc^{d-1}\text{-a.e. }s_1,s_2\in\Gamma^{m_\eps}_t
	\]
	for all $u\in BV(\Gamma^{m_\eps}_t)$ with some $C_1>0$ independent of $u$ and small $m_\eps$. More precisely, in the case that $\Gamma^{m_\eps}_t$ is replaced by $\R^{d-1}$ this inequality is shown in \cite[Lemma~2(c)]{BreitDieningGmeinederLipschitzBV}. Then the estimate in the present case can be shown by localization. Moreover, because of the continuous dependence on $t\in [0,T]$ and $m_\eps\in [0, m_0]$ (cf.\ Theorem~\ref{th_approx_m_dependence}), the constant can be chosen uniformly in $t\in[0,T_0]$, $m_\eps\in [0,m_0]$ for $\eps>0$ sufficiently small. Using the precise representative for $\tilde{h}_t$ we define
	\begin{align*}
		h_t(s):=\tilde{h}_t(s)\quad \text{ for all }\quad s\in A_t:=\{s\in\Gamma^{m_\eps}_t:(\mathcal{M}|D^\textup{tan} \tilde h_t|)(s)\leq\overline{c}\},
	\end{align*}
	where $\overline{c}>0$ is a small constant to be determined. The so-defined function $h$ is Lipschitz on $A_t$ with Lipschitz-constant bounded by $c_1\overline{c}$. We extend $h$ to all of $\Gamma^{m_\eps}_t$ as a Lipschitz function via the standard extension, cf.~\cite[Proposition 2.12]{Ambrosio2000a}), i.e.~
	\[
	h_t(s):=\inf_{\tilde{s}\in A_t}\{\tilde{h}_t(\tilde{s})+\textup{Lip}(\tilde{h}_t)|s-\tilde{s}|\},
	\]
	where the Lipschitz constant stays the same. Hence for $\overline{c}$ small enough (independent of $\tilde{h}_t$, $h_t$, $t$ and small $m_\eps$), we obtain that $h_t$ is $1$-Lipschitz and bounded by $\frac{3\delta}{4}$. Moreover, {using the weak $L^1$-estimate for the maximal operator we obtain} 
	\begin{align}\begin{split}\label{eq_lem_height}
		\mathcal{H}^{d-1}\left(\Gamma^{m_\eps}_t\setminus A_t\right)
		&\leq C \int_{\Gamma^{m_\eps}_t} |\nabla^\textup{tan} \tilde{h}_t| \,d\Hc^{d-1} + |D^s\tilde{h}_t|(\Gamma^{m_\eps}_t)\\
		&\leq CE[\varphi_\eps|\Gamma^{m_\eps}](t)
		+C\int_\Rd |\chi-\tilde{\chi}| \min\left\{d_{\Gamma^{m_\eps}_t},1\right\} \,dx,\end{split}
	\end{align}
	where we used \eqref{eq_HeightFunctionGradientEstimate} and Lemma \ref{th_levelset} in the last step. Because of $|\nabla^\textup{tan}h_t|\leq C$ as well as by \eqref{eq_HeightFunctionGradientEstimate} and the fact that $D^\textup{tan}\tilde{h}_t=\nabla^\textup{tan}\tilde{h}_t=\nabla^\textup{tan}h_t$ a.e.~on $A_t$, this establishes the bound
	\begin{align*}
		&\int_{\Gamma^{m_\eps}_t} |\nabla^{\tan} h_t|^2 \,d\Hc^{d-1}
		\leq C E[\varphi_\eps|\Gamma^{m_\eps}](t)
		+C\int_\Rd |\chi-\tilde{\chi}| \min\left\{d_{\Gamma^{m_\eps}_t},1\right\} \,dx.
	\end{align*}
	Hence the $|h_t|^2$-part of the bound \eqref{eq_HeightFunctionDirichletEstimate} is left to prove. The latter follows because of \eqref{eq_lem_height}, the boundedness of $h_t$ and \eqref{eq_HeightFunctionGradientEstimate}. Finally, \eqref{eq_ErrorHeightFunction} follows from \eqref{eq_ErrorGraph} in Proposition \ref{th_PropositionFH}, the estimate \eqref{eq_lem_height} and Lemma \ref{th_levelset}.
\end{proof}

\begin{corollary}\label{th_h_energyweight}
Let $t\in[0,T_0]$ be fixed such that $\varphi_\eps(.,t)\in H^1(\Omega)$, let $b(t)=b_\eps(t), S_{b(t)}$ be as in Lemma \ref{th_levelset} and let $h_t=h_{t,\eps}:\Gamma^{m_\eps}_t\rightarrow(-\delta,\delta)$ be as in Lemma \ref{th_lem_height}. Then with $\chi:=\chi_{\Omega^{m_\eps,+}_t}$ it holds 
\begin{align}\begin{split}\label{eq_height_energyweight}
&\int_{\Gamma^{m_\eps}_t(2\delta)}(|h_t|_{P_{\Gamma^{m_\eps}}}|^2+|\nabla_\tau (h_t|_{P_{\Gamma^{m_\eps}}})|^2)\left(\eps|\nabla\varphi_\eps|^2+\frac{1}{\eps}W(\varphi_\eps)\right)|_{(x,t)}\,dx\\
&\leq C E[\varphi_\eps|\Gamma^{m_\eps}](t)
+C \int_\Rd |\chi-\chi_{S_{b(t)}}|\min\left\{d_{\Gamma^{m_\eps}_t},1\right\} \,dx.
\end{split}\end{align}
\end{corollary}
\begin{proof}
	Using of the identity \eqref{eq_nablatau_trafo} for $\nabla_\tau$, we can exchange $\nabla_\tau (h_t|_{P_{\Gamma^{m_\eps}}})$ by $(\nabla^\textup{tan}h_t)|_{P_{\Gamma^{m_\eps}}}$ in the above estimate. Moreover, since both $h_t$ and $\nabla^\textup{tan}h_t$ are uniformly bounded, we can consider the estimate with $\partial_n\varphi_\eps$ instead of $\nabla\varphi_\eps$, since tangential derivatives are controlled by \eqref{eq_coercivity_n_tau}. Then we apply an integral transformation with the tubular neighbourhood coordinates to obtain
	\begin{align*}
	&\int_{\Gamma^{m_\eps}_t(2\delta)}(|h_t|_{P_{\Gamma^{m_\eps}}}|^2+|(\nabla^\textup{tan}h_t)|_{P_{\Gamma^{m_\eps}}}|^2)\left(\eps|\partial_n\varphi_\eps|^2+\frac{1}{\eps}W(\varphi_\eps)\right)|_{(x,t)}\,dx\\
	&=\int_{\Gamma^{m_\eps}_t}(|h_t|^2+|\nabla^\textup{tan}h_t|^2)(s)\int_{-2\delta}^{2\delta}\left(\eps|\partial_r(\varphi_\eps|_{{X_{m_\eps}}})|^2+\frac{1}{\eps}W(\varphi_\eps|_{{X_{m_\eps}}})\right)|_{(r,s,t)}J_t(r,s)\,dr\,ds,
	\end{align*}
	where the factor $J_t=J_t(m_\eps)$ satisfies for some $c_J,C_J>0$ independent of $t$ and $m_\eps$
	\begin{equation*}
		c_J\leq J_t(r,s)\leq C_J \qquad \text{for all }s\in \Gamma^{m_\eps}_t,r\in (-\delta,\delta).
	\end{equation*}
 	Let $\sigma=\psi(1)$ be as in \eqref{eq_psi_def}. For $s\in \Gamma^{m_\eps}_t$ such that the inner integral {with respect to $r\in (-2\delta,2\delta)$} is less or equal $4\sigma$, the desired estimate follows from Lemma \ref{th_lem_height}. For $s\in \Gamma^{m_\eps}_t$ for which this is not the case, we use that $|\psi(\varphi)|\leq \sigma$ for all $\varphi\in[-1,1]$ and therefore 
 	\[
 	\left|\int_{-2\delta}^{2\delta}\partial_r\left(\psi(\varphi_\eps|_{{X_{m_\eps}}(r,s,t)})\right)\,dr\right|\leq 2\sigma.
 	\]
 	Hence it follows for such $s\in \Gamma^{m_\eps}_t$ that
 	\begin{align*}
 	\int_{-2\delta}^{2\delta}&\left(\eps|\partial_r(\varphi_\eps|_{{X_{m_\eps}}})|^2+\frac{1}{\eps}W(\varphi_\eps|_{{X_{m_\eps}}})-\partial_r\left(\psi(\varphi_\eps|_{{X_{m_\eps}}})\right)\right)|_{(r,s,t)}J_t(r,s)\,dr\\
 	&\geq c_J\int_{-2\delta}^{2\delta}\left(\eps|\partial_r(\varphi_\eps|_{{X_{m_\eps}}})|^2+\frac{1}{\eps}W(\varphi_\eps|_{{X_{m_\eps}}})-\partial_r\left(\psi(\varphi_\eps|_{{X_{m_\eps}}})\right)\right)|_{(r,s,t)}\,dr\\
 	&\geq \frac{c_J}{2}\int_{-2\delta}^{2\delta}\left(\eps|\partial_r(\varphi_\eps|_{{X_{m_\eps}}})|^2+\frac{1}{\eps}W(\varphi_\eps|_{{X_{m_\eps}}})\right)|_{(r,s,t)}\,dr\\
 	&\geq \frac{c_J}{2C_J}\int_{-2\delta}^{2\delta}\left(\eps|\partial_r(\varphi_\eps|_{{X_{m_\eps}}})|^2+\frac{1}{\eps}W(\varphi_\eps|_{{X_{m_\eps}}})\right)|_{(r,s,t)}J_t(r,s)\,dr.
 	\end{align*}
 	Therefore the contribution in this case can be estimated with \eqref{eq_coercivity_n_tau} and by using that $h_t$ and $\nabla^\textup{tan}h_t$ are uniformly bounded.
\end{proof}

Finally, in order to control the second term in \eqref{eq_height_energyweight} in the end, we need the following lemma:

\begin{lemma}\label{th_Ebulk_BVerror}
	Let $t\in[0,T_0]$ be fixed such that $\varphi_\eps(.,t)\in H^1(\Omega)$ and let $b(t)=b_\eps(t), S_{b(t)}$ be as in Lemma \ref{th_levelset}. Then with $\chi:=\chi_{\Omega^{m_\eps,+}_t}$ and $E_\textup{bulk}[\varphi_\eps|\Gamma^{m_\eps}]$ from Section \ref{sec_Ebulk} it follows that
	\begin{align}
		\int_\Rd |\chi-\chi_{S_{b(t)}}|\min\left\{d_{\Gamma^{m_\eps}_t},1\right\} \,dx \leq CE_\textup{bulk}[\varphi_\eps|\Gamma^{m_\eps}](t).
	\end{align}
\end{lemma}

\begin{proof}
	The integrand is zero on $\R^d\setminus\overline{\Omega}$. Moreover, note that \eqref{eq_Ebulk_coercivity1} yields
	\begin{align*}
		\int_\Omega\left|\sigma\chi -\psi_\eps(.,t)\right|\min\{d_{\Gamma^{m_\eps}_t},1\}\,dx \leq C E_\textup{bulk}[\varphi_\eps|\Gamma^{m_\eps}](t).
	\end{align*}
	Hence we obtain
	\[
	\int_{\Omega^{m_\eps,+}_t}\left|\sigma -\psi_\eps(.,t)\right|\min\{d_{\Gamma^{m_\eps}_t},1\}\,dx +\int_{\Omega^{m_\eps,-}_t}\left|\psi_\eps(.,t)\right|\min\{d_{\Gamma^{m_\eps}_t},1\}\,dx\leq C E_\textup{bulk}[\varphi_\eps|\Gamma^{m_\eps}](t).
	\]
	For a.e.~$x\in\Omega^{m_\eps,+}_t\cap S_{b(t)}$ it holds $|\chi-\chi_{S_{b(t)}}|(x)=0$. For a.e.~$x\in\Omega^{m_\eps,+}_t \setminus S_{b(t)}$ we have $|\chi-\chi_{S_{b(t)}}|(x)=1$ and $\psi_\eps(x,t)\leq\psi(b(t))\leq\psi(\frac12)$. Hence $|\sigma-\psi_\eps(x,t)|\geq |\sigma-\psi(\frac12)|>0$ since $b(t)\in(-\frac12,\frac12)$. This shows the estimate on $\Omega^{m_\eps,+}_t$. Moreover, for a.e.~$x\in\Omega^{m_\eps,-}_t\setminus S_{b(t)}$ we have $|\chi-\chi_{S_{b(t)}}|(x)=0$. Finally, for a.e.~$x\in\Omega^{m_\eps,-}_t\cap S_{b(t)}$ it holds $|\chi-\chi_{S_{b(t)}}|(x)=1$ and $\psi_\eps(x,t)>\psi(b(t))\geq\psi(-\frac12)>0$. This yields the estimate on $\Omega^{m_\eps,-}_t$.
\end{proof}

\subsection{Estimate of Energy away from Strip around Level Set}\label{sec_energy_out_strip}
In this section we estimate the remainder terms \eqref{eq_estimate_lastterm2}-\eqref{eq_estimate_lastterm3} from Lemma \ref{th_Erel_lastterm} involving the energy density for $\varphi_\eps$. Therefore we show the following lemma:

\begin{lemma}\label{th_energy_out_strip}
	Let $t\in[0,T_0]$ be fixed such that $\varphi_\eps(.,t)\in H^1(\Omega)$ and let $b(t)=b_\eps(t), S_{b(t)}$ be as in Lemma \ref{th_levelset}. We set $\chi=\chi_{\Omega^{m_\eps,+}_t}$. Moreover, let $h=h_{t,\eps}:\Gamma^{m_\eps}_t\rightarrow(-\delta,\delta)$ be as in Lemma \ref{th_lem_height}. We define the shifted tubular neighbourhood $\Gamma^{m_\eps}_{t,h}(\tilde{\delta})$ for $\tilde{\delta}\in(0,\delta]$ as in \eqref{eq_tub_neighb_shifted}. For $\kappa>0$ fixed and $\eps>0$ small we obtain with some constants $c,C>0$ independent of $\varphi_\eps$, $h$, $t$ and $\kappa$
	\begin{align*}
          &\int_{\Gamma^{m_\eps}_t(\delta) \setminus \Gamma^{m_\eps}_{t,h}(\kappa\eps)} \left(\eps \frac{|\nabla \varphi_\eps|^2}2 + \frac{ W(\varphi_\eps)}\eps\right)\, dx\\
          &\leq C\left(E[\varphi_\eps|\Gamma^{m_\eps}](t)
          +\int_\Rd |\chi-\chi_{S_{b(t)}}|\min\left\{d_{\Gamma^{m_\eps}_t},1\right\}\,dx+ e^{-c\kappa}\right).
	\end{align*}
\end{lemma}
\begin{proof}
	Let $G_t$ be the graph of $h=h_{t,\eps}$ as in Lemma \ref{th_lem_height}. Then $R_t:=P_{\Gamma^{m_\eps}_t}(G_t\cap\supp|\nabla\chi_{S_{b(t)}}|)\subseteq\Gamma^{m_\eps}_t$ satisfies
	\[
	\Hc^{d-1}(\Gamma^{m_\eps}_t\setminus R_t)\leq C\left(E[\varphi_\eps|\Gamma^{m_\eps}](t)
	+\int_\Rd |\chi-\chi_{S_{b(t)}}|\min\left\{d_{\Gamma^{m_\eps}_t},1\right\}\,dx\right),
	\]
	where we used that $\Hc^{d-1}((\Gamma^{m_\eps}_t\setminus R_t)\cap P_{\Gamma^{m_\eps}_t}(\supp|\nabla\chi_{S_{b(t)}}|))$ is controlled by \eqref{eq_ErrorHeightFunction} and that for the remaining part $\tilde{R}_t:=(\Gamma^{m_\eps}_t\setminus R_t)\setminus P_{\Gamma^{m_\eps}_t}(\supp|\nabla\chi_{S_{b(t)}}|)$ we have
	\[
	\Hc^{d-1}(\tilde{R}_t)=\int_{\tilde{R}_t}\int_{\frac{\delta}{2}}^\delta \frac{2}{\delta}\,dr\,d\Hc^{d-1}(s)
	\leq C\int_\Rd |\chi-\chi_{S_{b(t)}}|\min\left\{d_{\Gamma^{m_\eps}_t},1\right\}\,dx
	\]
	since $s\in\tilde{R}_t$ implies $|\chi-\chi_{S_{b(t)}}|(s+r\vec{n}(s,t))=1$ either for all $r\in(-\delta,-\frac{\delta}{2})$ or for all $r\in(\frac{\delta}{2},\delta)$.

	Moreover, for $s\in R_t$ it holds 
	\begin{align}\label{eq_graph_level}
		\varphi_\eps(s+h(s)\vec{n}(s,t),t)=b(t).
	\end{align}
	The strategy is to use an ODE inequality argument in normal direction using the control of the relative entropy over the equipartition error. More precisely, by \eqref{eq_coercivity_n_tau} we have
	\begin{align}\label{eq_energy_out_equipart}
	\int_{\Gamma^{m_\eps}_t} \int_{-\delta}^{\delta} \frac{1}{2}\left|\sqrt{\eps}\partial_r\tilde{\varphi}_\eps-\frac{1}{\sqrt{\eps}}\sqrt{2W(\tilde{\varphi}_\eps)}\right|^2\!(.,t) J_t(r,s)\,dr \, d\mathcal{H}^{d-1}(s)\leq E[\varphi_\eps|\Gamma^{m_\eps}](t),
	\end{align}
	where $\tilde{\varphi}_\eps:=\varphi_\eps|_{X_{m_\eps}(.,t)}$ with $X_{m_\eps}$ as in Section \ref{sec_Energy_def_coercivity} and the factor $J_t=J_t(m_\eps)$ appears due to an integral transformation and satisfies for some $c_J,C_J>0$ independent of $t$ and $m_\eps$
        \begin{equation*}
          c_J\leq J_t(r,s)\leq C_J \qquad \text{for all }s\in \Gamma^{m_\eps}_t,r\in (-\delta,\delta).
        \end{equation*}
    With $f:=\sqrt{\eps}\partial_r\tilde{\varphi}_\eps-\frac{1}{\sqrt{\eps}}\sqrt{2W(\tilde{\varphi}_\eps)}$ it holds
	\begin{equation*}
		\partial_r\tilde{\varphi}_\eps = \frac1\eps \sqrt{2W(\tilde{\varphi}_\eps)}- \frac1{\sqrt{\eps}} f,
	\end{equation*}
	where $\sqrt{W(\varphi)}\geq c(1-\varphi)(1+\varphi)\geq c'(1-\varphi)$ if $\varphi\geq -\frac34$. Hence if $\tilde{\varphi}_\eps\geq -\frac34$ we have
	\begin{align*}
		\partial_r (1-\tilde{\varphi}_\eps)^2=2(1-\tilde{\varphi}_\eps)(-\partial_r\tilde{\varphi}_\eps)\leq - \frac{2c'}\eps (1-\tilde{\varphi}_\eps)^2+ \frac2{\sqrt{\eps}} f(1-\tilde{\varphi}_\eps)\leq  - \frac{c'}\eps (1-\tilde{\varphi}_\eps)^2+ C_1|f|^2.
	\end{align*}
	We multiply the preceding inequality by $e^{-\frac{c}{\eps}(r-h(s))}$ and integrate from $h(s)$ to $r$. Then we obtain for $s\in R_t$ because of \eqref{eq_graph_level}
	\begin{align}\label{eq_graph_level2}
		(1-\tilde{\varphi}_\eps)^2 (s,r)
		\leq  e^{-\frac{c}\eps (r-h(s))}(1-b(t))^2 + \frac{C_1}{c_J}\int_{-\delta}^{\delta} |f(s,r)|^2J_t(r,s)\, dr
	\end{align}
	for all $r\in [h(s),\delta)$ provided that $\tilde{\varphi}_\eps(\tilde{r},s)\geq-\frac34$ for all $\tilde{r}\in [h(s),r)$. We define
	\[
	S_t:= \left\{s\in R_t: \frac{C_1}{c_J}\int_{-\delta}^{\delta} |f(s,r)|^2J_t(r,s)\, dr\leq \frac14\right\}.
	\]
	Then, due to $b(t)\in(-\frac12,\frac12)$ and \eqref{eq_graph_level2} we obtain $\tilde{\varphi}_\eps(r,s)\geq-\frac{3}{4}$ for all $r\in [h(s),\delta)$ and $s\in S_t$. 
	Analogously one shows $\tilde{\varphi}_\eps(r,s)\leq\frac{3}{4}$ for all $r\in (-\delta,h(s)]$ and $s\in S_t$. Moreover, note that because of \eqref{eq_energy_out_equipart} and the definition of $f$ we have
	\begin{equation*}
		\Hc^{d-1}(R_t\setminus S_t)\leq \frac{4c_J}{C_1} \int_{R_t} \int_{-\delta}^{\delta} |f(s,r)|^2J_t(r,s)\, dr \,d\mathcal{H}^{d-1}(s)\leq \frac{8c_J}{C_1} E[\varphi_\eps|\Gamma^{m_\eps}](t).
	\end{equation*}

	We use an integral transformation to obtain
	\begin{align*}
		&\int_{\Gamma^{m_\eps}_t(\delta)\setminus \Gamma^{m_\eps}_{t,h}(\kappa\eps)} \left(\frac{\eps}2|\nabla \varphi_\eps|^2 + \frac{W(\varphi_\eps)}\eps \right)\!dx\\
		&= \int_{\Gamma^{m_\eps}_t} \int_{h(s)+\kappa\eps}^\delta\left(\frac{\eps}2|\nabla \varphi_\eps|^2|_{X_{m_\eps}} + \frac{W(\tilde{\varphi}_\eps)}\eps\right)\!J_t(r,s)\,dr\, \, d\mathcal{H}^{d-1}(s)\\
		&\quad+\int_{\Gamma^{m_\eps}_t} \int_{-\delta}^{h(s)-\kappa\eps}\left(\frac{\eps}2|\nabla \varphi_\eps|^2|_{X_{m_\eps}} + \frac{W(\tilde{\varphi}_\eps)}\eps\right)\!J_t(r,s)\,dr\, \, d\mathcal{H}^{d-1}(s).
	\end{align*}
	We use the definition \eqref{eq_Erel2} to derive an estimate. First, for the contribution over $\Gamma^{m_\eps}_t\setminus S_t$ we use an integral transformation as above for the last term in \eqref{eq_Erel2}. Since $\xi$ points in normal direction by definition \eqref{eq_xi_B_def}, the inner integral is uniformly bounded because $|\psi|\leq\sigma$. Hence the estimates for $\Hc^{d-1}(\Gamma^{m_\eps}_t\setminus R_t)$ and $\Hc^{d-1}(R_t\setminus S_t)$ yield that the contribution over $\Gamma^{m_\eps}_t\setminus S_t$ is suitably estimated by the right hand side of the estimate in the Lemma. For the remaining part over $S_t$, we note that due to integration by parts and since $\xi$ points in normal direction by definition \eqref{eq_xi_B_def} we have
	\begin{align*}
		&\left|\int_{S_t} \int_{h(s)+\kappa\eps}^{\delta}\xi|_{X_{m_\eps}}\cdot \nabla (\psi_\eps|_{X_{m_\eps}}-\sigma)J_t(r,s)\,dr\, \, \mathcal{H}^{d-1}(s)\right|\\
		&\leq C\left| \int_{S_t} \int_{h(s)+\kappa\eps}^{\delta}(\psi(\tilde{\varphi}_\eps)-\sigma)\,dr\, \, d\mathcal{H}^{d-1}(s)\right| +  \left|\int_{S_t}\left.(\psi(\tilde{\varphi}_\eps)-\sigma)J_t(r,s)\right|_{r=h(s)+k\eps}^{\delta}\, \, d\Hc^{d-1}(s)\right|\\
		&\leq  C\left(e^{-c\kappa}+ \|f\|_{L^2(\Gamma^{m_\eps}_t(\delta))}^2\right)\leq C\left(e^{-c\kappa}+ E[\varphi_\eps|\Gamma^{m_\eps}](t)\right),
	\end{align*}
	where we have used
	\begin{equation*}
		|\psi(\tilde{\varphi}_\eps(r,s))-\sigma|\leq C \left(e^{-\frac{c}\eps (r-h(s))} +\|f(s,.)\|^2_{L^2_{J_t(.,s)}(-\delta,\delta)}\right)\qquad \text{for all }r\in [h(s),\delta), s\in S_t
	\end{equation*}
	due to \eqref{eq_graph_level2} and $|\psi(r)-\sigma|\leq C (1-r)^2$ for all $r\in \R$.
	Analogously one shows
	\begin{align*}
		&\left|\int_{S_t} \int^{h(s)-\kappa\eps}_{-\delta}\xi|_{X_{m_\eps}}\cdot \nabla\psi_\eps|_{X_{m_\eps}} J_t(r,s)\,dr\, \, \Hc^{d-1}(s)\right|\leq C\left(e^{-c\kappa}+ E[\varphi_\eps|\Gamma^{m_\eps}](t)\right).
	\end{align*}   
	Finally, we have $\left|\int_{\Gamma^{m_\eps}_t\setminus S_t}\int_{(-\delta,h(s)-\kappa\eps)\cup(h(s)+\kappa\eps,\delta)} \xi\cdot\nabla\psi_\eps\,dx\right|\leq C\Hc^{d-1}(\Gamma^{m_\eps}_t\setminus S_t)$. Combining these bounds, this shows the claim.
\end{proof}

	As a corollary of Lemma \ref{th_energy_out_strip}, we estimate the remaining term \eqref{eq_estimate_lastterm3}.

\begin{corollary}\label{th_energy_out2}
	Let the assumptions and notation of Lemma \ref{th_energy_out_strip} be in place and let $C_0\geq 1$. Moreover, let $C_1>0$ such that $\int_{\Omega}\eps \frac{|\nabla \varphi_\eps|^2}2 + \frac{W(\varphi_\eps)}\eps\,dx\leq C_1$. Then for all $\eps$ small with $\delta_\eps:=C_0|\log\eps|\eps\leq\delta$ and for some uniform $C>0$ we have
    \begin{align*}
    &\int_{\Gamma^{m_\eps}_{t,h}(\delta_\eps)} \frac{1}{m_\eps} \left(\frac{\eps}{2}|\nabla\varphi_\eps|^2+\frac{1}{\eps}W(\varphi_\eps)\right)\!(x,t) \left|d_{\Gamma^{m_\eps}_t}-h(P_{\Gamma^{m_\eps}_t})\right|^2\,dx\\
    &\leq CC_0^3|\log\eps|^3 \frac{\eps^2}{m_\eps}\left(E[\varphi_\eps|\Gamma^{m_\eps}](t)
    +\int_\Rd |\chi-\chi_{S_{b(t)}}|\min\left\{d_{\Gamma^{m_\eps}_t},1\right\}\,dx\right) +CC_1\frac{\eps^2}{m_\eps}.
    \end{align*}
\end{corollary}
\begin{proof}
        Let $C_0\geq1$ be fixed, then for all $\eps$ sufficiently small it holds $2C_0|\log\eps|\eps\leq\delta$. We fix $\eps$ small. Let $N\in\mathbb{N}$ be such that $C_0|\log \eps|\leq N\leq 2C_0|\log\eps|$. Then $\delta_\eps\leq N\eps\leq\delta$ and 
        \begin{align*}
          &  \int_{\Gamma^{m_\eps}_{t,h}(N\eps)}\frac1{m_\eps} \left(\eps \frac{|\nabla \varphi_\eps|^2}2 + \frac{W(\varphi_\eps)}\eps\right)\!(x,t)\left|d_{\Gamma^{m_\eps}_t}-h(P_{\Gamma^{m_\eps}_t})\right|^2\, dx\\
          &\quad \leq \sum_{n=2}^N C\int_{\Gamma^{m_\eps}_{t,h}(n\eps)\setminus \Gamma^{m_\eps}_{t,h}((n-1)\eps)}\frac1{m_\eps} \left(\eps \frac{|\nabla \varphi_\eps|^2}2 + \frac{W(\varphi_\eps)}\eps\right)\!(x,t)\left|d_{\Gamma^{m_\eps}_t}-h(P_{\Gamma^{m_\eps}_t})\right|^2\, dx\\
          &\qquad + C\int_{\Gamma^{m_\eps}_{t,h}(\eps)}\frac1{m_\eps} \left(\eps \frac{|\nabla \varphi_\eps|^2}2 + \frac{W(\varphi_\eps)}\eps\right)\!(x,t)\left|d_{\Gamma^{m_\eps}_t}-h(P_{\Gamma^{m_\eps}_t})\right|^2\, dx\\
          &\quad \leq C\sum_{n=2}^N \frac{\eps^2}{m_\eps}n^2 \int_{\Gamma^{m_\eps}_{t,h}(n\eps)\setminus \Gamma^{m_\eps}_t((n-1)\eps)} \left(\eps \frac{|\nabla \varphi_\eps|^2}2 + \frac{W(\varphi_\eps)}\eps\right)\!(x,t)\, dx+ CC_1\frac{\eps^2}{m_\eps}\\
          &\quad \leq C\sum_{n=2}^N \frac{\eps^2}{m_\eps}n^2 \left(E[\varphi_\eps|\Gamma^{m_\eps}](t)
          +\int_\Rd |\chi-\chi_{S_{b(t)}}|\min\left\{d_{\Gamma^{m_\eps}_t},1\right\}\,dx 
          + e^{-c(n-1)}\right)+ CC_1\frac{\eps^2}{m_\eps}\\
          &\quad \leq C \frac{\eps^2}{m_\eps}N^3 \left(E[\varphi_\eps|\Gamma^{m_\eps}](t)
          +\int_\Rd |\chi-\chi_{S_{b(t)}}|\min\left\{d_{\Gamma^{m_\eps}_t},1\right\}\,dx\right) + CC_1 \frac{\eps^2}{m_\eps},
        \end{align*}
    	where we used Lemma \ref{th_energy_out_strip}. This shows the claim.
\end{proof}

\section{Bulk Error Identity}\label{sec_Ebulk_estimate}

For the bulk error functional $E_\textup{bulk}[\varphi_\eps|\Gamma^{m_\eps}]$ defined in \eqref{eq_Ebulk}
we have the following identity:
\begin{lemma}[\textbf{Bulk Error Identity}]\label{th_Ebulk_identity}
	Let the assumptions of Section \ref{sec_Energy_def_coercivity} hold. More precisely, let $(\ve_{m_\eps}^\pm,p_{m_\eps}^\pm,(\Gamma_t^{m_\eps})_{t\in[0,T_0]})$ for $m_\eps>0$ small be solutions of the adjusted two-phase Navier-Stokes equation \eqref{eq:ApproxLimit1}-\eqref{eq:ApproxLimit7} on $[0,T_0]$, cf.~Theorem \ref{th:approx_lim_uniform} below. Moreover, let $(\ve_\eps, \varphi_\eps)$ for $\eps>0$ small be energy dissipating weak solutions to the Navier-Stokes/Allen-Cahn system \eqref{eq_NSAC1}-\eqref{eq_NSAC5} on $[0,T_0]$ with constant mobility $m_\eps>0$ as in Remark \ref{th_weak_sol_NS}. Finally, let $\Gamma^{m_\eps}, \Omega^{m_\eps,\pm}, \ve^{m_\eps}$ be as in \eqref{eq_notation_meps}, $\sigma$, $\psi_\eps$, $\vec{n}_\eps$ be as in \eqref{eq_psi_def}-\eqref{eq_n_eps}, $\xi$, $B$ be as in \eqref{eq_xi_B_def}, $E_\textup{bulk}[\varphi_\eps|\Gamma^{m_\eps}]$ and $\vartheta$ be as in \eqref{eq_Ebulk} and $H_\eps$ be defined as in \eqref{eq_Heps_def}. Then for all $T\in[0,T_0]$:\begin{subequations}
		\begin{align}
			&E_\textup{bulk}[\varphi_\eps|\Gamma^{m_\eps}](T)\notag\\
			&=E_\textup{bulk}[\varphi_\eps|\Gamma^{m_\eps}](0)+\int_0^T\int_\Omega(\sigma\chi_{\Omega^{m_\eps,+}} -\psi_\eps)(\partial_t+B\cdot\nabla)\vartheta\,dx\,dt\label{eq_Ebulk1}\\
			&+\int_0^T\int_\Omega (\sigma\chi_{\Omega^{m_\eps,+}} -\psi_\eps)\vartheta\,\nabla\cdot B\,dx\,dt\label{eq_Ebulk2}\\
			&-\int_0^T\int_\Omega \vartheta(B-\ve^{m_\eps})\cdot(\vec{n}_\eps-\xi)|\nabla\psi_\eps|\,dx\,dt\label{eq_Ebulk3}\\
			&+\int_0^T\int_\Omega (\sigma\chi_{\Omega^{m_\eps,+}} -\psi_\eps)(\ve_\eps-\ve^{m_\eps})\cdot\nabla\vartheta\,dx\,dt\label{eq_Ebulk4}\\
			&-\int_0^T\int_\Omega \vartheta(B-\ve^{m_\eps})\cdot\xi (|\nabla\psi_\eps|-\eps|\nabla\varphi_\eps|^2)\,dx\,dt\label{eq_Ebulk5}\\
			&+m_\eps\int_0^T\int_\Omega\frac{1}{\sqrt{\eps}}\left(H_\eps-\frac{B-\ve^{m_\eps}}{m_\eps}\cdot\xi\,\eps|\nabla\varphi_\eps|\right)\vartheta\sqrt{\eps}|\nabla\varphi_\eps|\,dx\,dt\label{eq_Ebulk6}\\
			&-m_\eps\int_0^T\int_\Omega\vartheta\frac{1}{\sqrt{\eps}}\left(H_\eps+\sqrt{2W(\varphi_\eps)}\nabla\cdot\xi\right)\left(\sqrt{\eps}|\nabla\varphi_\eps|-\frac{\sqrt{2W(\varphi_\eps)}}{\sqrt{\eps}}\right)\,dx\,dt\label{eq_Ebulk7}\\
			&-m_\eps\int_0^T\int_\Omega\vartheta(\nabla\cdot\xi)\left|\sqrt{\eps}|\nabla\varphi_\eps|-\frac{\sqrt{2W(\varphi_\eps)}}{\sqrt{\eps}}\right|^2\,dx\,dt\label{eq_Ebulk8}\\
			&+m_\eps\int_0^T\int_\Omega\vartheta\sqrt{\eps}|\nabla\varphi_\eps|(\nabla\cdot\xi)\left(\sqrt{\eps}|\nabla\varphi_\eps|-\frac{\sqrt{2W(\varphi_\eps)}}{\sqrt{\eps}}\right)\,dx\,dt.\label{eq_Ebulk9}
	\end{align}\end{subequations}
\end{lemma}
\begin{proof}
	This can be shown analogously to \cite[Lemma 7]{HenselLiu}. The $m_\eps$-factors appear here because of the Allen-Cahn part \eqref{eq_NSAC3} through the equation for $\partial_t\psi_\eps$. Note that we use a different sign convention for $\vartheta$ compared to \cite{HenselLiu}, this just changes the signs in all terms.
\end{proof}

\section{Proof of Theorem \ref{th_stability}(Stability Estimate)}\label{sec_main_proof}
\begin{proof}[Proof of Theorem \ref{th_stability}]
	Up to \eqref{eq_Erel_inequ_problem1}, \eqref{eq_Erel_inequ_problem2}, \eqref{eq_Erel_inequ_new} and \eqref{eq_Erel_inequ_problem3} the terms in the estimate for the relative entropy from Lemma \ref{th_Erel_inequ} can be estimated analogously to \cite[Proof of Theorem 1]{HenselLiu}. Here note that $\|\nabla\ve_\eps-\nabla\ve^{m_\eps}\|_{L^2(0,T;L^2(\Omega))}^2$ and \eqref{eq_Erel_inequ_pos1}-\eqref{eq_Erel_inequ_pos2} are terms with a good sign. At this point, let us estimate a few terms for the convenience of the reader. For example,
	\[
	\left|\int_0^T \int_\Omega (\sigma\chi_{\Omega^{m_\eps,+}}-\psi_\eps) ((\ve_\eps-\ve^{m_\eps})\cdot\nabla)(\nabla\cdot\xi)\,dx\,dt\right|\leq 
	C\int_0^T \int_\Omega |\sigma\chi_{\Omega^{m_\eps,+}}-\psi_\eps| |\ve_\eps-\ve^{m_\eps}|\,dx\,dt,
	\]
	where the latter term can be estimated with \eqref{eq_Ebulk_coercivity2} by using $\frac{1}{2}\|\nabla\ve_\eps-\nabla\ve^{m_\eps}\|_{L^2(0,T;L^2(\Omega))}^2$ for absorption. Moreover, due to \eqref{eq_calib4} it holds
	\[
	\left|\int_0^T \int_\Omega \left((\partial_t+B\cdot\nabla)|\xi|^2\right) |\nabla\psi_\eps|\,dx\,dt\right|
	\leq \int_0^T \int_\Omega \min\{d_{\Gamma^{m_\eps}}^2,1\}|\nabla\psi_\eps|\,dx\,dt,
	\]
	which is controlled due to \eqref{eq_Erel_coercivity_1}. Finally, let us estimate
	\[
	\left|\int_0^T \int_\Omega (\vec{n}_\eps\otimes \vec{n}_\eps-\xi\otimes\xi):\nabla B(\eps|\nabla\varphi_\eps|^2-|\nabla\psi_\eps|)\,dx\,dt\right|
	\leq C\int_0^T \int_\Omega \sqrt{1-\vec{n}_\eps\cdot\xi}\left|\eps|\nabla\varphi_\eps|^2-|\nabla\psi_\eps|\right|\,dx\,dt,
	\]
	where we used $|(\vec{n}_\eps\otimes \vec{n}_\eps-\xi\otimes\xi):\nabla B|=|\vec{n}_\eps\otimes(\vec{n}_\eps-\xi):\nabla B+(\xi\cdot\nabla)B\cdot(\vec{n}_\eps-\xi)|\leq|\vec{n}_\eps-\xi|$ and $|\vec{n}_\eps-\xi|^2=2(1-\vec{n}_\eps\cdot\xi)$. Hence the above term is controlled via \eqref{eq_Erel_coercivity_2}.

	Let us consider the remaining four terms \eqref{eq_Erel_inequ_problem1}, \eqref{eq_Erel_inequ_problem2}, \eqref{eq_Erel_inequ_new} and \eqref{eq_Erel_inequ_problem3} for which new estimates are needed. First, the new choice \eqref{eq_xi_B_def} of $\xi$ and $B$ compared to \cite[(75)-(76)]{HenselLiu} yields \eqref{eq_calib3} and enables us to estimate \eqref{eq_Erel_inequ_problem1} with \eqref{eq_Erel_coercivity_1}. Moreover, for the estimate of \eqref{eq_Erel_inequ_problem2} it remains to control 
	\[
	\int_0^T\int_\Omega \frac{1}{m_\eps}\left|(\ve^{m_\eps}-B)\cdot(\vec{n}_\eps-\xi)\right|^2 \eps|\nabla\varphi_\eps|^2\,dx\,dt
	\]
	due to absorption with \eqref{eq_Erel_inequ_pos1}. However, we have $\ve^{m_\eps}-B=-m_\eps H\vec{n}\tilde{\eta}_{m_\eps}$ by definition \eqref{eq_xi_B_def} of $B$, hence the term is controlled by $m_\eps \int_0^TE[\ve_\eps,\varphi_\eps| \ve^{m_\eps},\Gamma^{m_\eps}](t)\,dt$ because of \eqref{eq_Erel_coercivity_1}. Additionally, to estimate the term \eqref{eq_Erel_inequ_new} we write $(\xi\otimes\xi(\nabla B)^\top\xi)\cdot(\vec{n}_\eps-\xi)=(\xi^\top(\nabla B)^\top\xi)\xi\cdot(\vec{n}_\eps-\xi)$. Hence because of $\xi\cdot(\vec{n}_\eps-\xi)=\vec{n}_\eps\cdot\xi-1+1-|\xi|^2$ and \eqref{eq_cutoff_quadr}, the term \eqref{eq_Erel_inequ_new} is controlled via \eqref{eq_Erel_coercivity_1}. Finally, we estimate \eqref{eq_Erel_inequ_problem3} with Lemma \ref{th_Erel_lastterm}, where the $1$-Lipschitz-function $h=h_\eps(.,t)$ for a.e.~$t\in[0,T_0]$ is taken from Lemma \ref{th_lem_height} and we consider $\delta_\eps:=C_0|\log\eps|\eps$ for some $C_0\geq 1$ and $\eps$ small. It remains to estimate \eqref{eq_estimate_lastterm2}-\eqref{eq_estimate_lastterm5}. First, note that \eqref{eq_estimate_lastterm4} is absorbed by \eqref{eq_Erel_inequ_pos1}. Moreover, \eqref{eq_estimate_lastterm5} is controlled via Corollary \ref{th_h_energyweight} and Lemma \ref{th_Ebulk_BVerror}. More precisely, we have
	\begin{align*}
	&\int_{\Gamma^{m_\eps}_{t,h}(\delta_\eps)}(|h|_{P_{\Gamma^{m_\eps}}}|^2+|\nabla_\tau (h|_{P_{\Gamma^{m_\eps}}})|^2)\left(\eps|\nabla\varphi_\eps|^2+\frac{1}{\eps}W(\varphi_\eps)\right)|_{(x,t)}\,dx\\
	&\leq C E[\varphi_\eps|\Gamma^{m_\eps}](t)
	+C \int_\Rd |\chi-\chi_{S_{b(t)}}|\min\left\{d_{\Gamma^{m_\eps}_t},1\right\} \,dx\leq C(E[\varphi_\eps|\Gamma^{m_\eps}](t)+E_\textup{bulk}[\varphi_\eps|\Gamma^{m_\eps}](t)),
	\end{align*}
	where $b(t)=b_\eps(t)$ and $S_{b(t)}$ are as in Lemma \ref{th_levelset}. Moreover, \eqref{eq_estimate_lastterm2} is suitably estimated by Lemma \ref{th_energy_out_strip} provided that $C_0\geq 1$ is large enough such that $e^{-cC_0|\log\eps|}\leq \eps^2$. Finally, \eqref{eq_estimate_lastterm3} is controlled by Corollary \ref{th_energy_out2} and Lemma \ref{th_Ebulk_BVerror}. Altogether we obtain
	\begin{align*}
		\frac{1}{2}\|\nabla\ve_\eps-\nabla\ve^{m_\eps}\|_{L^2(0,T;L^2(\Omega))}^2+E[\ve_\eps,\varphi_\eps| \ve^{m_\eps},\Gamma^{m_\eps}](T)
		\leq E[\ve_\eps,\varphi_\eps| \ve^{m_\eps},\Gamma^{m_\eps}](0)
		+C\frac{\eps^2}{m_\eps}\\+C\left(\frac{|\log\eps|^3\eps^2}{m_\eps}+1\right)\int_0^T E[\ve_\eps,\varphi_\eps| \ve^{m_\eps},\Gamma^{m_\eps}](t)+E_\textup{bulk}[\varphi_\eps|\Gamma^{m_\eps}](t)\,dt.
	\end{align*}

	Next, we estimate the terms on the right hand side in the identity for $E_\textup{bulk}[\varphi_\eps|\Gamma^{m_\eps}]$ from Lemma \ref{th_Ebulk_identity}. The terms \eqref{eq_Ebulk1}-\eqref{eq_Ebulk2} are controlled by the bulk error itself due to \eqref{eq_theta_evol}, \eqref{eq_theta_coerc1} and \eqref{eq_Ebulk_coercivity1}. Moreover, \eqref{eq_Ebulk3} and \eqref{eq_Ebulk5} can be estimated by the relative entropy because of \eqref{eq_Erel_coercivity_1} and \eqref{eq_Erel_coercivity_2}, respectively. The term \eqref{eq_Ebulk4} is controlled by the bulk functional due to \eqref{eq_Ebulk_coercivity2}. Moreover, we estimate \eqref{eq_Ebulk6}-\eqref{eq_Ebulk7} via Young's inequality in order to absorb one part with the positive terms \eqref{eq_Erel_inequ_pos2} and \eqref{eq_Erel_inequ_pos1}, respectively. The remainders as well as \eqref{eq_Ebulk8}-\eqref{eq_Ebulk9} are controlled by the relative entropy because of \eqref{eq_Erel_coercivity_1} and \eqref{eq_Erel_coercivity_equipart}. Finally, this yields
	\begin{align*}
		E_\textup{bulk}[\varphi_\eps|\Gamma^{m_\eps}](t)
		\leq E_\textup{bulk}[\varphi_\eps|\Gamma^{m_\eps}](0)
		+C\int_0^T E[\ve_\eps,\varphi_\eps| \ve^{m_\eps},\Gamma^{m_\eps}](t)+E_\textup{bulk}[\varphi_\eps|\Gamma^{m_\eps}](t)\,dt.
	\end{align*}

	Let $m_\eps=m_0\eps^{\beta}$ with $m_0>0$ and $\beta\in(0,2)$ be as in the theorem. Then it holds $\frac{|\log\eps|^3\eps^2}{m_\eps}\leq \frac{C_\beta}{m_0}$ for some constant $C_\beta>0$ independent of $\eps$ for $\eps>0$ small. Finally, the Gronwall inequality yields Theorem \ref{th_stability}.
    \end{proof}

\section{Existence for the Approximate Two-Phase flow}\label{sec_approx_twophase}
      In this section we consider the existence of strong solutions for the modified two-phase flow
      \begin{alignat}{2}
	\label{eq:ApproxLimit1}
	\partial_t \ve_m^\pm+\ve_m^\pm\cdot\nabla\ve_m^\pm-\Delta \ve_m^\pm  +\nabla p_m^\pm &= 0 &\quad &\text{in }\Omega^{m,\pm}_t, t\in [0,T_0],\\\label{eq:ApproxLimit2}
	\Div \ve_m^\pm &= 0 &\quad &\text{in }\Omega^{m,\pm}_t, t\in [0,T_0],\\\label{eq:ApproxLimit3}
	-\llbracket 2 D\ve_m^\pm -p_m^\pm \mathbf{I}\rrbracket\no_{\Gamma^m_t} &= \sigma H_{\Gamma^m_t}\no_{\Gamma^m_t} && \text{on }\Gamma^m_t, t\in [0,T_0],\\ \label{eq:ApproxLimit4}
	\llbracket\ve_m^\pm \rrbracket &=0 && \text{on }\Gamma^m_t, t\in [0,T_0],\\
	\label{eq:ApproxLimit5}
	V_{\Gamma^m_t}-\no_{\Gamma^m_t}\cdot \ve_m^\pm   &= m H_{\Gamma^m_t} && \text{on }\Gamma^m_t, t\in [0,T_0],\\
	\label{eq:ApproxLimit6}
	\ve_m^-|_{\partial\Omega}&= 0&&\text{on }\partial\Omega\times(0,T_0),\\
	\Gamma^m_0=\Gamma^0,\quad \ve_m^\pm|_{t=0}&= \ve_0^\pm && \text{in }\Omega^\pm_0,\label{eq:ApproxLimit7}
      \end{alignat}
      where $T_0>0$ is such that \eqref{eq_TwoPhase1}-\eqref{eq_TwoPhase7} possesses a smooth solution and $m>0$ is sufficiently small. Here, analogously as for  $\Omega^\pm_t$ and $\Gamma_t$, the domain $\Omega$ is the disjoint union of two sufficiently smooth domains $\Omega^{m,\pm}_t$ and an evolving interface $\Gamma^m_t=\partial\Omega^{m,+}_t$ for every $t\in[0,T]$. Furthermore, $\no_{\Gamma^m_t}$, $V_{\Gamma^m_t}$, and $H_{\Gamma^m_t}$ denote interior normal (with respect to $\Omega^{m,+}_t$), the normal velocity, and the mean curvature of $\Gamma^m_t$, respectively.
      We note that before $m=m_\eps>0$ depends on $\eps>0$ and $m_\eps \to_{\eps\to 0} 0$. But in the system above only the value of $m>0$ enters. Therefore we skip the $\eps$-dependence.
      The goal is to show that then for every $m>0$ sufficiently small also \eqref{eq:ApproxLimit1}-\eqref{eq:ApproxLimit7} possesses a strong solution on the same time interval $[0,T_0]$, which is close to the solution of \eqref{eq_TwoPhase1}-\eqref{eq_TwoPhase7} in a certain sense. The idea for the proof is to use the interface $\Gamma_t$ of the solution of \eqref{eq_TwoPhase1}-\eqref{eq_TwoPhase7} for every $t\in[0,T_0]$ as a reference surface and to transform the modified system \eqref{eq:ApproxLimit1}-\eqref{eq:ApproxLimit7} with the aid of the Hanzawa transformation with respect to $\Gamma_t$ to a perturbed two-phase flow problem in $\Omega^\pm_t$, $t\in[0,T_0]$. To show solvability of the transformed system for small $m>0$ we reduce the system to a fixed-point problem with the aid of the invertibility of the principal part of the linearized system and apply the contraction mapping principle as usual. But to apply this strategy invertibility of the principal part of the linearized system to \eqref{eq:ApproxLimit1}-\eqref{eq:ApproxLimit7} together with uniform estimates in $m>0$ are essential. These results are obtained in the following subsection.

      Throughout this section we will use the notation from \cite{PruessSimonettMovingInterfaces}. In particular, we note that for a Banach space $X$, $1\leq p\leq\infty$, and $T_0>0$
      \begin{equation*}
        {}_0 W^1_q(0,T;X)= \{u\in W^1_q(0,T;X): u|_{t=0} =0\}.
      \end{equation*}
      \subsection{Analysis of the Linearized System}

      In this subsection $(\Gamma_t)_{t\in[0,T_0]}$ is a smooth evolving family of $(d-1)$-dimensional submanifolds such that $\Gamma_t= \partial\Omega^+_t \subseteq \Omega$ and $\Omega$ is the disjoint union of $\Gamma_t$ and two smooth domains $\Omega^+_t$ and $\Omega^-_t$ for every $t\in[0,T_0]$. Moreover, we assume that there is a continuous $X\colon\overline{\Omega}\times [0,T_0]\to \overline{\Omega}$ such that $(X|_{\Omega_0^\pm \times [0,T_0]},\operatorname{id})\colon \overline{\Omega_0^\pm} \times [0,T_0] \to \overline{\Omega^\pm} $ are smooth and  
      \begin{enumerate}
      \item $X_t:=X(\cdot,t)\colon \overline{\Omega_0^\pm}\to \overline{\Omega_t^\pm}$ are smooth diffeomorphisms for all $t\in[0,T_0]$,
      \item $\det (D_\xi X_t(\xi))=1$ for all $\xi\in\Omega_0\setminus \Gamma_0$, $t\in[0,T_0]$.
      \end{enumerate}
      Here we use the same notation as in Section~\ref{sec_Energy_def_coercivity}.
      In particular, we have $X_t(\Gamma_0)= \Gamma_t$ for all $t\in[0,T_0]$.
      In the later applications $(\Gamma_t)_{t\in[0,T_0]}$ is given by the smooth solution of \eqref{eq_TwoPhase1}-\eqref{eq_TwoPhase7} and $X_t=X(\cdot,t)$ can be obtained as the solutions of
      \begin{alignat*}{2}
        \frac{d}{dt} X_t(\xi) &= \ve_0^\pm(X_t(\xi),t) &\quad & \text{for all }\xi\in\overline{\Omega^\pm_0}, t\in [0,T_0],\\
        X_0(\xi) &= \xi &\quad & \text{for all }\xi\in\overline{\Omega^\pm_0}.
      \end{alignat*}
      Since $	\llbracket\ve^\pm_0 \rrbracket=0$, $X\colon \overline{\Omega_0}\times [0,T_0]\to \overline{\Omega_0}$ is well defined and continuous.

      We consider the linearized system
      \begin{alignat}{2}
	\label{eq:Linear1}
	\partial_t \ve^\pm-\Delta \ve^\pm  +\nabla p^\pm &= \mathbf{f} &\qquad &\text{in }\Omega^\pm_t, t\in [0,T_0],\\\label{eq:Linear2}
	\Div \ve^\pm &= g &\qquad &\text{in }\Omega^\pm_t, t\in [0,T_0],\\\label{eq:Linear3}
	-\llbracket 2D\ve^\pm -p^\pm \mathbf{I}\rrbracket\no_{\Gamma_t} &= \sigma \Delta_{\Gamma_t} h \no_{\Gamma_t} + \mathbf{a} && \text{on }\Gamma_t, t\in [0,T_0],\\ \label{eq:Linear4}
	\llbracket\ve^\pm \rrbracket &=0 && \text{on }\Gamma_t, t\in [0,T_0],\\
	\label{eq:Linear5}
	\partial_t^\bullet h 
        &= \no_{\Gamma_t} \cdot\ve + m \Delta_{\Gamma_t} h+ b && \text{on }\Gamma_t, t\in [0,T_0],\\
	\label{eq:Linear6}
	\ve_0^-|_{\partial\Omega}&= 0&&\text{on }\partial\Omega\times[0,T_0],\\
	h|_{t=0}=h_0,\qquad \ve^\pm|_{t=0}&= \ve_0^\pm && \text{in }\Omega^\pm_0,\label{eq:Linear7}
      \end{alignat}
      where $\ve^\pm = \ve|_{\Omega^\pm}$ and $\partial_t^\bullet h$ denotes the material time derivative of $h\colon \Gamma\to \R$ defined by
      \begin{equation*}
        \partial_t^\bullet h = \partial_t \tilde{h} +\nabla \tilde{h}\cdot (\partial_t X_t)\circ X_t^{-1}\qquad \text{on }\Gamma,
      \end{equation*}
      where $\tilde{h}\colon \Gamma(\delta)\to \R$ with $\tilde{h}(x,t)= h(P_{\Gamma_t}(x),t)$ for all $x\in \Gamma_t(\delta), t\in[0,T_0]$ for a sufficiently small $\delta$.

      The main result of this subsection is:
      \begin{theorem}\label{thm:LinSystem}
        Let $q>d+2$, $T_0\in (0,\infty)$,
        $m\in(0,1)$,  and $\Omega, \Gamma_t, \Omega^\pm_t$, $t\in[0,T_0]$, be as before. Moreover, let
        \begin{alignat*}{1}
          \mathbf{f}&\in L^q(0,T_0; L^q(\Omega))^d,\quad          g\in  L^q(0,T_0; W^1_q(\Omega\setminus \Gamma_t))\cap W^1_q(0,T_0;\dot{W}^{-1}_q(\Omega)), \quad \ve_0 \in W^{2-\frac2q}_q(\Omega \setminus \Gamma_0)^d,\\
          \mathbf{a} &\in W^{\frac12-\frac1{2q}}_q(0,T_0;L^q(\Gamma_t))^d\cap L^q(0,T_0; W^{1-\frac1q}_q(\Gamma_t))^d,\quad h_0 \in W^{3-\frac2q}_q(\Gamma_0),\\
          b &\in W^{1-\frac1{2q}}_q(0,T_0;L^q(\Gamma_t))\cap L^q(0,T_0; W^{2-\frac1q}_q(\Gamma_t))
        \end{alignat*}
        satisfy the compatibility conditions:
        \begin{enumerate}
        \item $\ve_0|_{\partial\Omega}=0$, $\llbracket\ve_0^\pm \rrbracket =0$,
        \item $\operatorname{div} \ve_0 = g|_{t=0}$,
        \item $-P_{T\Gamma_0}\llbracket 2D\ve^\pm_0\rrbracket\no_{\Gamma_t}= P_{T\Gamma_0} \mathbf{a}|_{t=0}$,
        \end{enumerate}
        then there is a unique solution $(\ve, p, h)$ such that
        \begin{align*}
          \ve &\in W^1_q(0,T_0; L^q(\Omega))^d\cap L^q(0,T_0;W^2_q(\Omega\setminus \Gamma_t))^d,\quad p\in L^q(0,T_0; \dot{W}^1_q(\Omega\setminus \Gamma_t)),\\
          \llbracket p \rrbracket &\in W^{\frac12-\frac1{2q}}_q(0,T_0;L^q(\Gamma_t))\cap L^q(0,T_0; W^{1-\frac1q}_q(\Gamma_t)),\\
          h &\in W^{2-\frac1{2q}}_q(0,T_0; L^q(\Gamma_t))\cap W^1_q(0,T_0; W^{2-\frac1q}_q(\Gamma_t))\cap L^q(0,T_0; W^{4-\frac1q}_q(\Gamma_t)).
        \end{align*}
        Moreover, there is some $C>0$ independent of $m\in (0,1)$ 
        and the data $(\mathbf{f}, g, \mathbf{a}, b,\ve_0 , h_0)$ such that
        \begin{align}\nonumber
         & \|\ve\|_{W^1_q(0,T_0; L^q(\Omega))}+ \|\ve\|_{L^q(0,T_0;W^2_q(\Omega\setminus \Gamma_t))} + \|p\|_{L^q(0,T_0; \dot{W}^1_q(\Omega\setminus \Gamma_t))}+ \| \llbracket p \rrbracket\|_{W^{\frac12-\frac1{2q}}_q(0,T_0;L^q(\Gamma_t))}\\\nonumber
          &\quad  + \| \llbracket p \rrbracket\|_{L^q(0,T_0; W^{2-\frac1q}_q(\Gamma_t))}+\|h\|_{W^{2-\frac1{2q}}_q(0,T_0; L^q(\Gamma_t))} + \|h\|_{W^1_q(0,T_0; W^{1-\frac1q}_q(\Gamma_t))}+ \|h\|_{L^q(0,T_0; W^{3-\frac1q}_q(\Gamma_t))}\\\nonumber
          &\quad + m\|h\|_{L^q(0,T_0; W^{4-\frac1q}_q(\Gamma_t))}\\\nonumber
          &\qquad \leq C\left(
            \|\mathbf{f}\|_{L^q(0,T_0; L^q(\Omega))}+\|g\|_{L^q(0,T_0; W^1_q(\Omega\setminus \Gamma_t))}+\|\mathbf{a}\|_{W^{\frac12-\frac1{2q}}_q(0,T_0;L^q(\Gamma_t))\cap L^q(0,T_0; W^{2-\frac1q}_q(\Gamma_t))}\right.\\
          &\qquad \quad \qquad \left.+\|b\|_{L^q(0,T_0; W^{2-\frac1q}_q(\Gamma_t))}
            +\|\ve_0\|_{W^{2-\frac2q}_q(\Omega \setminus \Gamma_0)}+\| h_0\|_{ W^{4-\frac3q}_q(\Gamma_0)}
            \right)\label{eq:UnifEstim}
        \end{align}
      \end{theorem}
      Here we define $L^q(0,T_0; W^m_q(\Omega\setminus \Gamma_t))$ such that $g\in L^q(0,T_0; W^m_q(\Omega\setminus \Gamma_t))$ if and only if $g\circ (X,\operatorname{id}_{(0,T_0)})\in L^q(0,T_0; W^m_q(\Omega\setminus \Gamma_0))$ for $m\in\mathbb{N}_0$ and analogously for the other function spaces involving $\Gamma_t$.

      \subsection{Case of a Flat Interface}
      Throughout this section we assume that $\Omega^\pm_t= \R^d_\pm$, $\Gamma_t= \R^{d-1}\times \{0\}$, and $\Omega = \R^d$. We follow the strategy of \cite[Section 8.2]{PruessSimonettMovingInterfaces}. To this end we first assume that $\mathbf{f}=g=\mathbf{a}=\ve_0=h_0=0$, use a spectral shift, and consider for $\omega>0$
      \begin{alignat}{2}
	\label{eq:FlatLinear1}
	(\partial_t +\omega) \ve^\pm-\Delta \ve^\pm  +\nabla p^\pm &=0 &\qquad &\text{in }\R^d_\pm\times (0,\infty),\\\label{eq:FlatLinear2}
          \Div \ve^\pm &= 0 &\qquad &\text{in }\R^d_\pm\times (0,\infty),\\\label{eq:FlatLinear3}
	\llbracket 2D\ve^\pm -p^\pm \mathbf{I}\rrbracket e_d &= -\sigma \Delta_{\R^{d-1}} h e_d && \text{on }\R^{d-1}\times \{0\}\times (0,\infty),\\ \label{eq:FlatLinear4}
	\llbracket\ve^\pm \rrbracket &=0 && \text{on }\R^{d-1}\times \{0\}\times (0,\infty),\\
	\label{eq:FlatLinear5}
	(\partial_t+ \omega) h + \ve_d  - m \Delta_{\R^{d-1}} h &= b && \text{on }\R^{d-1}\times \{0\}\times (0,\infty),\\\label{eq:FlatLinear6}
        (\ve, h)|_{t=0} &= (0,0),
      \end{alignat}
      where $\ve^\pm =\ve|_{\R^d_\pm}$, $p^\pm =p|_{\R^d_\pm}$.
      
      The goal of this subsection is to prove:
      \begin{theorem}\label{thm:FlatInt}
        For any $\omega>0$, $1<q<\infty$,
        $$
        b \in {}_0W^{1-\frac1{2q}}_q(0,\infty;L^q(\R^{d-1}))\cap L^q(0,\infty; W^{2-\frac1q}_q(\R^{d-1}))
        $$
        there is a unique solution
        \begin{align*}
          \ve&\in {}_0W^1_q(0,\infty; L^q(\R^d))\cap L^q(0,\infty; W^2_q(\R^d\setminus (\R^{d-1}\times \{0\})))^d,\\
          p &\in L^q(0,\infty; \dot{W}^1_q(\R^d\setminus (\R^{d-1}\times \{0\})))\\
          &\text{ with }	\llbracket p^\pm \rrbracket \in W^{\frac12-\frac1{2q}}_q(0,\infty; L^q(\R^{d-1}))\cap L^q(0,\infty; W^{1-\frac1q}_q(\R^{d-1})),
          \\
        h &\in {}_0W^{2-\frac1{2q}}_q(0,\infty;L^q(\R^{d-1}))\cap {}_0W^1_q(0,\infty; W^{2-\frac1q}_q(\R^{d-1}))\cap L^q(0,\infty; W^{4-\frac1q}_q(\R^{d-1}))          
        \end{align*}
        satisfying
        \begin{align}\nonumber
          &\|\ve\|_{W^1_q(0,\infty; L^q)} + \|\ve\|_{L^q(0,\infty; W^2_q)}+ \|p\|_{L^q(0,\infty; W^1_q)}+\left\| \llbracket p^\pm \rrbracket\right\|_{W^{\frac12-\frac1{2q}}(0,\infty; L^q)\cap L^q(0,\infty; W^{1-\frac1q}_q)}\\
          &\quad + \|h\|_{W^{2-\frac1{2q}}_q(0,\infty; L^q)} + \|h\|_{W^1_q(0,\infty; W^{2-\frac1q}_q)}
            +\|h\|_{ L^q(0,\infty; W^{3-\frac1q}_q)}\nonumber\\\label{eq:EstimFlatInt}
          &\quad +m\|h\|_{ L^q(0,\infty; W^{4-\frac1q}_q)} \leq C\left(\|b\|_{W^{1-\frac1{2q}}_q(0,\infty;L^q)}+\|b\|_{L^q(0,\infty; W^{2-\frac1q}_q)}\right)
        \end{align}
        uniformly in $m\in[0,1]$.
      \end{theorem}
      Analogously as in \cite{PruessSimonettMovingInterfaces} one obtains that $(\ve^\pm, p^\pm, h)$ solves \eqref{eq:FlatLinear1}-\eqref{eq:FlatLinear6} if and only if
      \begin{align}\label{eq:h}
        (\partial_t+\omega) h + e_d \cdot (\mathcal{D}\mathcal{N})^{-1}
        \begin{pmatrix}
          0\\ -\sigma \Delta_{\R^{d-1}} h
        \end{pmatrix}
- m \Delta_{\R^{d-1}} h &= b\quad  \text{on }\R^{d-1}\times \{0\}\times (0,\infty),\\\nonumber
        (\ve, h)|_{t=0} &= (0,0),
      \end{align}
      where $\mathcal{D}\mathcal{N}$ is the Dirichlet-to-Neumann operator for the two-phase Stokes problem as in \cite[Section~8.3.2]{PruessSimonettMovingInterfaces} and $(\ve^\pm, p^\pm)$ is determined by \eqref{eq:FlatLinear1}-\eqref{eq:FlatLinear4} and $\ve|_{t=0}=0$ in dependence on $h$. More precisely, the Laplace-Fourier transform of $u=e_d \cdot (\mathcal{D}\mathcal{N})^{-1}
        \begin{pmatrix}
          0\\ -\sigma \Delta_{\R^{d-1}} h
        \end{pmatrix}$ is given by
        \begin{equation*}
          \hat{u}(\lambda,\xi)= \frac{-\sigma |\xi|^2}{2\lambda/|\xi|+ 2(\lambda+|\xi|^2)^{\frac12}+|\xi|} \hat{h}(\lambda,\xi),
        \end{equation*}
        cf.\ \cite[Equation (8.34)]{PruessSimonettMovingInterfaces}, where we note that $\eta_1=\eta_2= (\lambda+|\xi|^2)^{\frac12}+|\xi|$ since $\rho_1=\rho_2=\mu_1=\mu_2=1$ in our case. 
        Hence the Fourier-Laplace transformation of \eqref{eq:h} is $s_m(\lambda,|\xi|)\hat{h}(\lambda,\xi)= \hat{b}(\lambda,\xi)$, where
        \begin{equation*}
          s_m(\lambda,\tau) = \lambda + m\tau^2 +\frac{\sigma \tau^2}{2\lambda/\tau+ 2(\lambda+\tau^2)^{\frac12}+\tau}\quad \text{for }\lambda\in \C, \tau\in  \mathbb{C}\setminus \{0\}.
        \end{equation*}
        We note that $s_0(\lambda,|\xi|)$ coincides with the symbol $s_{0,0}(\lambda,\xi)$ studied in \cite[Section~8.3.4]{PruessSimonettMovingInterfaces}, for which it was shown
        \begin{equation*}
          |s_{0,0}(\lambda,\xi)|\leq C_\eta(|\lambda|+|\xi|) \qquad \text{for all }\lambda \in \Sigma_{\pi/2+\eta}, \xi\in \left(\Sigma_\eta \cup -\Sigma_\eta\right)^{n-1}, 
        \end{equation*}
        cf.\ \cite[Equation (8.50)]{PruessSimonettMovingInterfaces} for every $\eta\in [0,\frac\pi{2})$. Hence for every $\eta\in [0,\frac\pi{2})$ there is some $C_\eta$ such that 
        \begin{equation*}
          |s_m(\lambda,\tau)|\leq C_\eta(|\lambda|+|\tau|+m|\tau|^2) \qquad \text{for all }\lambda \in \Sigma_{\pi/2+\eta}, \tau\in \Sigma_\eta, 
        \end{equation*}
        uniformly in $m\in[0,1]$. The essential point is that we have the same kind of lower bound:
        \begin{lemma}
          For any $\omega_0>0$ there is some $\eta>0$ and $c>0$ such that
          \begin{equation}\label{eq:smEstim}
              |s_m(\lambda,\tau)|\geq c(|\lambda|+|\tau|+m|\tau|^2) \qquad \text{for all }\lambda \in \Sigma_{\pi/2+\eta},|\lambda|\geq \omega_0, \tau\in \Sigma_\eta, 
          \end{equation}
        \end{lemma}
        \begin{proof}
          First of all, we note that
          \begin{equation*}
            s_m(\lambda,\tau) = \lambda + m\tau^2+\sigma \tau k(z) \qquad \text{with } z= \frac{\lambda}{\tau^2}, \tau \neq 0,
          \end{equation*}
          where it was shown in \cite[Page~392]{PruessSimonettMovingInterfaces} that for any $\vartheta \in [0,\pi)$ there is some $C_\vartheta>0$ such that
          \begin{equation*}
            |k(z)|\leq \frac{C_\vartheta}{1+|z|}\qquad \text{for all }z\in \Sigma_\vartheta.
          \end{equation*}
          Moreover, $\operatorname{Re} k(z)>0$ if $\operatorname{Re} z\geq 0$.
          
          We first show that \eqref{eq:smEstim} holds for $\lambda \in \Sigma_{\pi/2-\delta}, \tau\in \Sigma_\eta$ and any $\delta\in (0,\pi/4)$ and  sufficiently small $\eta>0$ (depending on $\delta$). To this end we use that for any $\eta>0$ there is some $C_\eta >0$ such that
          \begin{equation*}
          |z|\leq C_\eta \operatorname{Re} z \qquad \text{for all } z\in \Sigma_{\pi/2-\eta}.
        \end{equation*}
        Furthermore, observe that
        \begin{equation*}
          \frac{2\lambda}\tau + 2(\lambda+\tau^2)^{\frac12}+\tau \in \Sigma_{\pi/2-\delta/2}\quad \text{for all }\lambda\in \Sigma_{\pi/2-\delta}, \tau \in \Sigma_{\delta/4}
        \end{equation*}
        and therefore
        \begin{equation*}
        \frac{\sigma \tau^2}{2\lambda/\tau+ 2(\lambda+\tau^2)^{\frac12}+\tau}\in \Sigma_{\pi/2-\delta/3}\quad \text{for all }\lambda\in \Sigma_{\pi/2-\delta}, \tau \in \Sigma_{\eta}  
        \end{equation*}
        if $\eta>0$ is sufficiently small. Since $\lambda, m\tau^2\in \Sigma_{\pi/2-\delta/3}$ for all $\lambda\in \Sigma_{\pi/2-\delta}, \tau \in \Sigma_{\eta}$ as well, we conclude 
        \begin{align*}
          |s_m(\lambda,\tau)|&\geq \operatorname{Re} s_m(\lambda,\tau) = \operatorname{Re} \lambda + m\operatorname{Re}\tau^2+ \operatorname{Re} \frac{\sigma \tau^2}{2\lambda/\tau+ 2(\lambda+\tau^2)^{\frac12}+\tau}\\
          &\geq c_\eta \left( |\lambda|+ m|\tau|^2+ |\tau||k(z)|\right)\qquad \text{with }z=\frac{\lambda}{\tau^2}
        \end{align*}
        provided $\eta>0$ is sufficiently small. Now, if $|z|\leq 1$, there is some $c>0$ such that $|k(z)|\geq c$ if $\operatorname{Re} z\geq 0$ and $|z|\leq 1$ and we conclude
        \begin{equation*}
          |s_m(\lambda,\tau)|\geq c_\eta \left( |\lambda|+ m|\tau|^2+ |\tau|\right)\qquad \text{if }|z|\leq 1.
        \end{equation*}
        On the other hand, if $|z|\geq 1$, then $|\lambda|\geq |\tau|^2$ and therefore
        \begin{equation*}
          |s_m(\lambda,\tau)|\geq c_\eta \left( |\lambda|+ m|\tau|^2+ |\tau|\right)\quad \text{for all }\lambda\in \Sigma_{\pi/2-\delta},|\lambda|\geq \omega_0, \tau \in \Sigma_{\eta}
        \end{equation*}
        uniformly in $m\in[0,1]$ if $\eta>0$ is sufficiently small.

        Next we consider $\lambda \in \Sigma_{\pi/2+\delta}\setminus \Sigma_{\pi/2-\delta}$ for $\delta>0$ sufficiently small. To this end we note that, since $\operatorname{Re} k(z)>0$ if $\operatorname{Re} z \geq 0$ and $k(z)\to_{z\in \Sigma_{\vartheta},|z|\to \infty} 0$ for every $\vartheta\in [0,\pi)$, we have that
        \begin{equation*}
          K:= \overline{\left\{k(z): z\in \Sigma_{\pi/2+\eta}\right\}}\subseteq \{z\in \C: \operatorname{Re} z>0\}
        \end{equation*}
        if $\eta>0$ is sufficiently small. Moreover, $K$ is compact. Therefore, there is some $\delta>0$ such that
        \begin{equation*}
          K \subseteq \Sigma_{\pi/2-3\delta}.
        \end{equation*}
        Moreover, it is easy to prove that there is some $C_\delta>0$ such that
        \begin{equation*}
          |z|+|w|\leq C_\delta |z+w|\qquad \text{for all }z \in \Sigma_{\pi/2+\delta}\setminus \Sigma_{\pi/2-\delta}, w\in \Sigma_{\delta}
        \end{equation*}
        provided that $\delta \in (0,\pi/4)$. Hence
        \begin{align*}
          |s_m(\lambda,\tau)|&\geq c_{\delta} \left(|\lambda|+|m\tau^2+\sigma \tau k(z)|\right)\\
                             &\geq c_{\delta} \left(|\lambda|+m\operatorname{Re} (\tau^2)+\sigma \operatorname{Re}(\tau k(z))\right)\\
                             &\geq c_{\delta} \left(|\lambda|+m|\tau^2|+\sigma |\tau||k(z)|\right)
        \end{align*}
        and with the same arguments as before
        \begin{align*}
          |s_m(\lambda,\tau)|&\geq c_{\delta} \left(|\lambda|+m|\tau^2|+|\tau|\right)
        \end{align*}
        for all $\lambda \in \Sigma_{\pi/2+\delta}\setminus \Sigma_{\pi/2-\delta}$, $|\lambda|\geq \omega_0$, $\tau \in \Sigma_\eta$ uniformly in $m\in[0,1]$ if $\eta>0$ is sufficiently small. This finishes the proof.        
      \end{proof}
      Now we can proceed as in \cite[Section~8.3.3]{PruessSimonettMovingInterfaces}. Let $D_n^{\frac12}=(-\Delta_{x'})^{\frac12}$ be the Fourier multiplication operator with symbol $|\xi|$, $\xi\in \R^{d-1}$. Then $D_n^{\frac12}$ possesses an $\mathcal{R}$-bounded functional calculus in $W^{2-\frac1q}_q(\R^{d-1})$, for any $1<q<\infty$. Therefore
      \begin{equation*}
        (\lambda + D_n^{\frac12}+ m D_n) s_{m}^{-1}(\lambda, D_n^{\frac12})
      \end{equation*}
      is $\mathcal{R}$-bounded and possesses a bounded $\mathcal{H}^\infty$-calculus on $\Sigma_{\pi/2+\eta}\setminus B_{\omega_0}(0)$ for some $\eta>0$ and any $\omega_0>0$, which is also uniformly bounded in $m\in[0,1]$. Hence the operator-valued $\mathcal{H}^\infty$-calculus of $G=\partial_t +\omega$ on ${ }_0H^r_q(0,\infty; K_q(\R^{d-1}))$ for $\omega>0$, $s,r\in\R$ and $K=H,W$ yields that
      \begin{equation*}
        (\partial_t +\omega + D_n^{\frac12}+ m D_n) s_{m}^{-1}(\partial_t +\omega, D_n^{\frac12}) \in \mathcal{L} \left({ }_0H^r_q(0,\infty; K_q(\R^{d-1}))\right)
      \end{equation*}
      is uniformly bounded in $m\in[0,1]$ for any $s,r\in\R$. Hence we obtain for $h= s_{m}^{-1}(\lambda, D_n^{\frac12})b $ and $b\in \mathbb{E}$ that $h$ solves \eqref{eq:h} and satisfies 
      \begin{equation*}
        \| \omega h\|_{\mathbb{E}}+ \| \partial_t h\|_{\mathbb{E}}+         \| \nabla_{x'} h\|_{\mathbb{E}}+m \| \nabla_{x'}^2 h\|_{\mathbb{E}}\leq C_{r,s,p} \|b\|_{\mathbb{E}}
      \end{equation*}
      uniformly in $m\in [0,1]$, where $\mathbb{E}:=L^q(0,\infty; W^{2-\frac1q}_q(\R^{d-1}))$ or $\mathbb{E}:={ }_0H^k_q(0,\infty; L^q(\R^{d-1}))$, $k=0,1$.
      By real interpolation we obtain that the same is true for $b\in \mathbb{E}={}_0W^{1-\frac1{2q}}_q(0,\infty; L^q(\R^{d-1}))$.
      This shows existence of a solution as in Theorem~\ref{thm:FlatInt} satisfying \eqref{eq:EstimFlatInt} uniformly in $m\in[0,1]$, where the existence of $(\ve,p)$ and the corresponding estimates follow from \cite[Corollary~8.3.3]{PruessSimonettMovingInterfaces}. Uniqueness can be shown by a standard duality argument. 
      Hence Theorem~\ref{thm:FlatInt} is proved.

      \subsection{Proof of Theorem~\ref{thm:LinSystem}}\label{subsec:ProofLinSystem}

      First we assume that $\Gamma_t=\Gamma$ is independent of $t\in[0,T_0]$. In this case one proves the result by the same localization and perturbation argument and reduction to semi-homogeneous data as in \cite[Section~8.2]{PruessSimonettMovingInterfaces} using the result for a flat interface due to Theorem~\ref{thm:FlatInt}. The only difference is related to the new term ``$m\Delta_{\Gamma} h$'' in the evolution equation for $h$. These extra-terms can be controlled by $m\|h\|_{L^q(0,T_0;W^{4-\frac1q}_q)}$ (uniformly in $m\in[0,1]$), while all other terms can be controlled by $\|h\|_{L^q(0,T_0;W^{3-\frac1q}_q)}$ as in the case $m=0$. Here we note that in the proof one can reduce to $h_0=0$ by subtracting some $\tilde{h}\in L^q(0,T_0;W^{4-\frac1q}_q(\Gamma))\cap W^1_q(0,T_0; W^{2-\frac1q}_q(\Gamma))$ from $h$, which exists because of
      \begin{equation*}
        W^{4-\frac3q}_q(\Gamma)=(W^{2-\frac1q}_q(\Gamma), W^{4-\frac1q}_q(\Gamma))_{1-\frac1q,q}
      \end{equation*}
      since $1-\frac1q, 1-\frac3q\not\in \mathbb{Z}$ due to $q>d+2>3$ and by the trace method for real interpolation spaces.

      Now, if $\Gamma_t$, $t\in [0,T_0]$, are smoothly evolving interfaces as before, we fix an arbitrary $t_0\in [0,T_0]$ and define $J_{t_0}:= [\max(t_0-\delta,0), \min(T_0,t_0+\delta)]$ for $\delta>0$ and
      $\tilde{X}\colon \overline\Omega\times J_{t_0} \to \overline\Omega \times J_{t_0}$ by $\tilde{X}(x,t)= (X_t(X_{t_0}^{-1} (x)),t)$ for all $x\in \Omega$, $t\in J_{t_0}$. Then \eqref{eq:Linear1}-\eqref{eq:Linear6} with $J_{t_0}$ instead of $[0,T_0]$ is equivalent to a system of the form
      \begin{alignat}{2}
	\label{eq:Linear1'}
	\partial_t \tilde\ve^\pm-\Delta \tilde\ve^\pm  +\nabla \tilde{p}^\pm &= \tilde{\mathbf{f}} +K_{\mathbf{f}} &\qquad &\text{in }\Omega^\pm_{t_0}\times J_{t_0},\\\label{eq:Linear2'}
	\Div \ve^\pm &= g +K_g &\qquad &\text{in }\Omega^\pm_{t_0}\times J_{t_0},\\\label{eq:Linear3'}
	-\llbracket 2D\ve^\pm -p^\pm \mathbf{I}\rrbracket\no_{\Gamma_t} &= \sigma \Delta_{\Gamma_t} h \no_{\Gamma_t} + \mathbf{a}+K_{\mathbf{a}} && \text{on }\Gamma_{t_0}\times J_{t_0},\\ \label{eq:Linear4'}
	\llbracket\ve^\pm \rrbracket &=0 && \text{on }\Gamma_{t_0}\times J_{t_0},\\
	\label{eq:Linear5'}
	\partial_t h 
        &= \no_{\Gamma_t} \cdot\ve + m \Delta_{\Gamma_{t_0}} h+ b +K_b && \text{on }\Gamma_{t_0}\times J_{t_0},\\
	\label{eq:Linear6'}
	\ve_0^-|_{\partial\Omega}&= 0&&\text{on }\partial\Omega\times J_{t_0},
      \end{alignat}
      where $(K_{\mathbf{f}}, K_g, K_{\mathbf{a}}, K_b)$ depend linearly on $(\ve^\pm,p^\pm, h)$ and
      \begin{align*}
        &\|K_\mathbf{f}\|_{L^q(J_{t_0}; L^q(\Omega))}+\|K_g\|_{L^q(J_{t_0}; W^1_q(\Omega\setminus \Gamma_{t_0}))}+\|K_b\|_{L^q(J_{t_0}; W^{2-\frac1q}_q(\Gamma_{t_0}))}\\
        &\quad +\|K_\mathbf{a}\|_{W^{\frac12-\frac1{2q}}_q(J_{t_0};L^q(\Gamma_{t_0}))\cap L^q(J_{t_0}; W^{2-\frac1q}_q(\Gamma_{t_0}))}\\
        & \leq C(\delta)\left( \|\ve\|_{W^1_q(J_{t_0}; L^q(\Omega))}+ \|\ve\|_{L^q(J_{t_0};W^2_q(\Omega\setminus \Gamma_{t_0}))} + \|p\|_{L^q(J_{t_0}; W^1_q(\Omega\setminus \Gamma_{t_0}))}\right.\\
        &\quad + \| \llbracket p \rrbracket\|_{W^{\frac12-\frac1{2q}}_q(J_{t_0};L^q(\Gamma_{t_0}))}+ \| \llbracket p \rrbracket\|_{L^q(J_{t_0}; W^{2-\frac1q}_q(\Gamma_{t_0}))}+ \|h\|_{L^q(J_{t_0}; W^{3-\frac1q}_q(\Gamma_{t_0}))}\\
        &\quad\left. + \|h\|_{W^1_q(J_{t_0}; W^{1-\frac1q}_q(\Gamma_{t_0}))}+ m\|h\|_{L^q(J_{t_0}; W^{4-\frac1q}_q(\Gamma_{t_0}))}\right),
      \end{align*}
      where
      \begin{align*}
        \tilde{\ve}^\pm&:= \ve^\pm \circ \tilde{X}|_{\Omega^\pm_{t_0}},\quad         p^\pm:= p^\pm \circ \tilde{X}|_{\Omega^\pm_{t_0}}, \quad         \tilde{h}:= h\circ \tilde{X}|_{\Gamma_{t_0}},\\
        \tilde{\mathbf{f}}&:= \mathbf{f}\circ \tilde{X},\quad         \tilde{g}:= g \circ \tilde{X}, \quad         \tilde{\mathbf{a}}:= \mathbf{a}\circ \tilde{X}|_{\Gamma_{t_0}}  , \quad
                            \tilde{b}:= b\circ \tilde{X}|_{\Gamma_{t_0}}  
      \end{align*}
      and
      $C(\delta)\to 0$ as $\delta\to 0$ since $\tilde{X}\to I$ in $C^k(\Omega^\pm_{t_0}\times J_{t_0})$ as $\delta\to 0$ for any $k\in\mathbb{N}$. Hence by a standard Neumann series argument \eqref{eq:Linear1'}-\eqref{eq:Linear6'} together with $\tilde\ve|_{t=\tilde{t}_0}=\ve_0\circ \tilde{X}(\cdot,\tilde{t}_0), \tilde{h}|_{t=\tilde{t}_0}=h_0\circ \tilde{X}(\cdot,\tilde{t}_0)|_{\Gamma_{\tilde{t}_0}}$, $\tilde{t}_0:= \max(0,t_0-\delta)$, possesses a unique solution $(\tilde{\ve}^\pm,\tilde{p}^\pm, \tilde{h})$ for every $(\mathbf{f},g,\mathbf{a}, b, \ve_0,h_0)$ as in Theorem~\ref{thm:LinSystem} with $[0,T_0]$ replaced by $J_{t_0}$ and the initial time $0$ replaced by $\tilde{t}_0=\max(0,t_0-\delta)$ provided $\delta=\delta(t_0)>0$ is sufficiently small.   Transforming back, this yields a unique solution $(\ve^\pm,p^\pm,{h})$ of \eqref{eq:Linear1}-\eqref{eq:Linear6} together with $\ve|_{t=\max(0,t_0-\delta)}=\ve_0, {h}|_{t=\max(0,t_0-\delta)}=h_0$ for every $(\mathbf{f},g,\mathbf{a}, b, \ve_0,h_0)$ as in Theorem~\ref{thm:LinSystem} with $[0,T_0]$ replaced by $J_{t_0}$ and the initial time $0$ replaced by $\tilde{t}_0$. Moreover, the Neumann series argument also shows that \eqref{eq:UnifEstim} with the same replacements as before holds true uniformly in $m\in[0,1]$.

      Since $\{(t_0-\delta(t_0), t_0+\delta(t_0)): t_0\in [0,T_0]\}$ is an open covering of $[0,T_0]$, there are finitely many $0=t_0< t_1<\ldots< t_N<T_0$ and $\delta_j>0$ such that $[0,T_0] =\bigcup_{j=0}^N J_{t_j}$ and we can solve \eqref{eq:Linear1}-\eqref{eq:Linear6} on $J_{t_j}=[\max(0,t_j-\delta),\min (t_j+\delta,T_0)]$ together with $\tilde\ve|_{t=\max(0,t_j-\delta_j)}=\ve_0, \tilde{h}|_{t=\max(0,t_j-\delta_j)}=h_0$ uniquely as before. Hence we can solve \eqref{eq:Linear1}-\eqref{eq:Linear7} by solving successively \eqref{eq:Linear1}-\eqref{eq:Linear6} on $J_{t_j}$, $j=0,\ldots, N$ with initial values $(\ve_0,h_0)$, $(\ve|_{t=t_{j-1}}, h|_{t=t_{j-1}})$ for $j=2,\ldots, N$. Finally, since  \eqref{eq:UnifEstim} holds true uniformly in $m\in[0,1]$ on the intervals $J_{t_j}$, $j=0,\ldots, N$, one also obtains \eqref{eq:UnifEstim} on $[0,T_0]$ uniformly in $m\in[0,1]$.

      \subsection{Existence of Solutions for the Transformed System}

      The idea of the proof is to represent $\Gamma^m_t$  from the solution of \eqref{eq:ApproxLimit1}-\eqref{eq:ApproxLimit7} for every $t\in[0,T_0]$ as a graph over $\Gamma_t$, where $\Gamma_t$, $t\in[0,T_0]$, is from the smooth solution of \eqref{eq_TwoPhase1}-\eqref{eq_TwoPhase7} and to transform \eqref{eq:ApproxLimit1}-\eqref{eq:ApproxLimit7} to a corresponding perturbed system in $\Omega_t^\pm$, $t\in[0,T_0]$ with the aid of the Hanzawa-Transformation associated to $\Gamma_t$. To this end let us denote 
      \begin{equation*}
        \Gamma_t (3\delta) = \{ x\in \R^d:  |d_{\Gamma_t}(x)|<3\delta\}, \qquad t\in[0,T_0],
      \end{equation*}
      where
      \begin{equation*}
        d_{\Gamma_t}(x) =
        \begin{cases}
          \operatorname{dist}(x,\Gamma_t)&\text{if }x\in \Omega^+_t,\\
          -\operatorname{dist}(x,\Gamma_t)&\text{else}
        \end{cases}
      \end{equation*}
      is the signed distance to $\Gamma_t$. Since $(\Gamma_t)_{t\in[0,T_0]}$ are smoothly evolving, compact, and $[0,T_0]$ is compact, there is some $\delta>0$ such that for every $x\in\Gamma_t(3\delta)$, $t\in[0,T_0]$ there is a unique closest point $P_{\Gamma_t}(x)\in \Gamma_t$ and
      \begin{equation*}
        \Gamma(3\delta):=\bigcup_{t\in[0,T_0]} \Gamma_t(3\delta)\times \{t\} \ni (x,t)\mapsto (d_{\Gamma_t}, P_{\Gamma_t})\in \R^{d+1}
      \end{equation*}
      is smooth. Moreover, we choose $\delta>0$ so small that $\Gamma(3\delta)\subseteq \Omega\times [0,T_0]$.

      Furthermore, for a given continuous ``height function'' $h\colon \Gamma\to \R$ let
\begin{equation*}
  \theta_h\colon \Gamma \to \mathbb{R}^n\colon x\mapsto x+h(x,t)\no_{\Gamma_t}(x).
\end{equation*}
Here $\Gamma$ is defined as in \eqref{eq:defnGamma}.
Then $\theta_h$ is injective provided that $\|h\|_{C^0(\Gamma)}<\delta$. Moreover, we define the  \emph{Hanzawa transformation} associated to $\Gamma$ as
\begin{equation}\label{eq:Hansawa}
  \Theta_h(x,t)= x+\chi(d_{\Gamma_t}(x)/\delta)h(P_{\Gamma_t}(x),t)\no_{\Gamma_t}(P_{\Gamma_t}(x)),
\end{equation}
where $\chi\in C^\infty(\R)$ such that $\chi(s)=1$ for $|s|\leq 2\delta$ and $\chi(s)=0$ for $|s|>\frac23$ as well as $|\chi'(s)|\leq 4$ for all $s\in\R$, and $\|h\|_{C^0(\Gamma)}<\delta$.
Then $\Theta_h(.,t)\colon\Omega\to\Omega$ is a smooth diffeomorphism for every $t\in[0,T_0]$, cf.~e.g.~\cite[Chapter~1, Section~3.2]{PruessSimonettMovingInterfaces}. 
Hence for any $h\colon \Gamma\to \R$ that is continuously differentiable with $\|h\|_{C^0(\Gamma)}<\delta$ we have that
$(\Gamma^h_t)_{t\in[0,T_0]}:=(\theta_h(\Gamma_t,t))_{t\in[0,T_0]}$ is an oriented, compact evolving $C^1$-manifold, such that $\Gamma_t^h$ is a $C^2$-manifold for every $t\in [0,T_0]$ if $h(\cdot,t)\colon \Gamma_t\to\R$ is twice continuously differentiable.

In the following we look for a solution of \eqref{eq:ApproxLimit1}-\eqref{eq:ApproxLimit7} such that $\Gamma_t^m= \Gamma^h_t$ for all $t\in [0,T_0]$ a sufficiently regular $h\colon \Gamma\to \R$ with $\|h\|_{C^0(\Gamma)}<\delta$. Then $(\ve_m^\pm, p_m^\pm, (\Gamma_t^m)_{t\in[0,T_0]})$ solves \eqref{eq:ApproxLimit1}-\eqref{eq:ApproxLimit7} if and only if
 \begin{alignat*}{2}
   \ve^\pm(x,t)&:=\ve_m^\pm(\Theta_{h}(x,t),t),\quad p^\pm(x,t)=p^\pm_m(\Theta_h(x,t),t)&\quad&\text{for }x\in \Omega^\pm_t,t\in[0,T_0],\\
   h(x,t) &:= h_m(P_{\Gamma_t}(\theta_h(x,t)),t) &\quad&\text{for }x\in \Gamma_t,t\in[0,T_0]
 \end{alignat*}
 solves the transformed system
\begin{alignat}{2}\label{eq:TransNS1}
  \partial_t \ve^\pm -\Delta \ve^\pm+\nabla p^\pm &= a^\pm(h;D_x)(\ve^\pm,p^\pm)\\\nonumber
  &\quad + \partial_t \Theta_h\cdot \nabla_h \ve^\pm- \ve^\pm\cdot \nabla_h \ve^\pm&\quad& \text{in} \ \Omega^\pm,\\\label{eq:TransNS2}
  \Div \ve^\pm &= \operatorname{tr}((I-A(h))\nabla \ve^\pm)=:g(h)\ve^\pm &\quad& \text{in} \ \Omega^\pm,\\
  \llbracket \ve \rrbracket &= 0 &\quad& \text{on} \ \Gamma,\\
\llbracket 2D\ve^\pm -p^\pm \mathbf{I}\rrbracket\no_{\Gamma_t} &= t(h;D_x) (\ve,p) + \sigma K(h) \no_{\Gamma^h_t}\circ \theta^h_t&\quad& \text{on} \ \Gamma,\\
\ve^-|_{\partial\Omega} &=0&&\text{on}\ \partial\Omega\times (0,T_0),
\\\label{eq:TransNS5}
\partial_t^\bullet h -\no_{\Gamma_t} \cdot \ve &= m K(h) +\big(\no_{\Gamma^h_t}\circ \theta^h_t -\no_{\Gamma_t}\big)\cdot \ve && \text{on }\Gamma, \\\label{eq:TransNS6}
\ve|_{t=0} &= \ve_0  &\quad& \text{on} \ \Omega_0^\pm,
\end{alignat}
where
\begin{align*}
  a^\pm(h;D_x)(\ve^\pm,p^\pm)&= \Div_h(\nabla_h \ve^\pm)-\Delta \ve^\pm +(\nabla-\nabla_h)p^\pm,\\
  \nabla_h  &= A(h)\nabla,\quad \Div_h u = \operatorname{tr} (\nabla_h u),\quad A(h)=D_x\Theta_h^{-T},\\
t(h,D_x)(\ve,p)&= [\![(\no_{\Gamma_0}-\no_h)\cdot (2\mu^\pm D \ve-pI)+2\no_h\cdot \operatorname{sym}(\nabla \ve-\nabla_h \ve)]\!],\\
 \no_h&= \frac{A(h)\no_{\Gamma_0}}{|A(h)\no_{\Gamma_0}|},\quad  K(h)= H_{\Gamma^h_t}\circ \theta^h_t.
\end{align*}
For the following fixed-point argument we introduce the solution space
\begin{align*}
  \mathbb{E}_m(T_0)&:= \mathbb{E}_1(T_0)\times\mathbb{E}_2(T_0)\times\mathbb{E}_3(T_0)\times\mathbb{E}_{4,m}(T_0),\\
  \mathbb{E}(T_0)&:= \mathbb{E}_1(T_0)\times\mathbb{E}_2(T_0)\times\mathbb{E}_3(T_0)\times\mathbb{E}_4(T_0),\\
  \mathbb{E}_1(T_0)&:={}_0W^1_q(0,T_0; L^q(\Omega))^d\cap L^q(0,T_0;W^2_q(\Omega\setminus \Gamma_t))^d,\\
  \mathbb{E}_2(T_0)&:= L^q(0,T_0; \dot{W}^1_q(\Omega\setminus \Gamma_t)),\\
    \mathbb{E}_3(T_0)&:={}_0 W^{\frac12-\frac1{2q}}_q(0,T_0;L^q(\Gamma_t))\cap L^q(0,T_0; W^{1-\frac1q}_q(\Gamma_t)),\\
  \mathbb{E}_{4,m}(T_0) &:= W^{2-\frac1{2q}}_q(0,T_0;L^q(\Gamma_t))\cap {}_0W^1_q(0,T_0; W^{2-\frac1q}_q(\Gamma_t))\cap L^q(0,T_0; W^{4-\frac1q}_q(\Gamma_t)),\\
  \mathbb{E}_4(T_0) &:= W^{2-\frac1{2q}}_q(0,T_0;L^q(\Gamma_t))\cap {}_0W^1_q(0,T_0; W^{2-\frac1q}_q(\Gamma_t))\cap L^q(0,T_0; W^{3-\frac1q}_q(\Gamma_t)), 
\end{align*}
where $\mathbb{E}_{4,m}(T_0)$ is normed by
\begin{align*}
  \|h\|_{\mathbb{E}_{4,m}(T_0)} &:= \|h\|_{\mathbb{E}_{4}(T_0)} + m \|h\|_{L^q(0,T_0; W^{4-\frac1q}_q(\Gamma_t))}+m\|h\|_{W^{1-\frac1{2q}}_q(0,T_0; W^2_q(\Gamma_t))}+m\|h|_{t=0}\|_{W^{4-\frac3q}_q(\Gamma_0)}\\
  \|h\|_{\mathbb{E}_{4}(T_0)} &:= \|h\|_{W^{2-\frac1q}_q(0,T_0; L^q(\Gamma_t))} + \|h\|_{W^1_q(0,T_0; W^{2-\frac1q}_q(\Gamma_t))} + \|h\|_{L^q(0,T_0; W^{3-\frac1q}_q(\Gamma_t)}+\|h|_{t=0}\|_{W^{3-\frac3q}_q(\Gamma_0)}
\end{align*}
and $ \mathbb{E}_1(T_0)$, $ \mathbb{E}_2(T_0)$ are normed in a standard manner. We note that in comparison to \cite{KohnePruessWilkeTwoPhase} the conditions $\we|_{t=0}=0$, $h|_{t=0}=0$ are already included in the definition of $\mathbb{E}_m(T_0)$. Moreover, we define the space for the right-hand side as
\begin{align*}
  \mathbb{F}(T_0) &:= \mathbb{F}_1(T_0)\times \mathbb{F}_2(T_0)\times \mathbb{F}_3(T_0)\times \mathbb{F}_4(T_0)\\
  \mathbb{F}_1(T_0) &:=L^q(0,T_0; L^q(\Omega))^d,\quad \mathbb{F}_2(T_0):=L^q(0,T_0; W^1_q(\Omega\setminus \Gamma_t))\cap W^1_q(0,T;\dot{W}^{-1}_q(\Omega)),\\
  \mathbb{F}_3(T_0)&:= W^{\frac12-\frac1{2q}}_q(0,T_0;L^q(\Gamma_t))^d\cap L^q(0,T; W^{1-\frac1q}_q(\Gamma_t))^d,\\
    \mathbb{F}_4(T_0)&:= W^{1-\frac1{2q}}_q(0,T_0;L^q(\Gamma_t))\cap L^q(0,T; W^{2-\frac1q}_q(\Gamma_t))
\end{align*}
normed in the standard manner.

Then \eqref{eq:TransNS1}-\eqref{eq:TransNS5} can be written as
\begin{alignat}{1}\label{eq:Abstract1}
  L z + m \bar{L} z &= N(z)+ m\bar{N}(z),
\end{alignat}
for $z= (\ve,p, \llbracket p \rrbracket , h)\in \mathbb{E}(T_0)$ with $\|h\|_{C^0(\Gamma)}<\delta$, where
\begin{align*}
  Lz &= (L_1(\ve, p), L_2 \ve, L_3 (\ve,\pi,h), L_4(\ve, h)), \\
  N(z) &= (N_1(\ve, p), N_2 (\ve,h), N_3 (\ve,\pi, h), N_4(\ve, h))    
\end{align*}
for $z=(\ve,p,\pi,h)$ with $\|h\|_{C^0(\Gamma)}<\delta$ are defined as in \cite[Section~4]{KohnePruessWilkeTwoPhase}. More precisely,
\begin{alignat*}{1}
  L_1 (\ve, p) &:= \partial_t \ve -\Delta \ve +\nabla p,\quad\\
  L_2 \ve &:= \operatorname{div}\ve,\\
  L_3 (\ve,p,h) &:= -\llbracket 2D\ve^\pm \rrbracket\no_{\Gamma_t}- \pi\no_{\Gamma_t} -(\Delta_{\Gamma_t} h)\no_{\Gamma_t}\\
  L_4 (\ve,h)&:= \partial_t^\bullet h - \no_{\Gamma_t} \cdot \ve + \ve_0\cdot \nabla_{\Gamma_t} h
\end{alignat*}
and  
\begin{alignat*}{1}
  N_1 (\ve, p,h) &:= F(h,\ve)\nabla \ve + M_4(h):\nabla^2\ve +M_1(h)\nabla p,\quad N_2 (\ve,h) := M_1(h):\nabla \ve,\\
  N_3 (\ve,h) &:= G_\tau (h) \nabla \ve+(G_{\nu}(h)\nabla \ve+ G_\gamma(h))\no_{\Gamma_t}\\
  N_4 (\ve,h)&:= ([M_0(h)-I]\nabla_{\Gamma_t} h)\cdot \ve+(\ve-\ve_0)\cdot \nabla_{\Gamma_t} h,
\end{alignat*}
where $F, M_1,\ldots,  M_4$ are defined as in \cite[Section~2]{KohnePruessWilkeTwoPhase},
with the only difference that the time-independent reference surface $\Sigma$ is replaced by the smoothly evolving reference surface $\Gamma_t$, $t\in[0,T_0]$ and $\partial_t h$ is replaced by $\partial_t^\bullet h$. Moreover,
\begin{align*}
  \bar{L} z &= (0,0,0, \Delta_{\Gamma_t} h),\\
  \bar{N} z &= (0,0,0, K(h)-\Delta_{\Gamma_t}h)
\end{align*}
for $z=(\ve,p,\pi, h)$. Let us note that
\begin{equation*}
  L z_0 = N(z_0),
\end{equation*}
for $z_0 = (\ve_0, p_0,0)$, where $\ve_0|_{\Omega^\pm}= \ve_0^{\pm}$, $p_0|_{\Omega^\pm}= p_0^{\pm}$ is the solution of the limit system \eqref{eq_TwoPhase1}-\eqref{eq_TwoPhase7}.
Therefore \eqref{eq:Abstract1} is equivalent to
\begin{align}
  \mathcal{L}_mw&:= L w + m \bar{L} w- DN(z_0)w\label{eq:Lm} \\\nonumber
  &= N(w+z_0)-N(z_0)-DN(z_0)w+m \bar{N}(w+z_0)=: \mathcal{N}_m(w)  
\end{align}
for
$w = (\mathbf{u}, \pi, \llbracket \pi\rrbracket, h)$ with $\mathbf{u}= \ve-\ve_0$, $\pi = p-p_0$, where
\begin{align*}
  DN(z_0)w &=
             \begin{pmatrix}
               \ve_0 \cdot \nabla \we+\we\cdot \nabla \ve_0 + (DM_1(0)h+DM_2(0)h+ DM_3(0)h)\nabla \ve_0\\
               (DM_1(0) h) : \nabla \ve_0\\
               (DG_\tau (0)h)\nabla \ve_0+ ((DG_\nu(0)h)\nabla \ve_0)\no_{\Gamma_t}\\
               0
             \end{pmatrix}
\end{align*}
since $\bar{L}z_0=0$, $M_0(0)=I$, $M_j(0)=0$ for $j=1,\ldots, 4$, $G_\tau(0)=G_\nu(0)=G_\gamma(0)=0$, and $DG_\gamma (0)h= \sigma (DH_{\Gamma_t}(0)- DH_{\Gamma_t}(0))h=0$.
Hence $DN(z_0)$ is a linear operator of lower order with respect to $w$ compared to $ L  + m \bar{L}$. 

As in \cite[Proposition~3]{KohnePruessWilkeTwoPhase} we have:
\begin{proposition}\label{prop:Analyticity}
  Let $q>d+2$. Then $N\in C^\omega (\mathcal{U},\mathbb{F}(T))$ and $\bar{N}\in C^\omega (\mathcal{U}_m,\mathbb{F}(T))$, where 
  \begin{align*}
    \mathcal{U}:=\{(\mathbf{u}, \pi,r,h)\in \mathbb{E}(T): \|h\|_{\mathbb{E}_4(T)}<r_0\},\quad \mathcal{U}_m= \mathcal{U}\cap \mathbb{E}_m(T)    
  \end{align*}
for $r_0>0$ sufficiently small. Moreover, for every $r_1\in (0,r_0)$ there is some $C>0$ such that
\begin{equation}\label{eq:barNEstim}
  \|m\bar{N}(w_1)- m\bar{N}(w_2)\|_{\mathbb{F}(T)}\leq CR \|w_1-w_2\|_{\mathbb{E}_m(T)}
\end{equation}
for all $w_1,w_2\in \mathcal{U}_m$ with $\|w_j\|_{\mathbb{E}_m(T)}\leq R$, $j=1,2$, where $R\in (0,r_1]$.
\end{proposition}
\begin{proof}
  The statement for $N$ is proved in the same way as in \cite[Proposition~3]{KohnePruessWilkeTwoPhase}, where the only difference is that the reference surface is time-dependent. But because of the compactness of $[0,T_0]$ all relevant norms are uniformly controlled. For the statement on $\bar{N}$ we use the quasilinear structure of $K(h)$, i.e.,
  \begin{equation*}
    K(h)= \mathcal{C}_0 (h): \nabla^2_{\Gamma_t} h +\mathcal{C}_1(h),
  \end{equation*}
  where $\mathcal{C}_j(h)$ depend only on $h$ and $\nabla_{\Gamma_t}h$ and are analytic in $(h,\nabla_{\Gamma_t}h)$ (pointwise), cf.\ e.g.\ \cite[Section~2.2.5]{PruessSimonettMovingInterfaces}. Moreover, due to \cite[Lemma~6.1]{PruessSimonettTwoPhaseNSt} $\mathbb{F}_4(T_0)$ is a Banach algebra and $m\|\nabla^2_{\Gamma_t} h\|_{\mathbb{F}_4(T_0)}\leq C\|w\|_{\mathbb{E}_m(T_0)}$ uniformly in $m\in (0,1]$. From this $\bar{N}\in C^\omega (\mathcal{U}_m,\mathbb{F}(T_0))$ easily follows and one obtains \eqref{eq:barNEstim} in a straight forward manner. 
\end{proof}
The central technical results are:
\begin{proposition}\label{prop:contraction}
  Let $r_0$ be as in Proposition~\ref{prop:Analyticity}.
  There are some $C_N>0$, $R_0\in (0,r_0)$ such that for every $R\in (0,R_0]$, $m\in (0,1]$ we have
  \begin{align}\label{eq:contraction}
    \|\mathcal{N}_m(w_1)-\mathcal{N}_m(w_2)\|_{\mathbb{F}(T_0)}\leq C_NR\|w_1-w_2\|_{\mathbb{E}_m(T_0)}
  \end{align}
  for all $w_1,w_2\in \mathbb{E}_m(T_0)$ with $\|w_j\|_{\mathbb{E}_m(T_0)}\leq R$ for $j=1,2$.
\end{proposition}
\begin{proof}
  Let $R_0\in (0,1]$ be at least so small that $R_0< r_0$. Then by the definition of $\mathcal{N}_m$
  \begin{align*}
    \mathcal{N}_m(w_1)-\mathcal{N}_m(w_2) = &N(w_1+z_0)-N(w_2+z_0)-DN(z_0)(w_1-w_2)\\
    &+m\left(\bar{N}(w_1+z_0)-\bar{N}(w_2+z_0)\right)
  \end{align*}
  for all $w_j\in \mathcal{U}$, $j=1,2$. Now using the power series expansion for $N(w_j+z_0)$ and $\bar{N}(w_j+z_0)$, we obtain for $R_0$ sufficiently small
  \begin{align*}
    \|\mathcal{N}_m(w_1)-\mathcal{N}_m(w_2)\|_{\mathbb{F}(T_0)}&\leq C\|w_1-w_2\|_{\mathbb{E}(T_0)}^2 + CR\|w_1-w_2\|_{\mathbb{E}_m(T_0)}\\
&\leq CR \|w_1-w_3\|_{\mathbb{E}_m(T_0)}
  \end{align*}
  for all $w_j\in \mathcal{U}_m$ with $\|w_j\|_{\mathbb{E}_m(T_0)}\leq R$, $j=1,2$.
\end{proof}
\begin{proposition}
Let $\mathcal{L}_m$ be defined as in \eqref{eq:Lm}. Then there are some $m_1\in (0,1]$, $C_L>0$ such that $\mathcal{L}_m\colon \mathbb{E}_m(T_0)\to \mathbb{F}(T_0)$ is invertible and
  \begin{equation*}
    \|\mathcal{L}_m^{-1}\|_{\mathcal{L}(\mathbb{F}(T_0),\mathbb{E}_m(T_0))}\leq C_L \qquad \text{for all }m\in (0,m_1].
  \end{equation*}
\end{proposition}
\begin{proof}
  First of all $L+m\bar{L}\colon\mathbb{E}_m(T)\to  \mathbb{F}(T)$ is invertible for all $m\in (0,m_1]$, $m_1\in (0,1]$ sufficiently small, and all $T\in (0,T_0]$ because of Theorem~\ref{thm:LinSystem}. Moreover, there is some $C_L'>0$ such that
  \begin{equation*}
    \|(L+m\bar{L})^{-1}\|_{\mathcal{L}(\mathbb{F}(T),\mathbb{E}_m(T))}\leq C_L' \qquad \text{for all }m\in (0,m_1]\text{ and }T\in (0,T_0].
  \end{equation*}
  Using
  $$
  \mathbb{E}_4(T)\hookrightarrow {}_0W^1_q(0,T;W^{2-\frac1q}_q(\Gamma_t))\cap W^1_r(0,T;W^1_q(\Gamma_t))\hookrightarrow C^{1-\frac1q}([0,T];W^{2-\frac1q}_q(\Gamma_t))
  $$ 
  uniformly in $T\in (0,T_0]$ for some $r>q$ and the smoothness of $\ve_0$ one can show
  \begin{align*}
    &\|\ve_0 \cdot \nabla \we+\we\cdot \nabla \ve_0 + (DM_1(0)h+DM_2(0)h+ DM_3(0)h)\nabla \ve_0\|_{L^q(\Omega\times (0,T))}\\
    &\leq CT^{\frac1q} \|h\|_{L^\infty(0,T;W^{2-\frac1q}_q(\Gamma_t))}\leq C'T^{\frac1q}\|h\|_{\mathbb{E}_4(T)}
  \end{align*}
  and
  \begin{align*}
    \|(DM_1(0) h) : \nabla \ve_0\|_{\mathbf{F}_2(T)}\leq C    \|(DM_1(0) h) : \nabla \ve_0\|_{W^1_q(\Omega)\times (0,T))}&\leq CT^{\frac1q-\frac1r}\|h\|_{\mathbb{E}_4(T)},\\
    \|(DG_\tau (0)h)\nabla \ve_0+ ((DG_\nu(0)h)\nabla \ve_0)\no_{\Gamma_t}\|_{\mathbb{F}_3(T)}&\leq CT^\alpha \|h\|_{\mathbb{E}_4(T)}
  \end{align*}
  uniformly in $T\in(0,T_0]$ and $h\in \mathbb{E}_4(T)$ for some $\alpha>0$ by straightforward estimates. Hence
  \begin{equation*}
     \|DN(z_0)(L+m\bar{L})^{-1}\|_{\mathcal{L}(\mathbb{F}(T))}\leq CT^{\alpha} \qquad \text{for all }m\in (0,m_1]\text{ and }T\in (0,T_0]
   \end{equation*}
   for some $\alpha>0$. Hence by a Neumann series argument $\mathcal{L}_m$ is invertible and 
   \begin{equation*}
    \|\mathcal{L}_m^{-1}\|_{\mathcal{L}(\mathbb{F}(T),\mathbb{E}_m(T))}\leq C_L \qquad \text{for all }m\in (0,m_1]
  \end{equation*}
  provided that $T\in (0,T_\ast]$ for some $T_\ast>0$ sufficiently small. Finally, the invertibility of $\mathcal{L}_m$ and the uniform estimate on the time interval $(0,T_0)$ can be shown by the same arguments as in the end of the proof of Theorem~\ref{thm:LinSystem} by dividing $(0,T_0)$ in finitely many intervals $(0,T_1)$, $(T_1,T_2), \ldots ,(T_{N-1},T_N)$ of length less than $T_\ast$ and solving the system iteratively on the intervals. 
\end{proof}

Now let
\begin{equation*}
  R_m:=2m C_L\|\bar{N}(z_0)\|_{\mathbb{F}(T)}
\end{equation*}
and  choose $m_0\in (0,m_1]$ so small that
\begin{equation*}
  C_L C_N (R_{m_0} +m_0)\leq\frac12\quad \text{and}\quad R_{m_0}\leq R_0.
\end{equation*}
Then we have for every $m\in (0,m_0]$
\begin{align*}
  \|\mathcal{L}_m^{-1}\left(\mathcal{N}_m(w_1)-\mathcal{N}_m(w_2)\right)\|_{\mathbb{E}_m(T)}&\leq C_LC_N(R_m+m)\|w_1-w_2\|_{\mathbb{E}_m(T)}\leq \frac12\|w_1-w_2\|_{\mathbb{E}_m(T)}
\end{align*}
for all $w_1,w_2\in \mathbb{E}_m(T)$ with $\|w_j\|_{\mathbb{E}_m(T)}\leq R_m$ for $j=1,2$ and
\begin{align*}
   \|\mathcal{L}_m^{-1}\mathcal{N}_m(w)\|_{\mathbb{E}_m(T)}&\leq 
                                                             \|\mathcal{L}_m^{-1}\left(\mathcal{N}_m(w)-\mathcal{N}_m(0)\right)\|_{\mathbb{E}_m(T)}+ C_L\|\mathcal{N}_m(0)\|_{\mathbb{E}_m(T)}\\
  &\leq \frac{R_m}2+ m C_L\|\bar{N}(z_0)\|_{\mathbb{F}(T)}\leq  R_m
\end{align*}
for all $w\in \mathbb{E}_m(T)$ with $\|w\|_{\mathbb{E}_m(T)}\leq R_m$.
Hence $\mathcal{L}_m^{-1}\mathcal{N}_m$ is a contraction on $\overline{B_{R_m}(0)}$ in $\mathbb{E}_m(T)$ for all $m\in (0,m_0]$ and we obtain for every $m\in (0,m_0]$ a unique $w_m\in\overline{B_{R_m}(0)} $ such that
  \begin{equation}\label{eq:FixedPoint}
     w_m = \mathcal{L}_m^{-1}(\mathcal{N}_m w_m).
   \end{equation}
   In summary we obtain
   \begin{theorem}\label{th_approx_m_dependence}
     There is some $m_0>0$ such that for every $m\in (0,m_0]$ the transformed system \eqref{eq:TransNS1}-\eqref{eq:TransNS6} possesses a solution $(\ve,p,\llbracket p\rrbracket,h)\in \mathbb{E}_m(T_0)$, which satisfies
     \begin{align*}
       \|(\ve-\ve_0,p-p_0,\llbracket p - p_0\rrbracket,h)\|_{\mathbb{E}_m(T_0)}\leq Cm
     \end{align*}
     for some $C>0$ independent of $m\in (0,m_0]$.
   \end{theorem}
   Transforming back with $\Theta_h^{-1}$ finally yields a solution $(\ve_m^\pm,p_m^\pm, (\Gamma^m_t)_{t\in[0,T_0]})$ of \eqref{eq:ApproxLimit1}-\eqref{eq:ApproxLimit7}.

   \begin{remark}\label{rem:UniformDeltaNeigborhood}
     Since $\mathbb{E}_{4,m}(T_0)\hookrightarrow C^0([0,T_0];C^2(\Gamma_t))$ uniformly in $m\in (0,1)$, one can show in a straight forward manner that there are $m_1\in (0,m_0]$ and $\delta>0$ such that the signed distance function $d_{\Gamma^m_t}$ and the orthogonal projection $P_{\Gamma^m_t}$ on $\Gamma_t$ for every $t\in [0,T_0]$ are well-defined and smooth in $\Gamma^m(2\delta)= \{(x,t)\in \Omega\times [0,T_0]: |\operatorname{dist}(x,\Gamma^m_t)|<2\delta\}$ for every $m\in (0,m_1]$. Moreover,
     \begin{equation*}
       \Gamma^m(2\delta)\subseteq \Gamma(3\delta) \quad \text{for all }m\in (0,m_1] 
     \end{equation*}
     and $\delta>0$ can be chosen (as before) such that the signed distance function $d_{\Gamma_t}$ and the orthogonal projection $P_{\Gamma_t}$ on $\Gamma_t$ for every $t\in [0,T_0]$ are smooth in $\Gamma(3\delta)$.
   \end{remark}

  \subsection{Uniform Regularity}
  The goal of this section is to prove:
  \begin{theorem}\label{th:approx_lim_uniform}
   Let $q>d+5$.  There is some $m_0>0$ such that for every $m\in (0,m_0]$  \eqref{eq:ApproxLimit1}-\eqref{eq:ApproxLimit6} possesses a solution $(\ve_m^\pm, p_m^\pm, \Gamma^m)$ with $\Gamma^m\subseteq \Gamma(\delta)$, $d_{\Gamma^m_t}\in C^1_t C^1_x\cap C^0_tC^3_x$ in $\overline{\Gamma(\delta)}$ and $\ve_m\in L^{q}(0,T;W^1_{q}(\Omega)^d)\cap L^{q}(0,T;W^2_{q}(\Omega\setminus \Gamma_t^m))\cap W^1_{q}(0,T;L^{q}_\sigma(\Omega))$ and $\nabla_\tau \nabla \ve_m\in L^\infty(\Omega\times (0,T))$, $\nabla_\tau p_m\in L^{q}(0,T;W^1_{q}(\Omega\setminus \Gamma_t))$ with uniformly bounded norms with respect to $m\in (0,m_0]$. 
  \end{theorem}
  \begin{proof} We use the so-called parameter-trick to obtain that the tangential derivative of the solution belongs locally to the same function space as the solution itself (together with uniform bounds). To this end we follow the arguments in \cite[Section~9.4.2]{PruessSimonettMovingInterfaces} with modifications to our time dependent reference manifold $\Gamma_t$.
    Let $(t_0,x_0)\in\Gamma$ be arbitrary and $\tilde{X}\colon \overline{\Omega}\times J_{t_0} \to \overline{\Omega}$ with  $\tilde{X}(x,t)= X_t(X^{-1}_{t_0}(x))$ for all $x\in\Omega$, $t\in J_{t_0}:= [\max(t_0-\delta,0), \min(T_0,t_0+\delta)]$ as in Section~\ref{subsec:ProofLinSystem}. Moreover, let $\varphi\colon B_{3R}(0)\subseteq \R^{d-1}\to \R^d$ be a local parametrization of $\Gamma_{t_0}$ with $\varphi(0)= x_0$. Furthermore, for $t\in J_{t_0}$ let $\phi_t \colon B_{3R}(0) \times (-2\delta_0, 2\delta_0)\to \R^d$ be defined by
    \begin{equation*}
      \phi_t(s,\rho)= \tilde{X}(\varphi(s),t)+ \rho\mathbf{n}_{\Gamma_t}\quad \text{for all }(s,\rho)\in B_{3R}(0)\times (-2\delta_0, 2\delta_0).
    \end{equation*}
    Then $P_{\Gamma_t}(\phi_t(s,\rho))= \tilde{X}(\varphi(s),t)$. We define the truncated shift
    $\tau_\xi \colon J_{t_0}\times B_{3R}(0)\times (-2\delta_0, 2\delta_0)\to B_{3R}(0)\times (-2\delta_0, 2\delta_0)$
    \begin{equation*}
      \tau_\xi (t,s,\rho)= (s+\xi\eta(t)\chi_0(s)\zeta_0(\rho),\rho)\quad \text{for all }(t,s,\rho)\in J_{t_0}\times B_{3R}(0)\times (-2\delta_0, 2\delta_0),
    \end{equation*}
    where $\xi\in B_r(0)\subseteq \R^{d-1}$, $0<r\leq R$, $\chi_0\in C^\infty_0(\R^{d-1})$ with $0\leq \chi_0\leq 1$, $\supp \chi_0\subseteq B_{2R}(0)$ and $\chi_0(s)=1$ if $|s|\leq R$, $\zeta_0\in C_0^\infty(\R)$ with $\supp \zeta_0 \subseteq [-5\delta/2, 5\delta/2]$ and $\zeta_0(\rho)=1$ if $|\rho|\leq 2\delta_0$, $\eta\in C_0^\infty(\R)$ with $\eta(t)=0$ if $t\in [t_0-\delta/2,t_0+\delta/2]$ and $\supp \eta \subseteq (t_0-\delta, t_0+\delta)$. Finally, we define
    \begin{equation*}
      \boldsymbol\tau_\xi(t,x)= \phi_t(\tau_\xi(t,\phi_t^{-1}(x)))\quad\text{for all } (t,x)\in U:=\bigcup_{t\in J_{t_0}}\{t\}\times U_t,\ U_t:= \phi_t\left(B_{3R}(0)\times (-\tfrac{3\delta_0}2, \tfrac{3\delta_0}2)\right).
    \end{equation*}
    and
$
\boldsymbol\tau_\xi (t,x)= (t,x)
$
for all $t\in J_{t_0}$, $x\in \Omega\setminus U_t$ and $t\in [0,T]\setminus J_{t_0}$, $x\in\Omega$.
Note that $P_{\Gamma_t}(\boldsymbol\tau_\xi(t,x))= \tau_\xi(t,P_{\Gamma_t}(x))$ for all $x\in U_t$, $t\in J_{t_0}$ and therefore
\begin{align*}
      h\circ (\operatorname{id},P_{\Gamma_\cdot})\circ \boldsymbol\tau_\xi(t,x)= h(t, P_{\Gamma_t} (\boldsymbol\tau_{\xi}(t,x)))= (h(t,\cdot)\circ\tau_\xi)\circ (\operatorname{id},P_{\Gamma_\cdot})(t,x)
\end{align*}
for all $(t,x)\in U$ and thus
\begin{align*}
  \Theta_h(\boldsymbol\tau_\xi(t,x),t)= \Theta_{h\circ(\operatorname{id}, \tau_\xi)} (x,t)\quad \text{for all }(t,x)\in U
\end{align*}
since $d_{\Gamma_t}\circ \boldsymbol\tau_\xi(t,\cdot)= d_{\Gamma_t}$ on $U_t$ and $\chi\circ (d_{\Gamma_t}/\delta)=0$ on $\Gamma_t(2\delta)\setminus U_t$ for all $t\in J_{t_0}$.

Altogether we observe that for $r>0$ sufficiently small $\boldsymbol\tau_\xi(t,\cdot)\colon \overline{\Omega}\to \overline{\Omega}$ is a smooth diffeomorphism, which depends smoothly on $\xi \in B_r(0)$ and $t\in [0,T]$.  We denote by $T_\xi\colon \mathbb{E}_m(T)\to \mathbb{E}_m(T)$ the operator obtained by pointwise composition with $\boldsymbol\tau_\xi$. Let $G_m\colon \overline{B_R(0)}\subseteq \mathbb{E}_m(T)\to \overline{B_R(0)}$ with $G_m(w):=w-\mathcal{L}_m^{-1}(\mathcal{N}_m w)$ and $w_m$ be as in \eqref{eq:FixedPoint}. Then
\begin{equation*}
  0= T_\xi G_m (w_m)= T_\xi G_m(T_\xi^{-1}T_\xi w_m)\qquad \text{for all }\xi\in B_r(0).
\end{equation*}
Therefore we define $\mathcal{G}_m\colon B_r(0)\times B_R(0)\subseteq \R^{d-1}\times \mathbb{E}_m(T)\to \mathbb{E}_m$ by
\begin{equation*}
  \mathcal{G}_m(\xi, w):= T_\xi G_m(T_\xi^{-1} w).
\end{equation*}
Then as in \cite[Section~9.4.2]{PruessSimonettMovingInterfaces} one observes that $\mathcal{G}_m$ is continuously differentiable and $\partial_{\xi_j} \mathcal{G}_m (0,w_m)\in \mathbb{E}_m(T)$ is bounded with respect to $m\in (0,1]$. 
Furthermore $w_{m,\xi}:= T_\xi w_m$ solves
\begin{equation*}
\mathcal{G}_m(\xi, w_{m,\xi})=0\qquad \text{for all }\xi\in B_r(\xi).   
\end{equation*}
In particular, $\mathcal{G}_m(0, w_m)=0$ and
\begin{equation*}
  D_w\mathcal{G}_m(0, w_m) = DG_m(w_m)=I-\mathcal{L}_m^{-1}D\mathcal{N}_m(w_m)\colon \mathbb{E}_m(T)\to \mathbb{E}_m(T)
\end{equation*}
is invertible since $\|\mathcal{L}_m^{-1}D\mathcal{N}_m(w_m)\|_{\mathcal{L}(\mathbb{E}_m(T))}\leq \frac12$ due to Proposition~\ref{prop:contraction} for $m\in (0,m_1]$ with $m_1>0$ sufficiently small as before. Furthermore,
$\|DG_m(w_m)^{-1}\|_{\mathcal{L}(\mathbb{E}_m(T))}\leq 2 $
uniformly in $m\in (0,m_1]$. Now the Implicit Function Theorem yields that, if $r>0$ is sufficiently small, there are some $r'>0$ and continuously differentiable $\Phi_m\colon B_r(0)\to B_{r'}(w_m)\subseteq \mathbb{E}_m(T)$ such that $B_{r'}(w_m)\subseteq B_R(0)$ and
\begin{enumerate}
\item $\mathcal{G}_m(\xi, \Phi_m(\xi))=0$ for all $\xi\in B_r(0)$.
\item If $\mathcal{G}_m(\xi, u)=0$ for some $u\in B_R(0)\subseteq \mathbb{E}_m(T)$ and $\xi\in B_r(0)$, then $u= \Phi_m(\xi)$.
\end{enumerate}
Hence $\Phi_m(\xi)= w_{m,\xi}$ for all $\xi\in B_r(0)$.
To obtain uniform boundedness of $\partial_{\xi_j} w_{m,\xi}|_{\xi=0}$ we use that 
\begin{equation*}
  \partial_{\xi_j} w_{m,\xi}|_{\xi=0} = - D_w \mathcal{G}_m(0,w_m)^{-1} \partial_{\xi_j} \mathcal{G}_m(0, w_m)= - D G_m(w_m)^{-1} \partial_{\xi_j} \mathcal{G}_m(0, w_m),
\end{equation*}
where $\partial_{\xi_j} \mathcal{G}_m(0, w_m)\in \mathbb{E}_m(T)$ and $D_w G_m(w_m)^{-1}$ are uniformly bounded in $m\in (0,m_1]$ by the same observations before. Since $\partial_{\xi_j} u\circ \tau_\xi|_{\xi=0} = \partial_{\boldsymbol{\tau}_j} u$ in a neighborhood of $x_0$, where $\boldsymbol{\tau}_j(x)=\partial_{p_j} \tilde{X}(\varphi(p),t)$, $j=1,\ldots,d-1$, (with $x=\phi_t(p,q)$) form a basis of $T_x\Gamma_t$, the statement of the theorem follows. 
\end{proof}

\section*{Acknowledgements}
J.~Fischer and M.~Moser have received funding from the European Research Council (ERC) under the European Union's Horizon 2020 research and innovation programme (grant agreement No 948819)\smash{\includegraphics[scale=0.03]{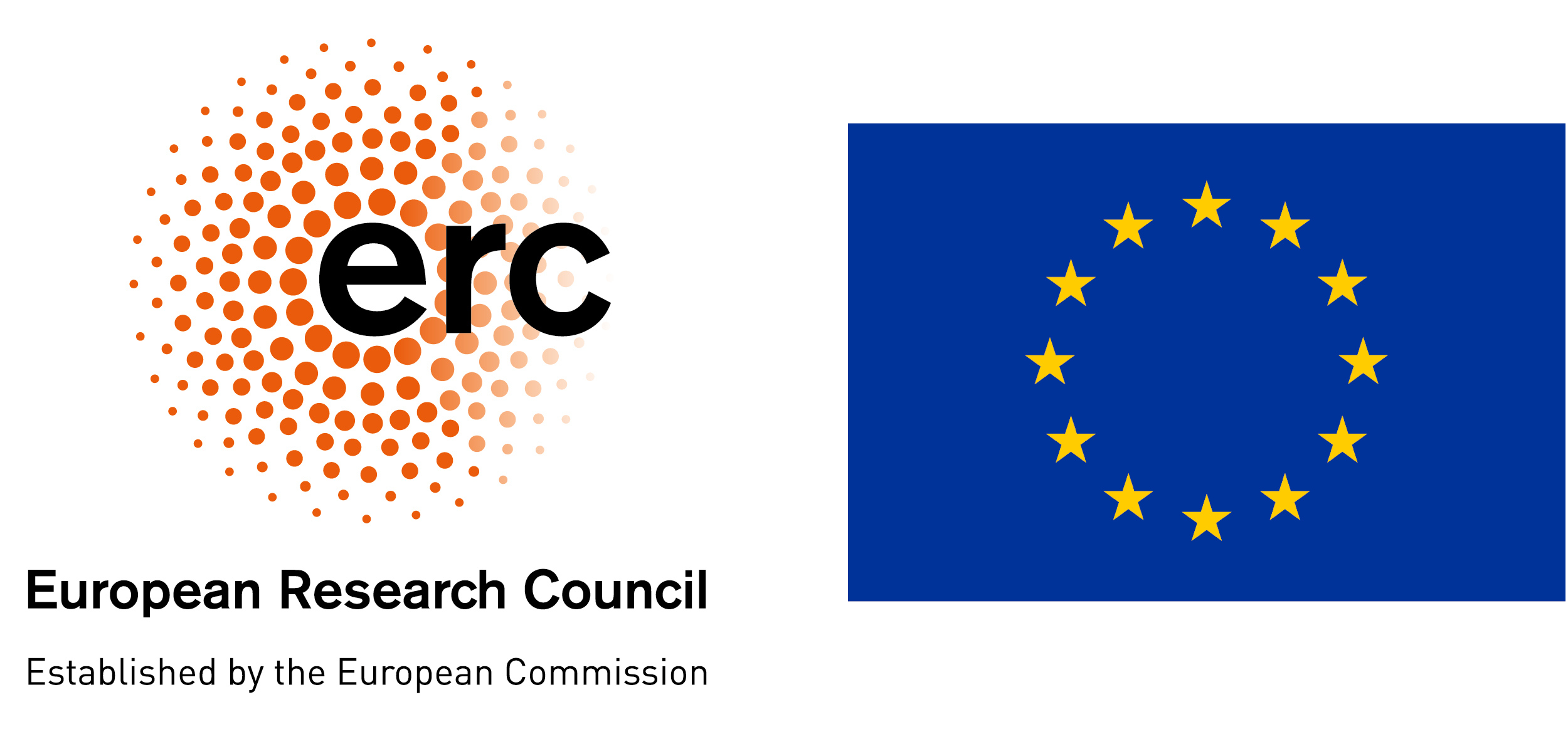}}.




\begin{thebibliography}{10}

\bibitem{GeneralTwoPhaseFlow}
H.~Abels.
\newblock On generalized solutions of two-phase flows for viscous
  incompressible fluids.
\newblock {\em Interfaces Free Bound.}, 9:31--65, 2007.

\bibitem{SILConvectiveAC}
H.~Abels.
\newblock ({N}on-)convergence of solutions of the convective {A}llen-{C}ahn
  equation.
\newblock {\em Partial Differ. Equ. Appl.}, 3(1):Paper No. 1, 11, 2022.

\bibitem{AbelsFei}
H.~Abels and M.~Fei.
\newblock Sharp interface limit for a {N}avier-{S}tokes/{A}llen-{C}ahn system
  with different viscosities.
\newblock {\em SIAM J. Math. Anal.}, 55(4):4039--4088, 2023.

\bibitem{AbelsFeiMoser}
H.~Abels, M.~Fei, and M.~Moser.
\newblock Sharp interface limit for a {N}avier-{S}tokes/{A}llen-{C}ahn system
  in the case of a vanishing mobility.
\newblock {{\em Calc. Var.}, 63(94), 2024.}

\bibitem{AGG}
H.~Abels, H.~Garcke, and G.~Gr\"{u}n.
\newblock Thermodynamically consistent, frame indifferent diffuse interface
  models for incompressible two-phase flows with different densities.
\newblock {\em Math. Models Methods Appl. Sci.}, 22(3):1150013, 40, 2012.

\bibitem{AbelsLengeler}
H.~Abels and D.~Lengeler.
\newblock On sharp interface limits for diffuse interface models for two-phase
  flows.
\newblock {\em Interfaces Free Bound.}, 16(3):395--418, 2014.

\bibitem{AbelsLiuNSAC}
H.~Abels and Y.~Liu.
\newblock Sharp interface limit for a {S}tokes/{A}llen-{C}ahn system.
\newblock {\em Arch. Ration. Mech. Anal.}, 229(1):417--502, 2018.
{\bibitem{AbelsMoserNSMF}
H.~Abels and M.~Moser.
\newblock Well-posedness of a {N}avier-{S}tokes/mean curvature flow system.
\newblock In {\em Mathematical analysis in fluid mechanics---selected recent
  results}, volume 710 of {\em Contemp. Math.}, pages 1--23. Amer. Math. Soc.,
  [Providence], RI, [2018] \copyright 2018.
\bibitem{AcerbiFusco}
E.~Acerbi and N.~Fusco.
\newblock Semicontinuity problems in the calculus of variations.
\newblock {\em Arch. Rational Mech. Anal.}, 86(2):125--145, 1984.}

\bibitem{Ambrosio2000a}
L.~Ambrosio, N.~Fusco, and D.~Pallara.
\newblock {\em Functions of Bounded Variation and Free Discontinuity Problems
  (Oxford Mathematical Monographs)}.
\newblock Oxford University Press, 2000.

\bibitem{BreitDieningGmeinederLipschitzBV}
D.~Breit, L.~Diening, and F.~Gmeineder.
\newblock The {L}ipschitz truncation of functions of bounded variation.
\newblock {\em Indiana Univ. Math. J.}, 70(6):2237--2260, 2021.

\bibitem{Chen}
X.~Chen.
\newblock Generation and propagation of interfaces for reaction-diffusion
  equations.
\newblock {\em J. Differential Equations}, 96(1):116--141, 1992.

\bibitem{DeMottoniSchatzman}
P.~de~Mottoni and M.~Schatzman.
\newblock Geometrical evolution of developed interfaces.
\newblock {\em Trans. Amer. Math. Soc.}, 347(5):1533--1589, 1995.

\bibitem{DenisovaTwoPhase}
I.~V. Denisova and V.~A. Solonnikov.
\newblock Solvability in {H}\"older spaces of a model initial-boundary value
  problem generated by a problem on the motion of two fluids.
\newblock {\em Zap. Nauchn. Sem. Leningrad. Otdel. Mat. Inst. Steklov. (LOMI)},
  188(Kraev. Zadachi Mat. Fiz. i Smezh. Voprosy Teor. Funktsii. 22):5--44, 186,
  1991.
{
\bibitem{EvansGariepy}
L.~C. Evans and R.~F. Gariepy.
\newblock {\em Measure theory and fine properties of functions}.
\newblock Textbooks in Mathematics. CRC Press, Boca Raton, FL, revised edition,
  2015.}

\bibitem{FischerHensel}
J.~Fischer and S.~Hensel.
\newblock Weak-strong uniqueness for the {N}avier-{S}tokes equation for two
  fluids with surface tension.
\newblock {\em Arch. Ration. Mech. Anal.}, 236:967--1087, 2020.

\bibitem{FischerLauxSimon}
J.~Fischer, T.~Laux, and T.~M. Simon.
\newblock Convergence rates of the {A}llen-{C}ahn equation to mean curvature
  flow: a short proof based on relative entropies.
\newblock {\em SIAM J. Math. Anal.}, 52(6):6222--6233, 2020.

\bibitem{FischerMarveggio}
J.~Fischer and A.~Marveggio.
\newblock Quantitative convergence of the vectorial {A}llen-{C}ahn equation
  towards multiphase mean curvature flow.
\newblock {\em Ann. Inst. H. Poincar{\'{e}} C Anal. Non Lin{\'{e}}aire}, 2024.
\newblock published online first.

\bibitem{GalGrasselliDCDS}
C.~G. Gal and M.~Grasselli.
\newblock Longtime behavior for a model of homogeneous incompressible two-phase
  flows.
\newblock {\em Discrete Contin. Dyn. Syst.}, 28(1):1--39, 2010.

\bibitem{GiorginiGrasselliWuJFA}
A.~Giorgini, M.~Grasselli, and H.~Wu.
\newblock On the mass-conserving {A}llen-{C}ahn approximation for
  incompressible binary fluids.
\newblock {\em J. Funct. Anal.}, 283(9):Paper No. 109631, 86, 2022.

\bibitem{GurtinTwoPhase}
M.~E. Gurtin, D.~Polignone, and J.~Vi{\~n}als.
\newblock Two-phase binary fluids and immiscible fluids described by an order
  parameter.
\newblock {\em Math. Models Methods Appl. Sci.}, 6(6):815--831, 1996.

\bibitem{HenselLiu}
S.~Hensel and Y.~Liu.
\newblock The sharp interface limit of a {N}avier--{S}tokes/{A}llen--{C}ahn
  system with constant mobility: Convergence rates by a relative energy
  approach.
\newblock {\em SIAM J. Math. Anal.}, 55(5):4751--4787, 2023.

\bibitem{HenselMoser}
S.~Hensel and M.~Moser.
\newblock Convergence rates for the {A}llen-{C}ahn equation with boundary
  contact energy: the non-perturbative regime.
\newblock {\em Calc. Var.}, 61(201):61 pp., 2022.

\bibitem{HohenbergHalperin}
P.~Hohenberg and B.~Halperin.
\newblock Theory of dynamic critical phenomena.
\newblock {\em Rev. Mod. Phys.}, 49:435--479, 1977.

\bibitem{VariableDensityNSAC}
J.~Jiang, Y.~Li, and C.~Liu.
\newblock Two-phase incompressible flows with variable density: an energetic
  variational approach.
\newblock {\em Discrete Contin. Dyn. Syst.}, 37(6):3243--3284, 2017.

\bibitem{JiangEtAlArXiv}
S.~Jiang, X.~Su, and F.~Xie.
\newblock Sharp interface limit for inhomogeneous incompressible
  {N}avier-{S}tokes/{A}llen-{C}ahn system in a bounded domain via a relative
  energy method.
\newblock {\em Preprint, arXiv:2305.09989}, 2023.

\bibitem{KoehnePruessWilkeTwoPhase}
M.~K{\"o}hne, J.~Pr{\"u}ss, and M.~Wilke.
\newblock Qualitative behaviour of solutions for the two-phase
  {N}avier-{S}tokes equations with surface tension.
\newblock {\em Math. Ann.}, 356(2):737--792, 2013.

\bibitem{KohnePruessWilkeTwoPhase}
M.~K\"{o}hne, J.~Pr\"{u}ss, and M.~Wilke.
\newblock Qualitative behaviour of solutions for the two-phase
  {N}avier-{S}tokes equations with surface tension.
\newblock {\em Math. Ann.}, 356(2):737--792, 2013.

\bibitem{LauxLiu}
T.~Laux and Y.~Liu.
\newblock Nematic--isotropic phase transition in liquid crystals: A variational
  derivation of effective geometric motions.
\newblock {\em Arch. Ration. Mech. Anal.}, 241(3):1785--1814, 2021.

\bibitem{LauxStinsonUllrich}
T.~Laux, K.~Stinson, and C.~Ullrich.
\newblock Diffuse-interface approximation and weak-strong uniqueness of
  anisotropic mean curvature flow.
\newblock {\em Preprint}, 2022.
\newblock arXiv:2212.11939.

\bibitem{LiuShenModelH}
C.~Liu and J.~Shen.
\newblock A phase field model for the mixture of two incompressible fluids and
  its approximation by a {F}ourier-spectral method.
\newblock {\em Phys. D}, 179(3-4):211--228, 2003.

\bibitem{Liu2}
Y.~Liu.
\newblock Phase transition of anisotropic ginzburg--landau equation.
\newblock {\em Preprint}, 2021.
\newblock 2111.15061.

\bibitem{Liu1}
Y.~Liu.
\newblock Phase transition of parabolic ginzburg--landau equation with
  potentials of high-dimensional wells.
\newblock {\em Preprint}, 2022.
\newblock arXiv:2207.12912.

\bibitem{PruessSimonettTwoPhaseNSt}
J.~Pr{\"u}ss and G.~Simonett.
\newblock On the two-phase {N}avier-{S}tokes equations with surface tension.
\newblock {\em Interfaces Free Bound.}, 12(3):311--345, 2010.

\bibitem{AnalyticityTwoPhaseFlow}
J.~Pr\"{u}ss and G.~Simonett.
\newblock Analytic solutions for the two-phase {N}avier-{S}tokes equations with
  surface tension and gravity.
\newblock In {\em Parabolic problems}, volume~80 of {\em Progr. Nonlinear
  Differential Equations Appl.}, pages 507--540. Birkh\"{a}user/Springer Basel
  AG, Basel, 2011.

\bibitem{PruessSimonettMovingInterfaces}
J.~Pr\"{u}ss and G.~Simonett.
\newblock {\em Moving interfaces and quasilinear parabolic evolution
  equations}, volume 105 of {\em Monographs in Mathematics}.
\newblock Birkh\"{a}user/Springer, [Cham], 2016.

\bibitem{SteinHarmonicAnalysis}
E.~M. Stein.
\newblock {\em Harmonic analysis: real-variable methods, orthogonality, and
  oscillatory integrals}, volume~43 of {\em Princeton Mathematical Series}.
\newblock Princeton University Press, Princeton, NJ, 1993.
\newblock With the assistance of Timothy S. Murphy, Monographs in Harmonic
  Analysis, III.

\end{thebibliography}

\bigskip

\noindent
\emph{ (H. Abels) Fakult\"at f\"ur Mathematik,
  Universit\"at Regensburg,
  D-93040 Regensburg}\\
\emph{ E-mail address: {\sf helmut.abels@mathematik.uni-regensburg.de} }\\[1ex]
\emph{
(J. Fischer) Institute of Science and Technology Austria, Am Campus 1, AT-3400 Klosterneuburg}\\
\emph{ E-mail address: {\sf julian.fischer@ist.ac.at} }\\[1ex]
\emph{
(M. Moser) Institute of Science and Technology Austria, Am Campus 1, AT-3400 Klosterneuburg}\\
\emph{ E-mail address: {\sf maximilian.moser@ist.ac.at} }\\[1ex]

\end{document}